\documentclass[10pt,a4paper]{article}

\usepackage{amsmath, amsfonts, amssymb, amsbsy, amsthm}
\usepackage{latexsym} \usepackage{stmaryrd} 
\usepackage{graphicx,mathrsfs}

\numberwithin{equation}{section}
\usepackage{hyperref}
\hypersetup{colorlinks,%
    citecolor=black,%
    filecolor=black,%
    linkcolor=black,%
    urlcolor=black
}

\newcommand{\U}{H^1(\Omega ; \R^2)} \newcommand{\Z}{H^1 (\Omega ; [0,1])}
\renewcommand{\U}{\mathcal{U}} \renewcommand{\Z}{\mathcal{Z}}
\newcommand{\F}{\mathcal{F}}\newcommand{\E}{\mathcal{E}}\newcommand{\D}{\mathcal{D}}
\renewcommand{\P}{\mathcal{P}}

\begin{document}

\newtheorem{theorem}{Theorem}[section]
\newtheorem{proposition}[theorem]{Proposition}
\newtheorem{corollary}[theorem]{Corollary}
\newtheorem{lemma}[theorem]{Lemma}
\newtheorem{definition}[theorem]{Definition}
\newtheorem{remark}[theorem]{Remark}
\newtheorem{example}[theorem]{Example}

\renewcommand{\proof}{\noindent \textbf{Proof. }}
\renewcommand{\qed}{ \hfill {\vrule width 6pt height 6pt depth 0pt} \medskip }

\newcommand{\separe}{\medskip \centerline{\tt -------------------------------------------- } \medskip}
\renewcommand{\separe}{}
\newcommand{\note}[1]{\medskip \noindent \framebox{\begin{minipage}[c]{\textwidth} {\tt #1} \end{minipage}}\medskip}

\newcommand{\eps}{\varepsilon} 
\renewcommand{\det}{\mathrm{det}} \newcommand{\0}{ {\mbox{\tiny 0}} }
\newcommand{\stress}{\boldsymbol{\sigma}} \newcommand{\strain}{\boldsymbol{\epsilon}}
\newcommand{\I}{\boldsymbol{I}}
\newcommand{\argmin}{ \mathrm{argmin} \,}\newcommand{\argmax}{ \mathrm{argmax} \,}
\newcommand{\weakto}{\rightharpoonup}\newcommand{\weakstarto}{\stackrel{*}{\rightharpoonup}}
\newcommand{\R}{\mathbb{R}}\newcommand{\N}{\mathbb{N}}
\newcommand{\calH}{\mathcal{H}}
\newcommand{\Om}{\Omega}
\newcommand{\di}{\mathrm{d}}
\newcommand{\coloneq }{\hspace{1pt}\raisebox{0.74pt}{\scalebox{0.8}{:}}\hspace{-2.2pt}=}
\newcommand{\FF}{\boldsymbol{F}}\newcommand{\EE}{\strain}
\newcommand{\II}{\boldsymbol{I}}
\def\Xint#1{\mathchoice
 {\XXint\displaystyle\textstyle{#1}}%
 {\XXint\textstyle\scriptstyle{#1}}%
 {\XXint\scriptstyle\scriptscriptstyle{#1}}%
 {\XXint\scriptscriptstyle\scriptscriptstyle{#1}}%
 \!\int}
\def\XXint#1#2#3{{\setbox0=\hbox{$#1{#2#3}{\int}$}
 \vcenter{\hbox{$#2#3$}}\kern-.5\wd0}}
 \def\ddashint{\Xint=} 
 \def\dashint{\Xint-}  
\def\interior{\mathaccent'27}




\newcommand{\tr}{\text{tr} \,}

%




\thispagestyle{empty} 

\phantom{1}

\vspace{0.5cm}
{\LARGE \noindent 
{\bf Analysis of staggered evolutions for nonlinear energies \\[3mm] in phase field fracture} }

\vspace{36pt}

\begin{small}

{\bf S.~Almi}

{Fakult\"at f\"ur Mathematik - TUM}

{Boltzmannstr.~3 - 85748 Garching bei M\"unchen - Germany}

{almi@ma.tum.de}

\vspace{24pt}

{\bf M.~Negri}

{Department of Mathematics -  University of Pavia} 

{Via A.~Ferrata 1 - 27100 Pavia - Italy}

{matteo.negri@unipv.it} 


\vspace{36pt}
\noindent {\bf Abstract.} We consider a class of separately convex phase field energies employed in fracture mechanics, featuring non-interpenetration
and a general softening behavior. We analyze the time-discrete evolutions generated by a staggered minimization scheme, where the fracture irreversibility is modeled by a monotonicity constraint on the phase field variable. We characterize the time-continuous limits of the discrete solutions in terms of balanced viscosity evolutions, parametrized by their arc-length with respect to the $L^2$-norm (for the phase field) and the $H^1$-norm (for the displacement field). By a careful study of the energy balance we deduce that time-continuous evolutions may still exhibit an alternate behavior in discontinuity times.


\end{small}

\section{Introduction} \label{Intro}

In the last decades the use of phase field models in computational fracture mechanics has been constantly increasing and has found many interesting applications. In the original formulation of~\cite{MR1745759} for quasi-static evolution of brittle fracture in linearly elastic bodies, the propagation of a crack, here represented by a phase field function~$z$, is described in terms of equilibrium configuration (or critical points) of the \emph{Ambrosio-Tortorelli functional}
\begin{equation}\label{intro1}
\mathcal{G}_{\varepsilon}(u,z):= \tfrac{1}{2} \int_{\Om} (z^{2}+\eta_{\varepsilon}) \stress(u){\,:\,}\strain(u)\,\di x + G_{c} \int_{\Om} \varepsilon |\nabla{z}|^{2} + \tfrac{1}{4\varepsilon} (1-z)^{2}\,\di x \,, 
\end{equation}
where~$\Om$ is an open bounded subset of~$\R^{n}$ with Lipschitz boundary~$\partial\Om$, $u\in H^{1}(\Om;\R^{n})$ is the displacement field,~$\strain(u)$ denotes the symmetric part of the gradient of~$u$,~$\stress(u):=\mathbb{C}\strain(u)$ is the stress,~$\mathbb{C}$ being the usual elasticity tensor, $\varepsilon$ and~$\eta_{\varepsilon}$ are two small positive parameters, and~$G_{c}$  is the toughness, a positive constant related to the physical properties of the material under consideration (from now on we impose $G_{c}=1$). In~\eqref{intro1} the function~$z\in H^{1}(\Om)$ is supposed to take values in~$[0,1]$, where~$z(x)=1$ if the material is safe at~$x$, while $z(x)=0$ means that the elastic body~$\Om$ presents a crack at~$x$. Hence, the zero level set of~$z$ represents the fracture and~$z$ can be interpreted as a regularization of a crack set.

The advantage in using phase field models like~\eqref{intro1} lies in their ability to handle the complexity of moving cracks, making the numerical implementation of the fracture process feasible even in rather involved geometrical settings. Indeed, energies of the form~$\mathcal{G}_{\varepsilon}$, defined on Sobolev spaces, can be easily discretized in finite element spaces or by finite differences. Furthermore, equilibrium configurations for~$\mathcal{G}_{\varepsilon}$ can be efficiently computed by means of \emph{alternate minimization} algorithms (see, e.g.,~\cite{MR2341850, MR1745759, MR2669398}), where~$\mathcal{G}_{\varepsilon}$ is iteratively minimized first w.r.t.~$u$ and then w.r.t.~$z$. This implies, in view of the quadratic nature of the functional, that at each step of the algorithm only a linear system has to be solved.

Starting from the seminal paper~\cite{MR1075076}, the connection between~\eqref{intro1} and brittle fracture mechanics has been drawn from a theoretical point of view by studying the~$\Gamma$-convergence as~$\varepsilon\to 0$ of~$\mathcal{G}_{\varepsilon}$ in BV-like spaces. A first result has been obtained in~\cite{MR2074682} in an~$SBD^{2}$ setting, while the generalization to~$GSBD^{2}$~\cite{MR3082250} has been presented in~\cite{ChambolleCrismale_D, MR3247391}. In this context, the limit functional~$\mathcal{G}_{0}$ is defined as
\begin{equation}\label{intro2}
\mathcal{G}_{0}(u):= \tfrac{1}{2} \int_{\Om} \stress(u){\,:\,}\strain(u)\,\di x + \mathcal{H}^{n-1}(J_{u}) \qquad\text{for $u\in GSBD^{2}(\Om)$},
\end{equation} 
where~$J_{u}$ denotes the approximate discontinuity set of~$u$ and therefore represents, in a suitable sense, a crack set.

While the $\Gamma$-convergence analysis ensures the convergence of minimizers of~\eqref{intro1} to minimizers of~\eqref{intro2}, and hence provides a rigorous justification of the phase field model~\eqref{intro1} at a static level, not so much is known for the convergence of evolutions, in particular for those obtained by alternate minimization schemes. 
A first analysis of convergence of these algorithms has been carried out in the recent paper~\cite{KneesNegri_M3AS17}, together with a full description of the limit evolutions in the language of rate-independent processes (see, e.g.,~\cite{MR2182832, MielkeRoubicek} and reference therein). The techniques developed in~\cite{KneesNegri_M3AS17} have then been applied to a finite dimensional approximation of~\eqref{intro1} in~\cite{Almi2017}. 

Let us briefly discuss the result obtained in~\cite{KneesNegri_M3AS17}. In dimension $n=2$, let~$[0,T]$ be a time interval and consider, for instance, a time dependent boundary condition $u = g(t) $ on~$\partial \Omega$ and initial conditions~$u_0$ and~$z_0$, with $0 \le z_0 \le 1$. Proceeding by time discretization, for every $k\in\mathbb{N}\setminus\{0\}$ let $\tau_k \coloneq T/k$ be a time increment and denote $t^{k}_{i} \coloneq i \tau_k$, for $i=0,...,k$. A discrete in time evolution is constructed using the following procedure: at time~$t^{k}_{i}$ for $j=0$  we define the sequences~$u^{k}_{i,j}$ and~$z^{k}_{i,j}$ setting $u^{k}_{i,0}:=u^{k}_{i-1}$, $z^{k}_{i,0}:=z^{k}_{i-1}$, and
\begin{eqnarray}
 && u^{k}_{i,j+1}:=\argmin\,\{ \mathcal{G}_{\varepsilon}(u,z^{k}_{i,j}):\, u\in H^{1}(\Om;\R^{2}),\, u=g(t^{k}_{i}) \text{ on~$\partial\Om$} \}\,, \label{altmin}\\[1mm]
&& z^{k}_{i,j+1}:=\argmin\,\{\mathcal{G}_{\varepsilon}(u^{k}_{i,j+1},z):\,z\in H^{1}(\Om),\, z\leq z^{k}_{i,j}\}\,. \label{altmin2}
\end{eqnarray}
In the limit $j\to \infty$, the algorithm~\eqref{altmin}-\eqref{altmin2} computes a limit pair~$(u^{k}_{i},z^{k}_{i})\in H^{1}(\Om;\R^{2})\times H^{1}(\Om)$, which turns out to be an equilibrium configuration of~$\mathcal{G}_{\varepsilon}$. We notice here that in the minimization~\eqref{altmin2} a \emph{strong irreversibility} is imposed, which forces the phase field variable~$z$ to decrease at each iteration.
A complete convergence result for the scheme~\eqref{altmin}-\eqref{altmin2} with the \emph{weaker} constraint~$z\leq z^{k}_{i-1}$ is still out of reach in our quasi static setting. We mention that a first result in this direction has been obtained in the work~\cite{A-B-N17} in the context of gradient flows, i.e., adding to the minimum problem~\eqref{altmin2} an~$L^{2}$-penalization of the distance between~$z^{k}_{i,j}$ and~$z^{k}_{i-1}$. 
Clearly the two above constraints are equivalent if we consider the simpler scheme with only one iteration of~\eqref{altmin}-\eqref{altmin2}, that has been employed in many mathematical papers (see, for instance,~\cite{MR3249813, MR2106765, MR3021776, MR3332887, MR3893258, Thomas13}). 
We also point out that the restriction to a two dimensional setting is rather technical, and is due to Sobolev embeddings that hold only in~$\Om\subseteq \R^{2}$.

In order to study the limit as the time step~$\tau_{k}$ tends to~$0$, it is not technically convenient to investigate the limit of each configuration~$(u^{k}_{i,j}, z^{k}_{i,j})$. On the contrary, in~\cite{KneesNegri_M3AS17} the authors provide a global description of the evolution by introducing an arc-length reparametrization of time, that is, a reparametrization based on the distance between the steps of the scheme~\eqref{altmin}-\eqref{altmin2}. This reminds of the usual approach to viscous approximation (see, e.g.,~\cite{MR3021776, MR3332887, MR3893258} in the context of phase field). The crucial point in~\cite{KneesNegri_M3AS17} is the choice of the norms used to compute the arc-length of the algorithm: while in the viscous setting it is natural to employ the viscosity norm, in~\eqref{altmin}-\eqref{altmin2} it is not clear whether there are preferable norms. Nevertheless, by its \emph{quadratic structure}, the functional~$\mathcal{G}_{\varepsilon}$ induces two weighted $H^{1}$-norms for~$u$ and~$z$, respectively, that are therefore referred to as \emph{energy norms}. With respect to these particular norms, it turns out that the affine interpolation curves between two consecutive states of the algorithm~\eqref{altmin}-\eqref{altmin2} are actually gradient flows of~$\mathcal{G}_{\varepsilon}$, whose lengths can be uniformly bounded. Gluing together all the interpolations and reparametrizing time, we obtain a piecewise linear curve with bounded velocity connecting all the states of the minimizing scheme and satisfying a discrete energy balance. As $k\to\infty$, the limit of these interpolation curves is a \emph{parametrized Balanced Viscosity evolution} complying with an equilibrium condition and an energy-dissipation balance (we refer to~\cite{MR3531671, MielkeRoubicek} for more details on this kind of solutions).

Despite the sound mathematical result, when reading~\cite{KneesNegri_M3AS17} one immediately notices that the whole convergence analysis strongly depends on the specific structure of the functional~\eqref{intro1}. This remark becomes clear if we try to repeat the above strategy with a different phase field energy, such as
\begin{equation}\label{intro3}
\mathcal{I}_{\varepsilon}(u,z):= \tfrac{1}{2} \int_{\Om}  W(z, \strain(u))\,\di x + \int_{\Om} \varepsilon |\nabla{z}|^{2} + f_{\varepsilon}(z) \,\di x\,,
\end{equation}
where some nonlinearities~$W$ and~$f_{\varepsilon}$ have been introduced, which make the functional~$\mathcal{I}_{\varepsilon}$ not separately quadratic. In this new context, there is no clear notion of energy norms. Hence, when trying to define an arc-length reparametrization of time, we would be forced to choose \emph{a priori} some norms in order to estimate the distance between two steps of the alternate minimization algorithm. Moreover, being~$\mathcal{I}_{\varepsilon}$ strongly nonlinear, we can not anymore expect that the linear interpolation between two consecutive states of the minimization algorithm can represent a gradient flow of~$\mathcal{I}_{\varepsilon}$. Therefore, the convergence of a numerical scheme of the form~\eqref{altmin}-\eqref{altmin2} for~$\mathcal{I}_{\varepsilon}$ does not trivially follow from the results of~\cite{KneesNegri_M3AS17} and needs further analysis, which is indeed the goal of the present paper.

More precisely, we focus, always in a two dimensional setting, on the phase field model introduced in~\cite{A-M-M09, ComiPerego_IJSS01}. The basic idea of the model is that an elastic material behaves differently under tension or compression, and a crack can appear or evolve only under tension. This means that the presence of a phase field should not affect the ability of the elastic body to store energy under compression. Hence, differently from~\eqref{intro1}, the factor~$z^{2}+\eta_{\varepsilon}$ can not pre-multiply the whole stress~$\stress(u)$. On the contrary, a splitting of~$\strain(u)$ into its volumetric $\strain_{v}(u):= \tfrac{1}{2} \tr \strain(u)\I$ and deviatoric~$\strain_{d}(u):= \strain(u) - \strain_{v}(u)$ components has to be considered, where the symbol $\tr$ stands the trace of a matrix. In order to further distinguish between tension and compression, we introduce~$\strain^{\pm}_{v}(u):= \tfrac{1}{2} (\tr \strain(u))_{\pm}\I$,~$(\cdot)_{\pm}$  denoting the positive and negative part, respectively. With this notation at hand, the elastic energy density~$W$ in~\eqref{intro3} takes the form
\begin{equation}\label{intro5}
W(z,\strain):= h(z) \big(\mu |\strain_{d}|^{2}+\kappa|\strain_{v}^{+}|^{2} \big) +\kappa|\strain_{v}^{-}|^{2} \qquad\text{for every~$z\in\R$ and every~$\strain\in\mathbb{M}^{2}_{s}$} ,
\end{equation}
where~$h$ is a suitable \emph{degradation function},~$\mu,\kappa>0$ are two positive constants related to the Lam\'e coefficients of the material, and~$\mathbb{M}^{2}_{s}$ denotes the space of symmetric matrices of order two with real coefficients. Introducing a time dependent boundary condition~$g\colon [0,T]\to H^{1}(\Om;\R^{2})$, the complete energy functional reads, for $t\in[0,T]$, $u\in H^{1}_{0}(\Om;\R^{2})$, and~$z\in H^{1}(\Om)$, as
\begin{equation}\label{intro6}
\F_{\varepsilon}(t, u,z):= \tfrac{1}{2} \int_{\Om}  W(z, \strain(u + g(t)))\,\di x + \int_{\Om} \varepsilon |\nabla{z}|^{2} +  f_{\varepsilon}(z) \,\di x \,.
\end{equation} 
We note that the explicit time dependence in~$\F_{\varepsilon}$ has been introduced in order to fix once and for all the ambient space~$H^{1}_{0}(\Om;\R^{2})$ for the displacement variable~$u$. This means that the real displacement will be $u+g(t)$, but the unknown of the problem is only~$u$. The advantage of this choice will be clear in the discussions of Section~\ref{s.gf}. More important, we again remark that in~\eqref{intro5} and~\eqref{intro6} we allow for nonlinearities~$h$ and~$f_{\varepsilon}$ different from the usual $z^{2}+\eta_{\varepsilon}$ and $\tfrac{1}{4\varepsilon}(1-z)^{2}$. This freedom is well justified by the existing literature on phase field fracture mechanics (see, for instance,~\cite{ MR3304294, PhysRevLett.87.045501, Wu_JMPS17}), where the modeling of different phenomena, such as brittle or cohesive fracture growth, results in the choice of different degradation profiles. Here, we will assume~$h \in C^{1,1}_{\mathrm{loc}}(\R)$, convex, positive, non-decreasing in~$[0,+\infty)$, and with minimum in~$0$, and $f_{\varepsilon} \in C^{1,1}_{\mathrm{loc}}(\R)$ strongly convex, non-negative, and with minimum in~$1$. We refer to Section~\ref{s.setting} for the precise setting.

The asymptotic behavior of~$\F_{\varepsilon}$ has been recently investigated in~\cite{MR3780140}. In dimension~$n=2$ and with the usual degradation functions $h(z) = z^{2}+\eta_{\varepsilon}$ and $f_{\varepsilon}(z)= \tfrac{1}{4\varepsilon}(1-z)^{2}$, it has been shown that~$\F_{\varepsilon}$ $\Gamma$-converges as~$\varepsilon\to 0$ to the functional
\begin{displaymath}
\F_{0}(u):= \left\{ \begin{array}{ll}
 \displaystyle \tfrac{1}{2}\int_{\Om} \big( 2\mu|\strain_{d}(u)|^{2} + \kappa |\strain_{v}(u)|^{2}\big) \,\di x + \mathcal{H}^{1}(J_{u}) & \text{if $u\in SBD(\Om)$, $[ u ] {\,\cdot\,}\nu_{u} \geq 0$ $\mathcal{H}^{1}$-a.e. in~$J_{u}$,}\\ 
 \displaystyle \vphantom{\int}+\infty &\text{otherwise},
 \end{array}\right.
\end{displaymath}
where~$ [u]$ stands for the approximate jump of~$u$ on~$J_{u}$ and~$\nu_{u}$ is the approximate unit normal to~$J_{u}$. The condition $[ u ] {\,\cdot\,}\nu_{u} \geq 0$ on~$J_{u}$ represents a \emph{linear non-interpenetration} constraint which, in the fracture mechanics language, forces the lips of the crack set~$J_{u}$ to not interpenetrate.

In this paper we are interested in the study of the convergence of alternate minimization algorithms for the evolution problem of the phase field model~\eqref{intro6}. To simplify the notation, we fix $\varepsilon:=\tfrac{1}{2}$ and denote with~$\F$ the functional~$\F_{\frac{1}{2}}$. Given $T>0$, $\tau_{k}:= T/k$, and $t^{k}_{i}:= i\tau_{k}$, we consider the following iterative procedure, similar to~\eqref{altmin}-\eqref{altmin2}: at time $t^{k}_{i}$, we set $u^{k}_{i,0}:= u^{k}_{i-1}$, $z^{k}_{i,0}:= z^{k}_{i-1}$ and, for $j\geq 1$, we define
\begin{eqnarray}
&& u^{k}_{i,j} := \argmin\,\{ \F (t^{k}_{i}, u, z^{k}_{i,j-1}): \, u\in  H^{1}_{0}(\Om;\R^{2}) \}\,, \label{altmin3}\\[1mm]
&& z^{k}_{i,j} := \argmin\,\{ \F (t^{k}_{i}, u^{k}_{i,j}, z): \, z\in H^{1}(\Om),\, z\leq z^{k}_{i,j-1} \}\,. \label{altmin4}
\end{eqnarray}
In the limit as $j\to\infty$ we detect a critical point $(u^{k}_{i}, z^{k}_{i})$ of~$\F$. In order to analyze the limit of the time-discrete evolution~$(u^{k}_{i}, z^{k}_{i})$ as the time step~$\tau_{k}\to 0$, we follow the general scheme of~\cite{KneesNegri_M3AS17}. First, we want to interpolate between all the steps of the scheme~\eqref{altmin3}-\eqref{altmin4} and reparametrize time w.r.t.~an arc-length parameter. As already mentioned, we have to face here the fact that the energy~$\F$ is highly nonlinear and not separately quadratic. This implies that there are no intrinsic norms stemming out from the functional, as it happens in~\cite{KneesNegri_M3AS17}. In our framework, instead, we a priori fix the $H^{1}$-norm for the displacement field~$u$ and the $L^{2}$-norm for the phase field~$z$. Our choices, made clear in Section~\ref{s.gf}, are guided by the possibility to construct suitable gradient flows connecting consecutive states of our alternate minimization algorithm. In particular, being~$\F$ differentiable w.r.t.~$u$, by classical results we get the existence of a gradient flow of~$\F$ in the $H^{1}$-norm starting from~$u^{k}_{i,j-1}$ and ending in~$u^{k}_{i,j}$. When constructing a gradient flow for~$z$ connecting~$z^{k}_{i,j-1}$ and~$z^{k}_{i,j}$, instead, we have to deal with the irreversibility condition~$z\leq z^{k}_{i,j-1}$ which forces us to work with the weaker $L^{2}$-norm (we refer to Theorem~\ref{t.3} for more details). As a byproduct of our construction, the total length of the scheme is uniformly bounded in~$k$. Hence, gluing all the gradient flows together and reparametrizing time we obtain a sequence of curves~$(t_{k}, u_{k}, z_{k})$ with bounded velocity interpolating between all the states of the minimization scheme and satisfying, once again, discrete in time equilibrium and energy balance.

In the limit as $k\to\infty$, we prove the convergence to a parametric BV evolution $(t,u,z)\colon [0,S]\to [0,T]\times H^{1}_{0}(\Om;\R^{2})\times H^{1}(\Om)$, which we characterize in terms of \emph{equilibrium} and \emph{energy-dissipation balance} as follows (see Theorem~\ref{t.1} for further details):
\begin{itemize}
\item[$(i)$] for every $s\in(0,S)$ such that $t'(s)>0$
\begin{displaymath}
| \partial_{u}\F| (t(s), u(s), z(s))=0\qquad\text{and}\qquad | \partial_{z}^{-}\F| (t(s), u(s), z(s))=0\,,
\end{displaymath}
\item[$(ii)$] for every $s\in[0,S]$
\begin{displaymath}
\begin{split}
\F(t(s),& \ u(s),z(s))  = \F(0,u_0,z_0)-\int_{0}^{s}|\partial_{u}\F|(t(\sigma),u(\sigma),z(\sigma))\, \| u' (s) \|_{H^1} \di\sigma\\
&-\int_{0}^{s} \!\! |\partial_{z}^{-}\F|(t(\sigma),u(\sigma),z(\sigma))\, \| z' (s) \|_{L^2} \di\sigma +\int_{0}^{s} \!\! \P( t(\sigma),u(\sigma), z(\sigma)) \,t'(\sigma)\,\di\sigma\,,
\end{split}
\end{displaymath}
\end{itemize}
where $| \partial_{u}\F|$ and~$| \partial_{z}^{-}\F|$ denotes the slopes of~$\F$ w.r.t.~$u$ and~$z$, respectively (see Definition~\ref{d.slope}) and~$\P$ is the power expended by the external forces (boundary datum~$g$ in our case), and is defined in~\eqref{e.power}.

Roughly speaking, the equilibrium condition~$(i)$ says that at continuity times, i.e., when $t'(s)>0$, the pair~$(u(s), z(s))$ is an equilibrium configuration for~$\F$, while the energy-dissipation balance~$(ii)$ gives us a complete description of the behavior of a solution at discontinuity times. As it was already noticed in~\cite{KneesNegri_M3AS17}, the characterization~$(i)$-$(ii)$ is very similar to the one obtained in~\cite{MR3021776, MR3332887, MR3893258} with a vanishing viscosity approach. The main advantage of the iterative minimization~\eqref{altmin3}-\eqref{altmin4} is that we do not have to add a fictitious viscosity term. Moreover, our constructive scheme is closer to the numerical applications, where alternate minimization schemes are usually adopted.

We conclude with a short description of main steps of the proof of~$(i)$ and~$(ii)$. The convergence of~$(t_{k}, u_{k}, z_{k})$ is obtained by a compactness argument. By the nonlinearity of~$\F$, we actually need a pointwise strong convergence of~$u_{k}$ in~$H^{1}(\Om;\R^{2})$, which is shown in Proposition~\ref{p.compactness} by studying convergence of gradient flows. The equilibrium~$(i)$ and the lower energy-dissipation inequality follow then from lower semicontinuity of the functional~$\F$ and of the slopes~$|\partial_{u} \F|$ and~$|\partial_{z}^{-} \F|$, discussed in Section~\ref{prop:slopes}. The technically hard part comes with the upper energy-dissipation inequality, where we pay the choice of the $L^{2}$-norm to estimate the arc-length of the algorithm~\eqref{altmin3}-\eqref{altmin4} w.r.t.~$z$. Comparing with~\cite{KneesNegri_M3AS17}, indeed, we can not employ a chain rule argument, since the evolution~$z$ is qualitatively the reparametrization of an $L^{2}$-gradient flow, instead of an $H^{1}$-gradient flow. For this reason, we need to exploit a Riemann sum argument (see, e.g.,~\cite{MR2186036, Negri_ACV}). In this respect, we have to face the lack of summability of the slope $|\partial_{z}^{-} \F| (t(\cdot), u(\cdot), z(\cdot))$, which does not follow from energy estimates, since we are only able to control $|\partial_{z}^{-} \F| (t(\cdot), u(\cdot), z(\cdot)) \| z'(\cdot) \|_{L^{2}}$. This problem is overcome by a careful analysis of the evolution of~$z$. The idea is to gain the summability of~$|\partial_{z}^{-} \F| (t(\cdot), u(\cdot), z(\cdot))$ outside the set $\{\| z' \|_{L^{2}}=0 \}$. This allows us to perform a further change of variable and employ a Riemann sum argument in the new variable. As a byproduct of our analysis, we also show that a limit evolution~$(t,u,z)$ may still exhibit an alternate behavior in discontinuity times. We refer to Section~\ref{s.inequality} and Appendix~\ref{AppB} for the full details.

\tableofcontents

\section{Setting and statement of the main result}\label{s.setting} 


\subsection{Elastic energy density with anisotropic softening}

Let us first introduce some notation. We denote by~$\mathbb{M}^{2}$ the space of squared matrices of order~$2$ (with real entries) and by~$\mathbb{M}^{2}_{s}$ the subspace of symmetric matrices. For every $\FF\in\mathbb{M}^{2}$, its volumetric and deviatoric part, respectively, are denoted by 
\begin{equation*} 
\FF_{v}\coloneq \tfrac{1}{2}(\tr \FF) \II\qquad \text{and} \qquad \FF_{d}\coloneq \FF-\FF_{v}\,,
\end{equation*}
where $\tr \FF$ stands for the trace of~$\FF$ and~$\II$ is the identity matrix. We notice that $\FF_{v}{\,:\,}\FF_{d}=0$, where the symbol~$:$ indicates the usual scalar product between matrices. As a consequence, we have that
\begin{displaymath}
|\FF|^{2}=|\FF_{v}|^{2}+|\FF_{d}|^{2}\qquad\text{for every $\FF\in\mathbb{M}^{2}$}\,,
\end{displaymath}
where $| \cdot |$ denotes Frobenius norm. Furthermore, we set 
\begin{displaymath}
\FF_{v}^{\pm}\coloneq \tfrac{1}{2}(\tr \FF)_{\pm}\II\,,
\end{displaymath}
where~$(\cdot)_{+}$ and~$(\cdot)_{-}$ denote positive and negative part, respectively. Clearly, $| \FF_v |^2 = | \FF^+_v |^2 + | \FF^-_v |^2$. 

With this notation, for a (strain) matrix $\EE\in \mathbb{M}^{2}_{s}$ we first rewrite the linear elastic energy density as
\begin{align}
\Psi(\EE)&= \tfrac{\lambda}{2} |\tr\EE|^{2}+\mu|\EE|^{2}=\lambda(|\EE_{v}^{+}|^{2}+|\EE_{v}^{-}|^{2})+\mu(|\EE_{v}^{+}|^{2}+|\EE_{v}^{-}|^{2}+|\EE_{d}|^{2}) \nonumber\\
&=\mu|\EE_{d}|^{2}+\kappa|\EE_{v}^{+}|^{2}+\kappa|\EE_{v}^{-}|^{2}=\Psi_{+}(\EE)+\Psi_{-}(\EE)\,,\label{e.density}
\end{align} 
where we have set $\kappa\coloneq\lambda+\mu$, $\Psi_{+}(\EE)\coloneq \mu|\EE_{d}|^{2}+\kappa|\EE_{v}^{+}|^{2}$, and $\Psi_{-}(\EE)\coloneq \kappa|\EE_{v}^{-}|^{2}$. We assume that $\mu >0$ and that $\kappa >0$.

\bigskip
Our phase field model, inspired by~\cite{A-M-M09, MR3780140}, does not allow for fracture under compression, i.e.~where $( \tr \EE)_-\neq 0$; this is obtained employing 
an \emph{elastic energy density} is of the form 
\begin{equation}\label{e.enelden}
W(z,\EE)\coloneq h(z) \Psi_{+}(\EE)+ \Psi_{-}(\EE)\qquad\text{for $z\in\R$ and  $\EE\in\mathbb{M}^{2}_{s}$}\,,
\end{equation}
where $z$ is the phase field variable and  $h$ is the {\it softening} or {\it degradation function}. We assume that $h$ is convex, of class $C^{1,1}_{\mathrm{loc}}(\R)$ and that
\begin{equation}
 \text{ $h(z)\geq h(0) >0$ for every $z\in\R$.}
\label{hp2} 
\end{equation}
Note that, under these assumptions, $h$ is non-decreasing in $[0,+\infty)$. 
We denote $\eta : = h (0) > 0$. Note that $W ( z, \cdot) $ is differentiable w.r.t.~$\strain$ and that 
\begin{equation} \label{stress}
\partial_{\strain} W(z,\strain ) = 2 h(z) \big(  \mu \strain_{d} + \kappa \strain_{v}^{+} \big) - 2\kappa \strain_{v}^{-} \,.
\end{equation} Further properties of the energy density $W$ are provided in Section~\ref{prop:energy}.

\subsection{Energy, slopes and power}

The reference configuration $\Omega$ is assumed to be a bounded, connected, open subset of~$\R^2$ with Lipschitz boundary~$\partial\Om$. We denote by $\partial_D \Omega$ a non-empty subset of $\partial \Omega$ made of finitely many,  relatively open, connected components. We consider a time interval $[0,T]$ and, for every $t \in [0,T]$, admissible displacements of the form $u + g(t)$ where $u$ belongs to $\U\coloneq \{u\in H^{1}(\Om;\R^{2}):\, u=0 \text{ on~$\partial_{D}\Om$}\}$ while the ``boundary datum'' $g$ belongs to $W^{1,q}([0,T];W^{1,p}(\Om;\R^{2}))$ with $q\in (1,+\infty)$ and $p\in(2,+\infty)$. The phase field $z$ belongs instead to $\Z\coloneq H^{1}(\Om)\cap L^{\infty}(\Om)$ (even though in the evolution it will take value in $[0,1]$). 


For $p\in[1,+\infty]$, we denote by~$\|\cdot\|_{W^{1,p}}$ and by~$\|\cdot\|_{L^p}$ the usual $W^{1,p}$ and~$L^{p}$-norms, respectively; we use also the notation~$\|\cdot\|_{H^{1}}$ for the $H^{1}$-norm.

Then, for every $(t,u,z) \in [0,T] \times \U \times \Z$ we define the \emph{elastic energy} as
\begin{equation}\label{e.E} 
\begin{split}
	 \E ( t , u , z ) &\coloneq\int_{\Om}W\big(z,\strain(u+g(t))\big)\,\di x  = \int_\Omega h(z) \Psi_{+}\big(\strain(u+g(t))\big)+ \Psi_{-}\big( \strain(u+g(t))\big)   \,\di x \\
	 &= \int_{\Om}h(z)\big(\mu |\strain_{d}(u+g(t))|^{2}+\kappa|\strain_{v}^{+}(u+g(t))|^{2}\big)\,\di x+\int_{\Om}\kappa|\strain_{v}^{-}(u+g(t))|^{2}\,\di x\,, 
\end{split}
\end{equation}
where $\strain(u+g(t))$ denotes the symmetric part of the gradient of the displacement $u+g(t)\in H^{1}(\Om;\R^{2})$. We introduce the \emph{dissipation pseudo-potential} for the phase field $z\in\Z$ as
\begin{equation}\label{e.D}
\D(z)\coloneq\tfrac{1}{2}\int_{\Om} |\nabla{z}|^{2}+f(z) \,\di x\,.
\end{equation}
We assume that $f : \mathbb{R} \to \mathbb{R}$ is strongly convex, of class $C^{1,1}_{\mathrm{loc}}$ and that $0 \le f(1) \le f(z)$ for every $z \in \mathbb{R}$.
The \emph{total energy} of the system $\F\colon[0,T]\times\U\times\Z\to[0,+\infty)$ is defined as the sum of elastic energy and dissipation pseudo-potential. Hence, for every $t\in[0,T]$, every $u\in\U$, and every $z\in\Z$ we set
\begin{equation}\label{e.F}
\F(t,u,z)\coloneq \E(t,u,z)+\D(z)\,.
\end{equation}



In our study of quasi-static evolutions we will often employ the following slopes for the functional~$\F$, w.r.t.~the displacement~$u$ and the phase field~$z$.

\begin{definition}\label{d.slope}
Let $(t,u,z) \in [0,T] \times \U \times \Z$. We define
\begin{align}
 |\partial_{u}\F|(t,u,z)&\coloneq 
\limsup_{ w\to u \text{ \rm in~$H^1$}}
\frac{(\F(t,w,z)-\F(t,u,z))_{-}}{\|w-u\|_{H^{1}}}\,,\label{e.slope1}\\
 |\partial_{z}^{-}\F|(t,u,z)&\coloneq 
\limsup_{ v \to z^- \text{ \rm in $L^2$}}
 \frac{(\F(t,u,v)-\F(t,u,z))_{-}}{\|v-z\|_{L^2}}\,,\label{e.slope2}
\end{align}
where $v \to z^-$ in $L^2$ means that $v \le z$ and $v \to z$ in $L^2$ (with $v \in \Z$).
\end{definition}

For the properties of the slopes we refer to Section~\ref{prop:slopes}.

\begin{remark} Note that here we employ a unilateral $L^2$-slope while in \cite{KneesNegri_M3AS17} we used a unilateral $H^1$-slope. \end{remark}

%

In order to simplify the notation later on, for a.e.~$t \in [0,T]$, every $u\in\U$, and every $z\in\Z$, we define the {\it power functional}
\begin{equation}\label{e.power} 
\P(t, u,z)\coloneq \int_{\Om}\partial_{\EE}W(z, \EE (u+g(t))) {\,:\,} \EE (\dot{g}(t)) \, \di x \,, 
\end{equation}
where $\dot{g}$ denotes the time derivative of $g$. We notice that for a.e.~$t\in[0,T]$, every $u\in \U$, and every $z\in\Z$ we have
\begin{equation}\label{e.pw}
\partial_{t} \F(t, u, z) = \P(t,u,z)\,.
\end{equation}

\subsection{Time-discrete evolutions and their time-continuous limit}

First, let us briefly describe the discrete alternate minimization scheme, without entering into the technical details. Let the initial condition be $u_0 \in \U$ and $z_0 \in\Z$ with $0 \le z_0 \le 1$ and
\begin{equation}\label{e.14.18}
 u_{0} = \min\,\{ \F(0,u,z_{0}):\, u \in \U\} \qquad\text{and}\qquad z_{0} = \min\,\{ \F(0,u_{0},z):\, z\in \Z,\, z\leq z_{0}\}\,.
\end{equation}
For $k \in \mathbb{N}$, $k \neq 0$, consider a time step $\tau_{k}:=T/k$ and let $t^{k}_{i}:= i\tau_{k}$ for every $i=0,\ldots, k$. The time-discrete evolution $( u^k_i , z^k_i)$ is defined by induction w.r.t.~the index $i\in\mathbb{N}$, as follows. We set $u^{k}_{0}\coloneq u_{0}$, $z^{k}_{0}\coloneq z_{0}$. In order to define~$u^k_{i}$ and~$z^k_{i}$, known~$u^k_{i-1}$ and~$z^k_{i-1}$, we need the auxiliary sequences~$u^{k}_{i,j}$ and~$z^k_{i,j}$ defined in this way: set $u^{k}_{i,0}:= u^{k}_{i-1}$ and $z^{k}_{i,0}:=z^{k}_{i-1}$, and, by induction w.r.t.~the index $j\in\mathbb{N}$, define by alternate minimization
\begin{eqnarray*}
&&\displaystyle u^{k}_{i,j+1}:=\argmin\,\{\F(t^{k}_{i},u,z^{k}_{i,j}):\,u\in \U \}\,,
\\[2mm]
&&\displaystyle z^{k}_{i,j+1}:=\argmin\,\{\F(t^{k}_{i},u^{k}_{i,j + 1},z):\, z\in \Z,\, z\leq z^{k}_{i,j}\}\,.
\end{eqnarray*}
We set $z^k_{i}=\lim_{j \to \infty} z^k_{i,j}$ and $u^k_{i}=\lim_{j \to \infty} u^k_{i,j}$ (existence of these sequences and of their limits will be proven in the sequel). 

In order to study the limit as $k \to \infty$, i.e.~as the time step $\tau_k$ vanishes, it will be technically convenient to interpolate all the configurations $u^k_{i,j}$ and $z^k_{i,j}$ by suitable rescaled gradient flows; this will ultimately provide, for every index $k$, an ``arc-length'' parametrization $s  \mapsto ( t_k (s) , u_k (s) , z_k (s))$ from a fixed inteval $[0,S]$ to  $[0,T] \times \U \times \Z$ which interpolates all the configuration $u^k_{i,j}$ and $z^k_{i,j}$. 

In the parametrized framework, 

\begin{definition}\label{d.continuitypoint}
A point $s \in [0,S]$ is a \emph{continuity point} for $(t,u,z)$ if for every $\delta>0$ there exists $s_\delta$ such that $| s_\delta -s | < \delta$ and $t( s_\delta ) \neq t (s)$. On the contrary, $s\in[0,S]$ is a \emph{discontinuity point} of~$(t,u,z)$ if~$t$ is constant in a neighborhood of~$s$.
\end{definition}

We are now ready to give the main result of this paper.

\begin{theorem}\label{t.1}
Up to subsequences, not relabelled, 
the parametrizations $(t_k, u_k, z_k)$ converge to a parametrization $(t,u,z)\colon[0,S]\to[0,T]\times\U\times\Z$ with $(t(0) , u(0) , z(0)) = (0,u_0,z_0)$, which satisfies the following properties: 
\begin{itemize}

\item [$(a)$] \emph{Regularity}: $(t,u,z)\in W^{1,\infty}([0,S];[0,T]\times H^1 ( \Omega ; \R^2) \times L^{2}(\Om))$, $z\in L^{\infty}([0,S];H^1(\Omega))$, and, for a.e.~$s\in[0,S]$,
\begin{displaymath}
|t'(s)|+\|u'(s)\|_{H^{1}}+\|z'(s)\|_{L^2}\leq 1\,,
\end{displaymath}
where the symbol~$'$ denotes the derivative w.r.t.~the parametrization variable $s$;

\item [$(b)$] \emph{Time parametrization}: the function $t\colon [0,S]\to[0,T]$ is non-decreasing and surjective; 

\item [$(c)$] \emph{Irreversibility}: the function $z\colon[0,S]\to\Z$ is non-increasing and $0 \le z(s) \leq 1$ for every $0\leq s \leq S$;

\item [$(d)$] \emph{Equilibrium}: for every continuity point $s\in[0,S]$ of $(t,u,z)$ 
\begin{displaymath}
|\partial_{u}\F|(t(s),u(s),z(s))=0\qquad\text{and}\qquad|\partial_{z}^{-}\F|(t(s),u(s),z(s))=0 \,;
\end{displaymath}

\item [$(e)$] \emph{Energy-dissipation equality}: for every $s\in[0,S]$
\begin{equation}\label{e.eneq}
\begin{split}
\F(t(s),& \ u(s),z(s))  = \F(0,u_0,z_0)-\int_{0}^{s}|\partial_{u}\F|(t(\sigma),u(\sigma),z(\sigma))\, \| u' (s) \|_{H^1} \di\sigma\\
&-\int_{0}^{s} \!\! |\partial_{z}^{-}\F|(t(\sigma),u(\sigma),z(\sigma))\, \| z' (\sigma) \|_{L^2} \di\sigma +\int_{0}^{s} \!\! \P( t(\sigma),u(\sigma), z(\sigma)) \,t'(\sigma)\,\di\sigma\,,
\end{split}
\end{equation}
where we intend that $|\partial_{z}^{-}\F|(t(\sigma),u(\sigma),z(\sigma))\, \| z' (s) \|_{L^2}=0$ whenever~$\| z'(\sigma)\|_{L^{2}}=0$ (including the case $|\partial_{z}^{-}\F|(t(\sigma),u(\sigma),z(\sigma))=+\infty$).
\end{itemize}
Any evolution satisfying the above properties, will be called \emph{parametrized Balanced Viscosity evolution}~\cite{MR3531671}.
\end{theorem}

The proof of this theorem is contained in Section~\ref{s.prooft1}.

\begin{remark}
We note that the equilibrium condition~\eqref{e.14.18} is not strictly necessary. 
However, it allows to shorten some proofs, without affecting the convergence analysis and the behavior of solutions. 
\end{remark}

\begin{remark}
The convention $|\partial_{u}\F|(t(\sigma),u(\sigma),z(\sigma))\, \| u' (\sigma) \|_{H^1}=0$ when~$\| u'(\sigma)\|_{H^{1}}=0$ is not necessary, since~$u(\sigma)\in H^{1}(\Om;\R^{2})$, which implies that $|\partial_{u}\F|(t(\sigma), u(\sigma), z(\sigma)) < +\infty$ for every~$\sigma\in[0,S]$.
\end{remark}

\begin{remark} 
In Section \ref{s.inequality} we prove a refined energy-dissipation identity which implies (see Appendix~\ref{AppB}) that the limit evolution may still present an alternate behavior in discontinuity points.
\end{remark}

\section{Lemmata} \label{s.lemmata}

In this section we collect some technical results that will be useful in the forthcoming discussions.

\subsection{Properties of the energy}

We first show some basic properties of the elastic energy density $W$.

\begin{lemma}\label{l.HMWw}
The function $W\colon\R\times\mathbb{M}^{2}_{s}\to[0,+\infty)$ is of class~$C^{1,1}_{\mathrm {loc}}$. Moreover, there exist two positive constants $c,C$ such that for every $z\in [0,1]$ and every $\strain_{1},\strain_{2} \in\mathbb{M}^{2}_{s}$ the following holds:
\begin{itemize}
\item[$(a)$] $\big(\partial_{\strain} W(z,\strain_{1})-\partial_{\strain}W(z,\strain_{2})\big){\,:\,}(\strain_{1}-\strain_{2})\geq c|\strain_{1}-\strain_{2}|^{2}$;\label{e.2}
\item[$(b)$] $\big|\partial_{\strain} W(z,\strain_{1})-\partial_{\strain}W(z,\strain_{2})\big|\leq C |\strain_{1}-\strain_{2}|$.\label{e.3} 
\end{itemize}
Since $\partial_{\strain} W ( z , 0 )  = 0$ it follows also that for every $\strain \in \mathbb{M}^2_s$ we have 
\begin{itemize}
\item[$(c)$] $| \partial_{\strain} W ( z , \strain ) | \leq C | \strain | $. \label{e.4}
\end{itemize}
\end{lemma}

\begin{proof} 
Write
\begin{displaymath}
	\partial_{\strain}W(z,\strain_{1})-\partial_{\strain}W(z,\strain_{2})= 2 h(z)\big( \mu (\strain_{1,d}-\strain_{2,d}) + \kappa (\strain_{1,v}^{+}-\strain_{2,v}^{+})\big) - 2\kappa(\strain_{1,v}^{-}-\strain_{2,v}^{-})\,.
\end{displaymath}
By linearity and orthogonality, to prove~$(a)$ it is enough to check that
\begin{eqnarray*}
	& h(z) \mu  (\strain_{1,d}-\strain_{2,d}) {\,:\,} (\strain_{1,d}-\strain_{2,d}) \geq \eta \mu | \strain_{1,d}-\strain_{2,d}|^{2}\,, \\[1mm]
& h(z)\kappa (\strain_{1,v}^{+}-\strain_{2,v}^{+}){\,:\,} (\strain_{1,v}-\strain_{2,v}) - \kappa(\strain_{1,v}^{-}-\strain_{2,v}^{-}){\,:\,}(\strain_{1,v}-\strain_{2,v}) \geq c | \strain_{1,v}  -   \strain_{2,v} |^{2}\,.
\end{eqnarray*}
The first inequality is straightforward. For the second we can write the left hand side  in terms of traces as
\begin{equation} \label{e.piuemeno}
\begin{split} 
\tfrac12 h(z)\kappa &\big((\tr\strain_{1})_{+}-(\tr\strain_{2})_{+}\big)(\tr\strain_{1}-\tr\strain_{2})- \tfrac12 \kappa \big((\tr\strain_{1})_{-}-(\tr\strain_{2})_{-}\big)(\tr\strain_{1}-\tr\strain_{2}) . 
\end{split}
\end{equation}
Let $c = \tfrac12 \kappa \min \{ h(z) , 1 \} \ge \tfrac12 \kappa \eta >0 $.
Since $(\cdot)_+$ is monotone non-decreasing we get 
$$
	\tfrac12 h(z)\kappa \big((\tr\strain_{1})_{+}-(\tr\strain_{2})_{+}\big)(\tr\strain_{1}-\tr\strain_{2}) \ge c \big((\tr\strain_{1})_{+}-(\tr\strain_{2})_{+}\big)(\tr\strain_{1}-\tr\strain_{2}) .
$$
Using the fact that $-(\cdot)_-$ is non-decreasing, we can argue in a similar way for the second term in \eqref{e.piuemeno} and get 
$$
	-	\tfrac12 \kappa \big((\tr\strain_{1})_{-}-(\tr\strain_{2})_{-}\big)(\tr\strain_{1}-\tr\strain_{2}) \ge - c \big((\tr\strain_{1})_{-}-(\tr\strain_{2})_{-}\big)(\tr\strain_{1}-\tr\strain_{2}) .
$$
Taking the sum of the last two inequalities gives the required estimate. 


Finally, $(b)$ follows from \eqref{stress} thanks to the fact that $z \in [0,1]$ and $h$ is continuous. \qed 
\end{proof}

\separe

We notice that for every $t\in[0,T]$, every $u,\varphi\in\U$, and every $z,\psi\in\Z$ we can express the partial derivatives of~$\F ( t, \cdot, \cdot)$ w.r.t.~$u$ and~$z$ as
\begin{align*}
\partial_{u}\F(t,u,z)[\varphi]&=\int_{\Om} \partial_{\strain}W(z,\strain(u + g(t) )){\,:\,}\strain(\varphi)\,\di x\\
&=2\int_{\Om}h(z)\big(\mu\strain_{d}(u + g(t)){\,:\,}\strain_{d}(\varphi)+\kappa\strain_{v}^{+}(u + g(t)) {\,:\,}\strain_{v}(\varphi)\big)\,\di x-2\kappa\int_{\Om}\strain_{v}^{-}(u + g(t)) {\,:\,}\strain_{v}(\varphi)\,\di x\,,\\[1mm]
 \partial_{z}\F(t,u,z)[\psi]&=\int_{\Om} \partial_z W(z,\strain(u+ g(t))) \psi \,\di x+\int_{\Om}\nabla{z}{\,\cdot\,}\nabla\psi\,\di x + \int_{\Om} f'(z) \psi\,\di x \\
	& = \int_{\Om}h'(z) \psi \Psi_+ \big( \strain (u + g(t)) \big)  \,\di x +\int_{\Om}\nabla{z}{\,\cdot\,}\nabla\psi\,\di x + \int_{\Om} f'(z) \psi\,\di x\,.
\end{align*}

\begin{remark} \label{r.sepstrconv}\textnormal{
It is important to note that the energy $\F ( t , \cdot ,\cdot)$ is separately strongly convex in $\U \times \Z$, with respect to the $H^1$-norms. More precisely, there exists $C>0$ such that, uniformly w.r.t.~$t \in [0,T]$ and $u \in \U$, it holds
\begin{equation} \label{e.strconv-z}
	\partial_z \F ( t, u , z_2 ) [ z_2 - z_1]  - \partial_z \F ( t , u , z_1)  [z_2 - z_1]  \ge C \| z_2 - z_1 \|^2_{H^1} \,,
\end{equation}
indeed, by convexity of $h$ and strong convexity of $f$, we can write the left hand side as 
\begin{align*}
	\int_\Omega [ h' (z_2) - h'(z_1) ]   (z_2 - z_1) \Psi_+ \big( \strain (u + g(t)) \big) &  + |  \nabla (z_2  -z_1) |^2 +  [ f' (z_2) - f'(z_1)] ( z_2 - z_1) \, dx  \ge \\
	& \ge \int_\Omega |  \nabla (z_2  -z_1) |^2 +  c \, ( z_2 - z_1)^2 \, dx  . 
\end{align*}
In a similar way, there exists $C>0$ such that, uniformly w.r.t.~$t \in [0,T]$ and $z \in \Z$, it holds
\begin{equation} \label{e.strconv-u}
		\partial_u \F ( t, u_2 , z ) [ u_2 - u_1]  - \partial_u \F ( t , u_1 , z)  [u_2 - u_1]  \ge C \| u_2 - u_1 \|^2_{H^1} \,,
\end{equation}
indeed, by (a) in Lemma \ref{l.HMWw} the left hand side reads
\begin{align*}
	\int_{\Om} \big(  \partial_{\strain}W(z,\strain(u_2 + g(t) ))  -  \partial_{\strain} W(z,\strain(u_1 + g(t) )) \big) {\,:\,}\strain(u_2 - u_1)\,\di x 
	 & \ge C \| \strain (u_2 + g(t) )   - \strain (u_1 + g(t)) \|^2_{L^2}  \\
	& \ge C \| \strain (u_2 - u_1) \|^2_{L^2} \ge C' \| u_2 - u_1 \|^2_{H^1} ,
\end{align*}
where we used Korn inequality for the last estimate. In particular, the elastic energy $\E (t , \cdot , z)$ is strongly convex.}
\end{remark}

\separe

\begin{lemma} \label{l.lscFE} Let $(t_m , u_m , z_m) \in [0,T] \times \U \times \Z$. If $t_m \to t$, $u_m \rightharpoonup u$ in $H^1(\Omega , \mathbb{R}^2)$, and $z_m \weakto z$ in $H^1(\Omega)$ then 
\begin{equation}\label{e.lscFE}
	\E ( t , u ,z) \le \liminf_{m \to \infty}  \E ( t_m , u_m ,z_m) , \qquad 
	\F ( t , u ,z) \le \liminf_{m \to \infty}  \F ( t_m , u_m ,z_m) .
\end{equation}
\end{lemma}

\begin{proof}
 Recalling the definition~\eqref{e.enelden} of the elastic energy~$W(z,\cdot)$ is convex in~$\mathbb{M}^{2}$ for every~$z \in \R$. Hence, we are in a position to apply~\cite[Theorem~7.5]{Fonseca2007} in order to deduce the first inequality in~\eqref{e.lscFE}. The second inequality follows immediately since the dissipation pseudo-potential~$\mathcal{D}$ is lower semicontinuous w.r.t.~weak convergence in~$H^{1}(\Om)$.\qed
\end{proof}
\separe

\subsection{Higher integrability and continuity of the displacement field} \label{prop:energy}

We now establish a uniform, continuous dependence estimates for the minimizer of the functional~$\F(t,\cdot,z)$ which follows from~\cite[Theorem~1.1]{HerzogMeyerWachsmuth_JMAA11}. 

In the following, for every $\beta \in(1,+\infty)$ we denote
\begin{displaymath}
W^{1,\beta}_{D}(\Om;\R^{2})\coloneq\{u\in W^{1,\beta}(\Om;\R^{2}):\,u=0\text{ on~$\partial_{D}\Om$}\}
\end{displaymath}
and let $W^{-1,\beta'}_{D}(\Om;\R^{2})$ be its dual. Furthermore, given $z\in\Z$ and $g\in W^{1,\beta}(\Om;\R^{2})$, we define the operator $A_{z,g}\colon W^{1,\beta}_{D}(\Om;\R^{2})\to W^{-1,\beta}_{D}(\Om;\R^{2})$ as 
\begin{equation}\label{e.53}
\langle A_{z,g}(u),\varphi\rangle\coloneq \int_{\Om}\partial_{\strain}W(z,\strain(u+g)){\,: \,}\strain(\varphi)\,\di x, \text{ for $\varphi\in W^{1,\beta'}_{D}(\Om;\R^{2})$.}
\end{equation}
With this notation, if $\xi \in W_D^{-1,\beta}(\Omega;\R^2)$ then $u = A_{z,g}^{-1}(\xi)$ if and only if  $ u \in W^{1,\beta}_D(\Om;\R^{2})$ is the solution of the variational problem
$$
	\int_{\Om}\partial_{\strain}W(z,\strain(u+g)){\,: \,}\strain(\varphi)\,\di x = \langle \xi , \varphi \rangle ,
	\text{ for every $\varphi \in W^{1,\beta'}_D(\Om;\R^{2}) $.}
$$

\begin{lemma} \label{l.HMWTh}
Let us fix $p>2$ and $M>0$. Then, there exists $\tilde{p}\in(2,p)$ such that the operator $A_{z,g} \colon W^{1,\beta}_{D}(\Om;\R^{2})\to W^{-1,\beta}_{D}(\Om;\R^{2})$ is invertible for every $\beta\in[2,\tilde{p}]$, every $g\in W^{1,p}(\Om;\R^{2})$, and every $z\in\Z$ with $\|z\|_{\infty}\leq M$. In particular, there exist two constants $C_{1},C_{2}>0$ (independent of~$g$, $z$, and $\beta \in [2,\tilde{p}]$) such that
\begin{equation}\label{e.54}
\|A_{z,g}^{-1}(\xi)\|_{W^{1,\beta}}\leq C_{1}(\|\xi\|_{W_D^{-1,\beta}} + \| g \|_{W^{1,\beta}}) \quad\text{and}\quad \|A^{-1}_{z,g}(\xi_{1}) - A^{-1}_{z,g}(\xi_{2})\|_{W^{1,\beta}}\leq C_{2}\|\xi_{1}-\xi_{2}\|_{W_D^{-1,\beta}}
\end{equation}
for every $\xi,\xi_{1},\xi_{2}\in W^{-1,\beta}_{D}(\Om;\R^{2})$.
\end{lemma}

\begin{proof}
The inequalities~\eqref{e.54} follow from a direct application of~\cite[Theorem~1.1 and Remark~1.3]{HerzogMeyerWachsmuth_JMAA11}, whose hypotheses are satisfied in view of Lemma~\ref{l.HMWw}. \qed
\end{proof}

By a direct application of Lemma~\ref{l.HMWTh}, for $M=1$, we deduce the next corollary.

\begin{corollary}\label{c.3}
Let $g\in W^{1,q}([0,T];W^{1,p}(\Om;\R^{2}))$ for $q\in(1,+\infty)$ and $p\in(2,+\infty)$. Let $\tilde{p} \in (2,p)$ be as in Lemma \ref{l.HMWTh}. 
%
Then, there exist a positive constant~$C_1$ such that for every
for every $\beta \in[2,\tilde{p}]$, $t\in [0,T]$, and $z \in \Z$ with $0 \le z \le 1$ it holds
\begin{equation}\label{e.pbound}
\| u \|_{W^{1,\beta}}\leq C_{1} \| g(t)\|_{W^{1,p}}\,,
\end{equation}
where $u :=\argmin\,\{\F(t,w,z):\,w\in\U\}$.

Moreover, there exists $\nu\in(2,+\infty)$ and a positive constant $C_2$ such that for every $\beta\in[2,\tilde{p})$, $t_{1},t_{2}\in[0,T]$, and $z_{1},z_{2}\in\Z$ with $0 \le z_i \le 1$ (for $i=1,2$) it holds 
\begin{equation}\label{e.25}
\|u_{1}-u_{2}\|_{W^{1,\beta}}\leq C_{2} (\|g(t_{1})-g(t_{2})\|_{W^{1,p}}+\|z_{1}-z_{2}\|_{L^\nu})\,,
\end{equation}
where $u_{i}=\argmin\,\{\F(t_i,w,z_i):\,w\in\U\}$ (for $i=1,2$) and $\frac{1}{\nu} = \frac{1}{\beta} - \frac{1}{\tilde{p}}$. 

\end{corollary}

\begin{proof}
Inequality~\eqref{e.pbound} is a direct consequence of Lemma~\ref{l.HMWTh}. Indeed, being $1 \le \beta' \le 2$ the Euler-Lagrange equation
\begin{displaymath}
\int_{\Om}\partial_{\strain}W(z,\strain(u +g(t))){\,:\,}\strain(\varphi)\,\di x=0\qquad\text{for every $\varphi\in\U = W^{1,2}_D(\Omega ; \R^2) \supset W^{1,\beta'}_D ( \Omega ; \R^2)$}
\end{displaymath}
gives $u = A^{-1}_{z,g(t)}(0)$. 
Applying Lemma~\ref{l.HMWTh} we deduce that $u\in W^{1,\beta}_{D}(\Om;\R^{2})$ for every $\beta \in [2, \tilde{p})$ and~\eqref{e.pbound} is satisfied.

Let us now show~\eqref{e.25}. 
Using the Euler-Lagrange equation for $u_2$, we get 
\begin{align*}
\int_{\Om} \partial_{\strain} & W(z_{1},\strain(u_{2}+g(t_{1}))){\,:\,}\strain(\varphi)\,\di x = \\
&=\int_{\Om} h(z_{1}) \big( \mu\strain_{d}(u_{2}+g(t_{1})){\,:\,}\strain_{d}(\varphi)+\kappa\strain_{v}^{+}(u_{2}+g(t_{1})){\,:\,}\strain_{v}(\varphi) \big) - \kappa  \strain_{v}^{-}(u_{2}+g(t_{1})){\,:\,}\strain_{v}(\varphi)\,\di x  \\
& \quad - \int_{\Om} h(z_{2}) \big( \mu\strain_{d}(u_{2}+g(t_{2})){\,:\,}\strain_{d}(\varphi)+\kappa\strain_{v}^{+}(u_{2}+g(t_{2})){\,:\,}\strain_{v}(\varphi) \big) - \kappa  \strain_{v}^{-}(u_{2}+g(t_{2})){\,:\,}\strain_{v}(\varphi)\,\di x \\
&=\int_{\Om}(h(z_{1})-h(z_{2}))(\mu\strain_{d}(u_{2}+g(t_{1})){\,:\,}\strain_{d}(\varphi)+\kappa\strain_{v}^{+}(u_{2}+g(t_{1})){\,:\,}\strain_{v}(\varphi))\,\di x \nonumber\\
&\quad +\mu\int_{\Om}h(z_{2})\strain_{d}(g(t_{1})-g(t_{2})){\,:\,}\strain_{d}(\varphi)\,\di x\\
&\quad + \kappa\int_{\Om} h(z_{2}) ( \strain_{v}^{+} ( u_{2} + g(t_{1}) ) - \strain_{v}^{+} ( u_{2} + g(t_{2})) ) {\,:\,} \strain_{v} ( \varphi ) \, \di x \nonumber \\
&\quad +\kappa \int_{\Om} (\strain_{v}^{-}(u_{2}+g(t_{2}))-\strain_{v}^{-}(u_{2}+g(t_{1})){\,:\,}\strain_{v}(\varphi)\,\di x =:\langle \xi,\varphi\rangle\,, \nonumber 
\end{align*}
where $\xi\in L^{\beta}(\Om;\mathbb{M}^{2})$ for every $\beta\in[2,\tilde{p})$. Therefore, $u_2 = A^{-1}_{z_1, g(t_1)} ( \xi)$ while $u_1 = A^{-1}_{z_1, g(t_1)} (0)$; applying the second  estimate of Lemma~\ref{l.HMWTh}, we get that there exists a positive constant~$C_{2}$ (independent of~$z_{i}$, $t_{i}$, and $\beta \in [2, \tilde{p})$) such that
\begin{equation} \label{e.56}
\|u_{1}-u_{2}\|_{W^{1,\beta}}\leq C_{2}\|\xi\|_{L^\beta}\,.
\end{equation}
Let $\tfrac{1}{\beta} = \tfrac{1}{\nu} + \tfrac{1}{\tilde{p}}$, then by H\"older inequality we have 
\begin{equation}\label{e.57}
\|\xi\|_{L^\beta}\leq \|h(z_{1})-h(z_{2})\|_{L^\nu}\|u_{2}+g(t_{1})\|_{L^{\tilde{p}}}+ C(1 + \|h\|_{L^\infty(0,1)} ) \|g(t_{1})-g(t_{2})\|_{W^{1,\beta}}\,.
\end{equation}
Since $0 \le z_{i} \le 1$ and $h\in C^{1}(\R)$, we have that $\| h(z_{1}) - h(z_{2})\|_{L^\nu}\leq C\|z_{1}-z_{2}\|_{L^\nu}$ for some positive constant $C$. Combining~\eqref{e.56} and~\eqref{e.57} we obtain~\eqref{e.25}, and the proof is concluded. \qed
\end{proof}

\subsection{Continuous dependence of the phase field} \label{prop:phase}

\begin{proposition}\label{p.5}
Let $\tilde{p}\in(2,+\infty)$ be as in Lemma~\ref{l.HMWTh}. Let $t_{1}, t_{2} \in[0,T]$, $u_{1},u_{2}\in W^{1,\tilde{p}}(\Om;\R^{2})$, and $z_{0},z_{1},z_{2}\in H^{1}(\Om;[0,1])$ be such that
\begin{equation}\label{e.26}
z_{i}=\argmin\,\{\F(t_{i} ,u_{i},z):\,z\in \Z ,\,z_{i}\leq z_{i-1}\}\qquad\text{for $i=1,2$}\,.
\end{equation}
Then there exist a positive constant~$C$, independent of~$t$,~$u_{i}$, and~$z_{i}$, such that
\begin{equation}\label{e.27}
\|z_{1}-z_{2}\|_{H^{1}}\leq  C(\|u_{1}+ g(t_{1}) \|_{W^{1,\tilde{p}}} + \| u_{2} + g(t_{2}) \|_{W^{1,\tilde{p}}})( \| u_{1} - u_{2} \|_{H^{1}} + \| g(t_{1}) - g(t_{2}) \|_{H^{1}}) \,.
\end{equation}

\end{proposition}

\begin{proof} We adapt the proof of \cite[Lemma A.2]{KneesNegri_M3AS17}.
By~\eqref{e.26}, for every~$\varphi\in H^{1}(\Om)$, $\varphi\leq0$, we have
\begin{equation}\label{e.28}
\partial_{z}\F(t_{i}, u_{i},z_{i})[\varphi]\geq0\qquad\text{for~$i=1,2$}\,.
\end{equation}
Moreover, for $\varphi=z_{2}-z_{1}$ we get
\begin{equation}\label{e.29}
\partial_{z}\F(t_{2},u_{2},z_{2})[z_{2}-z_{1}]=0\,.
\end{equation}
Therefore, combining~\eqref{e.28} and~\eqref{e.29}, we obtain
\begin{equation*}\label{e.30}
\big(\partial_{z}\F(t_{2}, u_{2}, z_{2}) - \partial_{z} \F ( t_{1}, u_{1}, z_{1} ) \big ) [ z_{2} - z_{1} ] \leq 0 \,.
\end{equation*}
Adding and subtracting the term $\partial_{z}\F(t_{1},u_{1},z_{2})[z_{2}-z_{1}]$ to the previous inequality, we get
\begin{equation}\label{e.31}
\big(\partial_{z}\F(t_{1},u_{1},z_{2})-\partial_{z}\F(t_{1},u_{1},z_{1})\big)[z_{2}-z_{1}]\leq\big(\partial_{z}\F(t_{1},u_{1},z_{2})-\partial_{z}\F(t_{2},u_{2},z_{2})\big)[z_{2}-z_{1}]\,.
\end{equation}
The left-hand side of~\eqref{e.31} reads as
\begin{align*}
\| \nabla{z_{1}} - \nabla{z_{2}}\|_{L^2}^{2} & + \int_{\Om}(h'(z_{2})-h'(z_{1}))(z_{2}-z_{1})\big(\mu|\strain_{d}(u_{1} + g(t_{1}) )|^{2}+\kappa|\strain_{v}^{+}(u_{1}+g(t_{1}))|^{2}\big)\, \di x \\ & + \int_\Omega ( f'(z_2) - f'(z_1) ) ( z_2 - z_1) \, \di x  \,.
\end{align*}
Being $h$ convex, the second term in the previous expression is positive. By the strong convexity of $f$ we have
$$
	( f' (z_2) - f'(z_1) )  ( z_2 - z_1) \ge C | z_2 - z_1 |^2 .
$$
Thus, we can continue in~\eqref{e.31} with
\begin{equation*}
\begin{split}
C' \|z_{1}-z_{2}\|_{H^{1}}^{2} & \leq \big(\partial_{z}\F(t_{1},u_{1},z_{2})-\partial_{z}\F(t_{1},u_{1},z_{1})\big)[z_{2}-z_{1}] \le 
\big(\partial_{z}\F(t_{1},u_{1},z_{2})-\partial_{z}\F(t_{2},u_{2},z_{2})\big)[z_{2}-z_{1}] 
\end{split}
\end{equation*}
where the right hand side reads
\begin{align*}
& \big(\partial_{z}\F(t_{1},u_{1},z_{2}) -\partial_{z}\F(t_{2},u_{2},z_{2})\big)[z_{2}-z_{1}]  \\ 
 & = \int_{\Om} h'(z_{2}) (z_{2}-z_{1}) \big(\mu |\strain_{d}(u_{1}+ g(t_{1}))|^2 - \mu |\strain_{d}( u_{2} + g(t_{2}))|^2 + \kappa|\strain_{v}^{+}(u_{1}+ g(t_{1}))|^2 - \kappa |\strain_{v}^{+}(u_{2}+ g(t_{2}))|^2 \big)\, \di x \,.
\end{align*}
Since $z_{2}\in H^{1}(\Om;[0,1])$, we have that~$h'(z_{2})$ is bounded. Moreover, 
\begin{align*}
	|\strain_{v}^{+}(u_{1}+ g(t_{1}))| - |\strain_{v}^{+}(u_{2}+ g(t_{2}))| & =  ( \tr ( \strain (u_{1}+ g(t_{1})) ) )_+  - ( \tr  ( \strain (u_{2}+ g(t_{2})) ) )_+ \\
	& \le  | \tr ( \strain ( u_1 +g(t_{1})- u_2 - g(t_{2})) ) | \\ & = |\strain_{v} (u_{1} + g(t_{1}) - u_2 - g(t_{2}) )| .
\end{align*}
Thus, there exists a positive constant $C$ such that
\begin{align*}\label{e.33}
\|z_{1}-z_{2}\|_{H^{1}}^{2}  \leq & \ C \int_{\Om} |z_{2}-z_{1}|\big( |\strain_{d}(u_{1}+ g(t_{1}))| + |\strain_{d}(u_{2} + g(t_{2})) | \big)  | \strain (u_{1} + g(t_{1}) - u_{2} - g(t_{2}))| \, \di x\\
& + C\int_{\Om}| z_{2} - z_{1} | \big( | \strain_{v}^{+}(u_{1} + g(t_{1}))| + |\strain_{v}^{+}(u_{2} + g(t_{2})) | \big) | \strain (u_{1} + g(t_{1}) - u_{2} - g(t_{2}))| \, \di x .
\end{align*}

By hypothesis, we have $u_{1},u_{2}\in W^{1,\tilde{p}}(\Om;\R^{2})$. Hence, applying H\"older inequality with $ \tfrac{1}{\alpha} + \tfrac{1}{\tilde{p}} + \tfrac12 = 1$ 
we get that
\begin{equation*}\label{e.34}
\|z_{1}-z_{2}\|_{H^{1}}^{2} \leq C \| z_{2} - z_{1} \|_{L^\alpha} ( \| u_{1} + g(t_{1}) \|_{ W^{ 1, \tilde{p} } } + \| u_{2} + g(t_{2}) \|_{ W^{1,\tilde{p}}} ) \| u_{1} + g(t_{1}) - u_{2} - g(t_{2}) \|_{H^{1}} .
\end{equation*}
Inequality~\eqref{e.27} follows by triangle inequality and by Sobolev embedding in dimension~$2$. \qed


\end{proof}

\begin{proposition}\label{p.6}
Let $\tilde{p}\in(2,+\infty)$ be as in Lemma~\ref{l.HMWTh}. Let $( t_{k}, u_k , z_k) \in [0,T] \times \U \times \Z$ with $0 \le z_k \le 1$ and 
\begin{displaymath}
u_{k}=\argmin\,\{\F(t_{k},v,z_{k}):\,v\in\U\} \qquad \text{for every~$k$}\,.
\end{displaymath}
If $t_{k}\to t$, $z_{k}\rightharpoonup z$ in~$H^{1}(\Om)$, and $u\coloneq\argmin\,\{\F(t,v,z):\, v\in\U\}$, then 
$u_{k}\to u$ in $W^{1,\beta}(\Om;\R^{2})$ for every $\beta\in[2,\tilde{p})$.
\end{proposition}

\begin{proof}
In view of the hypotheses of the proposition and of Corollary~\ref{c.3}, we have that the sequence~$u_{k}$ is a Cauchy sequence in~$W^{1,\beta}(\Om;\R^{2})$ for every $\beta\in[2,\tilde{p})$. We denote by~$\overline{u}$ the limit function. By the strong convergence in~$W^{1,\beta}(\Om;\R^{2})$, it is easy to see that~$\overline{u}$ is the solution of
$\min\,\{\F(t,v,z):\,v\in\U\}$.
Hence, by uniqueness of minimizer we have $\overline{u}=u$.\qed
\end{proof}


\subsection{Properties of the slopes} \label{prop:slopes}

Now, we can give a convenient characterization of the slopes introduced in Definition~\ref{d.slope}.  

\begin{remark}\label{p.1}
Let $(t, u, z) \in [0,T] \times \U \times \Z$, then
\begin{align}
|\partial_{u}\F|(t,u,z)&=\max \,\{-\partial_{u}\F(t,u,z)[\varphi]:\,\varphi\in \U,\, \|\varphi\|_{H^{1}}\leq 1\}\,,\label{e.slope3}\\[1mm]
|\partial_{z}^{-}\F|(t,u,z)&=\sup\,\{-\partial_{z}\F(t,u,z)[\psi]:\,\psi\in \Z,\,\psi\leq 0,\,\|\psi\|_{L^2}\leq 1\}\label{e.slope4}\,.
\end{align}
For the proof of \eqref{e.slope3} we refer for instance to \cite[Proposition 1.4.4]{AGS}. For the proof of \eqref{e.slope4} it is sufficient to employ the arguments of~\cite[Lemma~2.3]{A-B-N17} or \cite[Lemma 2.2]{Negri_ACV}.
\end{remark}

Next two lemmata are devoted to lower semicontinuity and continuity of the slopes.

\begin{lemma}\label{l.2}
Let $(t_k, u_k, z_k) \in [0,T] \times \U \times \Z$ such that $t_{k} \to t$ in $[0,T]$, $u_{k}\rightharpoonup u$ weakly in~$H^1( \Omega , \mathbb{R}^2)$, and $z_{k}\rightharpoonup z$ weakly in~$H^{1}(\Om)$ with $0\leq z_{k}\leq 1$, for every~$k$. Then
\begin{displaymath}
|\partial_{z}^{-}\F|(t,u,z)\leq\liminf_{k \to \infty}|\partial_{z}^{-}\F|(t_{k},u_{k},z_{k})\,.
\end{displaymath}
\end{lemma}

\begin{proof}
By Remark~\ref{p.1}, for every~$\psi\in\Z$ such that $\psi\leq0$ and $\|\psi\|_{L^2}\leq1$ we have that
\begin{equation}\label{e.10}
\begin{split}
|\partial_{z}^{-}\F|(t_{k},u_{k},z_{k})\geq-\int_{\Om} & h' (z) \psi \Psi_+ \big( \strain(u_{k} + g(t_{k})) \big) \,\di x 
-\int_{\Om} \nabla{z_{k}}{\,\cdot\,}\nabla{\psi} + f'( z_{k} ) \psi \,\di x\,.
\end{split}
\end{equation}
Since~$h\in C^{1,1}_{\mathrm{loc}}(\R)$ is non-decreasing in~$[0,+\infty)$ and $z_{k} \to z$ in~$L^{r}( \Om )$ for every $r<+ \infty$ with $0\leq z_{k} \leq 1$, we deduce that $-h'(z_{k})\psi\geq0$ for every~$k$ and that~$h'(z_{k})\psi\to h'(z)\psi$ in~$L^{r}(\Om)$ for every $r<+\infty$. In a similar way $f'(z_k) \psi \to f'(z) \psi$ in $L^1(\Omega)$. Hence, passing to the liminf in~\eqref{e.10} as $k\to+\infty$ and applying for instance~\cite[Theorem~7.5]{Fonseca2007} we deduce that
\begin{displaymath}
\liminf_{k \to \infty}|\partial_{z}^{-}\F|(t_{k},u_{k},z_{k})  \geq-\int_{\Om} h' (z) \psi \Psi_+ \big( \strain(u + g(t)) \big) \,\di x -\int_{\Om} \nabla{z}{\,\cdot\,}\nabla{\psi} + f'(z) \psi \,\di x  = -\partial_{z} \F (t, u, z)[\psi] \,.
\end{displaymath}
We conclude by taking the supremum over~$\psi$ in the previous inequality. \qed
\end{proof}

\begin{lemma}\label{l.3} Let $(t_k, u_k, z_k) \in [0,T] \times \U \times \Z$ such that
$t_{k}\to t$ in~$[0,T]$, $u_{k} \to u$ in~$H^1( \Omega , \mathbb{R}^2)$, and $z_{k}\rightharpoonup z$ in~$H^1(\Omega)$ with $0\leq z_{k}\leq 1$, for every~$k$. Then
\begin{displaymath}
|\partial_{u}\F|(t,u,z)=\lim_{k \to \infty}\,|\partial_{u}\F|(t_{k},u_{k},z_{k})\,.
\end{displaymath}
\end{lemma}

\begin{proof}
By Remark~\ref{p.1}, for every $\varphi\in\U$ with~$\|\varphi\|_{H^1}\leq 1$ we have that
\begin{equation}\label{e.21}
\begin{split}
|\partial_{u}\F|(t_{k},u_{k},z_{k}) \geq -\int_{\Om}  \partial_{\strain} W( z_{k}, \strain (u_{k} + g(t_{k}) ){\,:\,} \strain (\varphi) \,\di x .
\end{split}
\end{equation}
Remember that $\partial_{\strain} W(z,\strain ) = 2 h(z) \big(  \mu \strain_{d} + \kappa \strain_{v}^{+} \big) - 2\kappa \strain_{v}^{-}$.
Since $u_{k}\to u$ in~$H^1(\Omega, \R^2)$ and since $g$ belongs to $W^{1,q}([0,T];W^{1,p}(\Om;\R^{2}))$, we have that $\strain_{d}(u_{k}+g(t_{k}))\to \strain_{d}(u+g(t))$ and $\strain^{\pm}_{v}(u_{k}+g(t_{k}))\to \strain_{v}^{\pm}(u+g(t))$ in $L^{2}(\Om;\mathbb{M}^{2})$. Being $0\leq z_{k}\leq 1$ and $z_{k}\to z$ in~$L^{2}(\Om)$, we have that $h(z_{k})(\mu\strain_{d}(u_{k}+g(t_{k}))+\kappa\strain^{+}_{v}(u_{k}+g(t_{k})))$ converges to $h(z)(\mu\strain_{d}(u+g(t))+\kappa\strain_{v}^{+}(u+g(t)))$ in~$L^{2}(\Om;\mathbb{M}^{2})$. Therefore, $ \partial_{\strain} W( z_{k}, \strain (u_{k} + g(t_{k}) )$ converges to $ \partial_{\strain} W( z, \strain (u + g(t) )$ in~$L^{2}(\Om;\mathbb{M}^{2})$ and, passing to the liminf in~\eqref{e.21}, we obtain
\begin{displaymath}
\begin{split}
\liminf_{k\to\infty}\,|\partial_{u}\F|(t_{k},u_{k},z_{k})\geq - \int_{\Om} & \ \partial_{\strain} W( \strain (u_{} + g(t_{}))) {\,:\,} \strain(\varphi) \, \di x = - \partial_{u}\F(t,u,z)[\varphi]\,.
\end{split}
\end{displaymath}
Passing to the supremum over $\varphi\in\U$ with $\|\varphi\|_{H^1}\leq 1$, we deduce that
\begin{equation}\label{e.22}
|\partial_{u}\F|(t,u,z)\leq \liminf_{k\to\infty}\,|\partial_{u}\F|(t_{k},u_{k},z_{k})\,.
\end{equation}

As for the opposite inequality, for every~$k$ let $\varphi_{k}\in\U$ with $\|\varphi_{k}\|_{H^1}\leq 1$ be such that $|\partial_{u}\F|(t_{k},u_{k},z_{k})=-\partial_{u}\F(t_{k}.u_{k},z_{k})[\varphi_{k}]$. Up to a subsequence, we have that $\varphi_{k}\rightharpoonup \varphi$ weakly in~$H^1(\Omega ; \R^2)$ for some $\varphi\in\U$ with $\|\varphi\|_{H^1}\leq 1$. Hence, by the strong convergence of $\partial_{\strain} W( z_{k}, \strain (u_{k} + g(t_{k}) )$, we get that
\begin{equation}\label{e.23} 
\begin{split}
\limsup_{k\to\infty} \, | \partial_{u} \F | ( t_{k}, u_{k}, z_{k} ) & = - \partial_{u} \F ( t_{k}, u_{k}, z_{k} ) [ \varphi_{k} ]  \\
&= \limsup_{k\to\infty} \, - \int_{\Om}\partial_{\strain} W(\strain (u_{k} + g(t_{k}) ) ) {\,:\,}\strain (\varphi_{k} ) \, \di x  \\
& = - \int_{\Om}  \partial_{\strain} W( z, \strain (u + g(t) ) {\,:\,} \strain (\varphi) \, \di x  = - \partial_{u} \F ( t, u, z ) [\varphi] \leq | \partial_{u} \F | ( t, u, z ) \,. 
\end{split}
\end{equation}
This concludes the proof of the proposition. \qed

\end{proof}


%
%


\section{Auxiliary gradient-flows} \label{s.gf}

In this section we present some auxiliary results for two gradient flows which will be employed in the interpolation of the discrete evolutions obtained by alternate minimization. 

\subsection[An $H^1$-gradient flow for the displacement field]{An \boldmath{$H^1$}-gradient flow for the displacement field}

Given, $t\in[0,T]$ and $z\in \Z$, we start with recalling some results about the system
$$ 
\left\{\begin{array}{ll}
u'(l)=-\nabla_{u}\F(t,u(l),z)=-\nabla_{u}\E(t,u(l),z)\,,\\
u(0)=u^{0}\,,
\end{array}\right.
$$
where $u^0 \in \U$ and $\nabla_u  \F ( t, u , z)$ denotes the $H^1$-element representing, by Riesz Theorem, the functional $\partial_u \F ( t, u ,z)$, i.e., $\partial_u \F ( t,u,z) [\phi] =  \langle \nabla_u \F ( t,u,z) ,  \phi \rangle$ for every $\phi\in \U$. Note that $\| \nabla_u \F ( t,u,z) \|_{H^1} = | \partial_u \F | ( t,u,z)$. 

\begin{theorem}\label{t.2}
Let $(t, u^{0}, z)\in [0,T] \times \U \times \Z$. Then, there exists a unique evolution $u \colon [0,+\infty)\to \U$ such that the following facts hold:
\begin{itemize}
\item[$(a)$] $u\in W^{1,\infty}([0,+\infty); H^{1}(\Om;\R^{2}))$ and $u'\in  L^{2}([0,+\infty); H^{1}(\Om;\R^{2}))$;

\item[$(b)$] $u(0)=u^0$ and for a.e.~$l \in[0,+\infty)$ we have $u'(l)=-\nabla_{u}\F(t,u(l),z)$;

\item[$(c)$] for every $\ell \in[0,+\infty)$
\begin{equation}\label{e.38}
\F(t,u(\ell),z)=\F(t,u^0,z)-\tfrac{1}{2}\int_{0}^{\ell}|\partial_{u}\F|^{2}(t,u(l),z)+\|u'(l)\|^{2}_{H^{1}}\,\di l\,;
\end{equation} 

\item[$(d)$] $u(l)$ converges strongly to $\overline{u}$ in~$H^{1}(\Om;\R^{2})$ as $l \to+\infty$, where $\overline{u}=\argmin\,\{\F(t,u,z):\,u\in\U\}$.   Moreover, 
\begin{align}
& \F(t,\overline{u},z) =\F(t,u^0,z)-\tfrac{1}{2}\int_{0}^{+\infty}|\partial_{u}\F|^{2}(t,u(l),z)+\|u'(l)\|^{2}_{H^{1}}\,\di l\,, \label{e.41}\\[1mm]
& \|u(l)-\overline{u}\|_{H^{1}}  \leq e^{-cl}\|u^0-\overline{u}\|_{H^{1}}\,,\label{e.40}
\end{align}
where $c$ depends only on the constant appearing in (a) of Lemma \ref{l.HMWw}.
\end{itemize}
\end{theorem}

\begin{proof} We invoke \cite[Theorem 3.1, Lemma 3.3, and Theorem 3.9]{Brezis_73} for the operator $\mathcal{A}:= \nabla_{u} \E (t, \cdot, z )\colon \U \to \U$. Indeed, $\mathcal{A}$ is maximal monotone, by convexity and continuity of~$\E (t,\cdot, z)$. Moreover, by~$(a)$ of Lemma~\ref{l.HMWw} and by Korn inequality, the operator~$\mathcal{A}$ is strongly monotone, that is,
\begin{equation}\label{e.40.1}
 \langle \mathcal{A} u - \mathcal{A} v , u - v \rangle \geq c  \| u - v \|_{H^{1}}^{2} \,.
\end{equation}
Therefore, we are in a position to apply~\cite[Theorem 3.1 and Theorem 3.9]{Brezis_73} which, put together, prove~$(b)$, that $u'\in L^{\infty}([0,+\infty); H^{1}(\Om;\R^{2}))$, and that~$u(l)$ admits the limit~$\overline{u}=\argmin\,\{\F(t,u,z):\,u\in\U\}$ in~$H^{1}(\Om;\R^{2})$ as $l\to+\infty$, and the exponential decay~\eqref{e.40}, where the constant~$c$ coincides with the ellipticity constant of~\eqref{e.40.1}. In view of~\cite[Lemma~3.3]{Brezis_73} we get~$(c)$ and the uniform boundedness of~$u(\cdot)$ in~$H^{1}(\Om;\R^{2})$. Passing to the limit in~\eqref{e.38} as~$\ell\to+\infty$ and applying monotone convergence theorem, we deduce~\eqref{e.41} and that $u'\in L^{2}([0,+\infty); H^{1}(\Om;\R^{2}))$.  \qed
\end{proof}

Moreover, by a \L ojasiewicz \cite{Lojasiewicz_AIF93} argument we have the following result on the length of the flow.

\begin{theorem}\label{t.2,5} 
Let~$u$ be the solution of the above gradient flow. Then, either $u(l) \equiv u^0$ or $\| u' (l) \|_{H^1} \neq 0$ in $[0, +\infty)$. Moreover, there exists a constant~$\overline{C}$ (independent of $t$, $u^0$, and $z$) 
such that
\begin{equation}\label{e.lengthu}
\int_{0}^{+\infty}\|u'(l)\|_{H^{1}}\,\di l\leq \overline{C} \|u^0-\overline{u}\|_{H^{1}}\,,
\end{equation}
\end{theorem}

\begin{proof} 
If $u^{0}=\overline{u}$, then $u(l)\equiv \underline{u}$, since $\underline{u}=\argmin\,\{\F(t,u,z):\, u\in \U\}$. Let us therefore assume that $u^{0}\neq \overline{u}$. In what follows we denote by~$C$ a generic positive constant which could change from line to line. Let $s(l) = \frac12 \| u(l) - \overline{u} \|^2_{H^1}$. Then, being $\nabla_u \F ( t, \overline{u} , z) = 0$ and by $(b)$ in Lemma \ref{l.HMWw}
$$
	s'(l) = \langle u(l) - \overline{u} , u' (l) \rangle = \langle u(l) - \overline{u}  , - \nabla_u \F( t , u(l) , z) + \nabla_u \F ( t , \overline{u} , z) \rangle \ge - C \| u(l) - \overline{u} \|^2_{H^1} = - C' s(l)  \,.
$$
It follows that $s(l) \ge s(0) e^{-C'l} >0$ and then $u(l) \neq \overline{u}$ for every $l \in [0,+\infty)$. As a consequence $u '(l) = - \nabla_u \F (t, u(l),z) \neq 0$ for a.e.~$l \in [0,+\infty)$. 

We now prove the bound \eqref{e.lengthu}. By convexity, for every $l\in [0,+\infty)$ we have
\begin{equation}\label{e.18.09}
0 \leq \F (t, u(l), z) - \F (t, \overline{u}, z) \leq \langle \nabla_{u} \F (t, u(l), z) , u(l) - \overline{u} \rangle \leq \| \nabla_{u} \F (t, u(l), z)\|_{H^1} \| u(l) - \overline{u} \|_{H^1} \,.
\end{equation} 
By $(a)$ of Lemma \ref{l.HMWw} and by Korn inequality, we get
\begin{displaymath}
\begin{split}
C \| u (l) - \overline{u}\|_{H^{1}}^{2}&\leq \left\langle \nabla_{u}\F ( t, \overline{u}, z ) - \nabla_{u} \F ( t, u(l), z) , \overline{u} - u(l) \right\rangle = - \left\langle \nabla_{u} \F ( t, u(l), z) , \overline{u} - u(l) \right\rangle \\
&\leq \| \nabla_{u} \F ( t, u(l), z) \|_{H^1} \| u(l) - \overline{u} \|_{H^1} \,,
\end{split}
\end{displaymath}
which implies
\begin{equation}\label{e.18.07}
 C \| u(l) - \overline{u} \|_{H^1} \leq \| \nabla_{u} \F ( t, u(l), z) \|_{H^1}\,.
\end{equation}
 Combining \eqref{e.18.09} and \eqref{e.18.07} we deduce that
\begin{equation}\label{e.18.10}
C (\F (t, u(l), z) - \F (t, \overline{u}, z))^{\frac{1}{2}} \leq \| \nabla_u \F (t, u(l), z) \|_{H^1} \,.
\end{equation}
We now apply a \L ojasiewicz argument: in view of $(b)$ of Theorem \ref{t.2} and of the monotonicity of $l\mapsto \F(t, u(l), z)$, for a.e. $l\in [0,+\infty)$ we have
\begin{equation}\label{e.18.36}
\begin{split}
-2\frac{\di}{\di l}\big( \F ( t, u(l), z ) - \F ( t, \overline{u}, z ) \big)^{\frac12} & = - \big( \F (t, u(l), z ) - \F ( t,\overline{u}, z ) \big)^{-\frac12} \langle \nabla_{u} \F (t, u(l), z), u'(l) \rangle \\ 
& = \big( \F (t, u(l), z ) - \F ( t,\overline{u}, z ) \big)^{-\frac12} \| \nabla_{u}\F ( t, u(l) , z )\|_{H^1}^{2} \\
& \geq C \| \nabla_{u}\F ( t, u(l) , z )\|_{H^1} = C \| u'(l) \|_{H^{1}} \,.
\end{split}
\end{equation}
Hence, inequality~\eqref{e.18.36} implies that for every~$\ell\in[0,+\infty)$
\begin{equation}\label{e.18.37}
C\int_{0}^{\ell} \| u'(l) \|_{H^{1}} \, \di l \leq - 2 \int_{0}^{\ell} \frac{\di}{\di l} \big ( \F (t, u(l), z) - \F (t ,\overline{u}, z ) \big)^{\frac12} \, \di l \leq 2 \big( \F (t, u^{0}, z ) - \F ( t, \overline{u}, z ) \big)^{\frac12} \,.
\end{equation}
In the limit as~$\ell\to+\infty$ in~\eqref{e.18.37} we obtain by monotone convergence theorem
\begin{equation}\label{e.18.38}
C\int_{0}^{+\infty} \| u'(l) \|_{H^{1}} \, \di l \leq 2 \big( \F (t, u^{0}, z ) - \F ( t, \overline{u}, z ) \big)^{\frac{1}{2}} \,.
\end{equation}
By convexity and minimality of~$\overline{u}$ we have that 
\begin{align}
 \big( \F ( t, u^{0}, z ) - \F ( t, \overline{u}, z ) \big)^{\frac12} & \leq \big( -\left\langle \nabla_{u} \F ( t, u^{0}, z ) , \overline{u} - u^0 \right\rangle \big)^{\frac12} \nonumber \\ 
	& = \big( \left\langle \nabla_{u} \F ( t, \overline{u}, z ) , \overline{u} - u^0 \right\rangle - \left\langle \nabla_{u} \F ( t, u^{0}, z ) , \overline{u} - u^0 \right\rangle \big)^{\frac12}  \,,\label{e.18.43}  \\
		& \leq C \| u^{0} - \overline{u} \|_{H^{1}} \nonumber \,,
\end{align}
where in the last inequality we applied~$(b)$ of Lemma~\ref{l.HMWw}. Combining~\eqref{e.18.37} and~\eqref{e.18.43} we conclude~\eqref{e.lengthu}. In particular, we notice that all the constants appearing in~\eqref{e.18.07}-\eqref{e.18.43} do not depend on~$t$,~$u^{0}$,~$z$, and $l$.\qed 
\end{proof} 

As a corollary of Theorems~\ref{t.2} and~\ref{t.2,5}, we define a suitable reparametrization of $l\in[0,+\infty)$ which makes the gradient flows computed in Theorem~\ref{t.2} $1$-Lipschitz continuous. This reparametrization will be exploited in Section~\ref{s.prooft1} for the proof of Theorem~\ref{t.1}.

\begin{corollary}\label{c.gfu}
Let~$t$,~$u^{0}$, and~$z$ be as in the statement of Theorem~\ref{t.2}. Let~$u$ be the gradient flow computed in Theorem~\ref{t.2} with initial condition~$u^{0}$. Let us assume that $u^{0} \neq \overline{u}:= \argmin\,\{ \F(t, u, z):\, u\in\U \}$, and let us set
\begin{displaymath}
L (u) :=\int_{0}^{+\infty} \| u' (l) \|_{H^{1}}\,\di l \,,  \qquad  \lambda( \ell ):=\int_{0}^{\ell} \| u' (l) \|_{H^{1}}\,\di l \,, \quad \text{for $\ell \in [0, +\infty ]$} .   \\
\end{displaymath}
Moreover, let $\rho\colon [0,L(u)]\to [0,+\infty]$ be defined by $\rho:= \lambda^{-1}$. Then, the function $\omega:= u\circ \rho$ belongs to $W^{1,\infty} ([0,L(u)]; H^{1}(\Om;\R^{2}) )$, $\| \omega' (s) \|_{H^1} = 1$ for a.e.~$s \in [0, L(u)]$, $\omega(0)=u^{0}$, $\omega(L(u)) = \overline{u}$, and
\begin{equation}\label{e.enbal}
\F( t, \omega(s), z) = \F( t, u^{0}, z)-\int_{0}^{s} |\partial_{u} \F|(t, \omega( \sigma ), z ) \| \omega' ( \sigma ) \|_{H^1}  \, \di \sigma \qquad \text{for every $s \in [0, L(u) ]$}\,.
\end{equation}
\end{corollary}

\begin{proof}
We notice that the function $\rho \colon [0,L(u)]\to [0,+\infty]$ is well defined since~$\lambda$ is monotone increasing thanks to Theorem~\ref{t.2,5}, and hence invertible. As a consequence, also~$\omega\colon [0,L(u)] \to \U$ is well defined, continuous, with $\omega(0)= u^{0}$ and $\omega(L(u))= \overline{u}$. Moreover, by Theorem~\ref{t.2,5}~$\rho$ is Lipschitz continuous in~$[0,s]$ for every $s< L(u)$. Thanks to Theorem~\ref{t.2}, we have that $\omega \in W^{1,\infty}([0,s] ; H^{1}(\Om;\R^{2}))$ with $\| \omega' ( \sigma ) \|_{H^{1}} = 1$ for a.e.~$\sigma\in [0, s]$ and every~$s\in[0,L(u))$. Hence, we deduce that $\omega\in W^{1,\infty}([0,L(u)]; H^{1}(\Om;\R^{2}))$. Furthermore, by~$(b)$-$(d)$ of Theorem~\ref{t.2} we know that 
\begin{displaymath}
\F(t,u(\ell),z)=\F(t,u^0,z) - \int_{0}^{\ell} |\partial_{u}\F| ( t, u(l), z ) \| u' (l) \|_{H^{1}}\,\di l \qquad\text{for every $\ell \in [0,+\infty]$}\,.
\end{displaymath}
By the change of variable $l = \rho(\sigma)$ for~$\sigma \in [0,L(u)]$ we deduce~\eqref{e.enbal}. \qed 
\end{proof}

\begin{remark}\label{r.3}
In the notation of Corollary~\ref{c.gfu}, we notice that, as a consequence of Theorem~\ref{t.2,5}, $L(u)\leq \overline{C} \|u^{0}-\overline{u}\|_{H^{1}}$.
\end{remark}

We now prove a continuity property of the gradient flows w.r.t.~the data.

\begin{proposition}\label{p.continuitygradflow}
Let $(t_{m}, u^{0}_{m}, z_{m}) \in [0,T] \times \U \times \Z$ be such that $t_{m} \to t_{\infty}$ in~$[0,T]$, $u^{0}_{m} \to u^{0}_{\infty}$ in~$H^1(\Omega ; \R^2)$, and $z_{m}\rightharpoonup z_{\infty}$ weakly in~$H^{1}(\Om)$. Let $u_{m}, u_{\infty}\colon [0,+\infty)\to \U$ be the gradient flows computed in Theorem~\ref{t.2} with initial data $u_{m}(0)=u^{0}_{m}$ and $u_{\infty}(0)= u_{\infty}^{0}$ and parameters~$(t_{m},z_{m})$ and~$(t_{\infty},z_{\infty})$, respectively. 

Then,~$u_{m}$ converges strongly to~$u_{\infty}$ uniformly in $[0,+\infty)$, i.e.,~in $C( [0,+\infty);  H^{1}(\Om;\R^{2}))$. Moreover, if~$l_{m}\to +\infty$ as $m\to\infty$ and $\overline{u}_{\infty}:= \argmin\, \{ \F(t_{\infty}, u, z_{\infty} ):\, u\in\U \}$ then $u_{m}(l_{m}) \to \overline{u}_{\infty}$ in~$H^1(\Omega ; \R^2)$.
\end{proposition}


\begin{proof}
To prove the desired convergence we want to apply~\cite[Theorem~3.16]{Brezis_73}. In the notation of~\cite{Brezis_73}, we consider the operators~$A_{m} := \nabla_{u} \F ( t_{m}, \cdot, z_{m} ) = \nabla_{u}\E( t_{m}, \cdot, z_{m} )$, and $A_{\infty} :=  \nabla_{u} \F ( t_{\infty}, \cdot, z_{\infty} ) = \nabla_{u}\E( t_{\infty}, \cdot, z_\infty )$. In view of the hypotheses on~$h$ and~$W$, the operators~$A_{m}$ and~$A_{\infty}$ defined on the Hilbert space~$\U$ (endowed with the~$H^{1}$-norm) are maximal monotone. For~$\lambda>0$ and~$w\in \U$, let us denote with~$\varphi_{m}(\lambda,w)$ the solution of
\begin{equation}\label{e.gf2}
\min\,\{\tfrac{1}{2} \|\varphi \|^{2}_{H^{1}} + \lambda \E( t_{m}, \varphi, z_{m} ) - \langle w, \varphi \rangle:\, \varphi\in \U\}\,,
\end{equation}
where $\langle \cdot , \cdot \rangle$ is the usual duality pairing in~$\U$.
By strict convexity of~$\E$ in~$\U$,~$\varphi_{m}(\lambda, w)$ is well-defined, since the solution of the minimum problem~\eqref{e.gf2} is unique. Moreover,~$\varphi_{m}( \lambda, w )$ solves the equation
\begin{equation}\label{e.gf3}
\varphi_{m}( \lambda, w) + \lambda A_{m} \varphi_{m}(\lambda, w)= w \,,
\end{equation}
so that $\varphi_{m}(\lambda, w)=(\mathrm{I} + \lambda A_{m})^{-1} w$. In the same way, we can define~$\varphi_{\infty}(\lambda, w)$ as the solution of~\eqref{e.gf2} where we replace~$(t_{m},z_{m})$ with~$(t_{\infty},z_{\infty})$. Again, we have $\varphi_{\infty}(\lambda, w) = ( \mathrm{I} + \lambda A_{\infty})^{-1} w$.

To make use of~\cite[Theorem~3.16]{Brezis_73}, we have to show that for every~$\lambda>0$ and every~$w\in \U$ the function~$\varphi_{m}(\lambda, w)$ converges to~$\varphi_{\infty}(\lambda, w)$ in~$H^{1}(\Om;\R^{2})$. Using~\eqref{e.gf2} it is easy to see that the sequence~$\varphi_{m} (\lambda, w)$ is bounded in~$H^{1}(\Om;\R^{2})$, so that, up to a subsequence, we may assume that~$\varphi_{m}( \lambda, w ) \rightharpoonup \overline{\varphi}$ weakly in~$H^{1}(\Om;\R^{2})$ for some~$\overline{\varphi}\in \U$. We now show that $\overline{\varphi} = \varphi_{\infty}(\lambda, w )$. Indeed, by~\eqref{e.gf2} and by Lemma \ref{l.lscFE} for every $\varphi \in \U$ we have that
\begin{equation}\label{e.gf4}
\begin{split}
\tfrac{1}{2} \| \overline{\varphi} \|_{H^{1}}^{2} &+ \lambda \E(t_{\infty}, \overline{\varphi}, z_{\infty})  - \langle w,  \overline{\varphi}\rangle \\
&  \leq\liminf_{m\to\infty}\,\tfrac{1}{2} \|\varphi_{m} (\lambda, w) \|^{2}_{H^{1}} + \lambda \E( t_{m}, \varphi_{m} (\lambda, w), z_{m} ) - \langle w, \varphi_{m} (\lambda, w) \rangle\\
&\leq \limsup_{m\to\infty}\, \tfrac{1}{2} \|\varphi_{m} (\lambda, w) \|^{2}_{H^{1}} + \lambda \E( t_{m}, \varphi_{m} (\lambda, w), z_{m} ) - \langle w, \varphi_{m} (\lambda, w) \rangle\\
& \leq \limsup_{m\to\infty}\, \tfrac{1}{2} \|\varphi \|^{2}_{H^{1}} + \lambda \E( t_{m}, \varphi, z_{m} ) - \langle w, \varphi \rangle  =\tfrac{1}{2} \|\varphi \|^{2}_{H^{1}} + \lambda \E( t_{\infty}, \varphi, z_{\infty} ) - \langle w, \varphi \rangle\,,
\end{split}
\end{equation}
which implies that $\overline{\varphi}= \varphi_{\infty}(\lambda, w)$ by uniqueness of minimizer. Repeating the argument of~\eqref{e.gf4} with $\varphi= \overline{\varphi}$, we also deduce that
\begin{equation}\label{e.gf5}
\begin{split}
\lim_{m\to\infty}\,  \tfrac{1}{2} \|\varphi_{m} (\lambda, w) \|^{2}_{H^{1}} & + \lambda \E( t_{m}, \varphi_{m} (\lambda, w), z_{m} ) - \langle w, \varphi_{m} (\lambda, w) \rangle \\
& =  \tfrac{1}{2} \|\varphi_{\infty} (\lambda, w) \|^{2}_{H^{1}} + \lambda \E( t_{\infty}, \varphi_{\infty} (\lambda, w), z_{\infty} ) - \langle w, \varphi_{\infty} (\lambda, w) \rangle \,.
\end{split}
\end{equation}
As a consequene of~$(a)$ in Lemma~\ref{l.HMWw}, 
 there exists a constant~$\beta=\beta(\lambda)>0$ such that for every~$t,s \in[0,T]$, every~$z\in\Z$, and every~$u_{1},u_{2}\in\U$ we have
 \begin{displaymath}
 \begin{split}
  \lambda \E(t, u_{1}, z )  - \lambda \E(s , u_{2}, z ) & \geq \lambda \left\langle \nabla_{u} \E(s , u_{2}, z) , u_{1} +g(t)- u_{2} - g(s) \right\rangle \\
& \qquad + \beta \|\strain (u_{1} + g(t)) - \strain( u_{2} + g(s)) \|_{L^2}^{2}\,.
 \end{split}
 \end{displaymath}
 Therefore, for every~$m$ we can write
\begin{equation}\label{e.gf6}
\begin{split}
\tfrac{1}{2} \| \varphi_{\infty}&( \lambda, w ) \|_{H^{1}}^{2} +  \lambda \E(t_{\infty}, \varphi_{\infty} (\lambda, w), z_{m} )  \ - \ \tfrac{1}{2} \| \varphi_{m}( \lambda, w ) \|_{H^{1}}^{2} - \lambda \E(t_{m}, \varphi_{m} (\lambda, w), z_{m} ) \\
&  \geq \langle \varphi_{m}( \lambda, w), \varphi_{\infty}(\lambda, w) - \varphi_{m}(\lambda, w)\rangle + \lambda \left\langle \nabla_{u} \E(t_{m}, \varphi_{m}(\lambda, w ), z_{m}) ,\varphi_{\infty} ( \lambda, w) - \varphi_{m} (\lambda, w) \right\rangle \\
& \qquad + \beta \|\strain (\varphi_{m}(\lambda, w) + g(t_{m})) - \strain( \varphi_{\infty}(\lambda, w) + g(t_{\infty})) \|_{L^2}^{2} \\
&=  \beta \|\strain (\varphi_{m}(\lambda, w) + g(t_{m})) - \strain( \varphi_{\infty}(\lambda, w) + g(t_{\infty})) \|_{L^2}^{2} \,,
\end{split}
\end{equation}
where, in the last inequality, we have used the minimality of~$\varphi_{m}(\lambda, w)$. We now pass to the limit in~\eqref{e.gf6} as $m\to\infty$. In view of~\eqref{e.gf5} and of the convergences of~$t_{m}$ and~$z_{m}$, the left-hand side of~\eqref{e.gf6} tends to~$0$, so that $\strain (\varphi_{m}(\lambda, w) + g(t_{m}))$ converges to $ \strain( \varphi_{\infty}(\lambda, w) + g(t_{\infty}))$ in~$L^{2}(\Om;\mathbb{M}^{2}_{s})$. By Korn inequality, we get that $\varphi_{m}(\lambda, w) \to \varphi_{\infty} (\lambda, w)$ in~$H^1(\Omega; \R^2)$.

Therefore, we are in a position to apply~\cite[Theorem~3.16]{Brezis_73}, from which we deduce the convergence of~$u_{m}$ to~$u_{\infty}$ uniformly in~$H^1(\Omega; \R^2)$ on compact subsets of~$[0,+\infty)$. To show the convergence in $L^{\infty}([0,+\infty); H^{1}(\Om;\R^{2}))$ it remains to control what happens in a neighborhood of~$\infty$. Let us fix~$\delta>0$. By~\eqref{e.40}, for every $l\in[0,+\infty)$ and for every $m\in\mathbb{N}\cup\{\infty\}$ we have
\begin{equation}\label{e.11.38}
 \| u_{m}(l) - \overline{u}_{m} \|_{H^{1}} \leq e^{-c l} \|u^{0}_{m} - \overline{u}_{m} \|_{H^{1}}\,,
\end{equation}
where the constant~$c>0$ does not depend on~$m$. By hypothesis $u^{0}_{m} \to u^{0}_{\infty}$, while applying Proposition~\ref{p.6} we get that $\overline{u}_{m} \to \overline{u}_{\infty}$ in~$H^1(\Omega; \R^2)$ as $m\to\infty$, which implies that $u^{0}_{m} - \overline{u}_{m}$ is bounded in~$H^1(\Omega; \R^2)$. Hence, by~\eqref{e.11.38} there exists $\ell_{\delta}\in[0,+\infty)$ such that $\| u_{m}(l) - \overline{u}_{m} \|_{H^{1}} \leq \tfrac{\delta}{4}$ for every $l \geq \ell_{\delta}$ and every~$m\in\mathbb{N}\cup\{\infty\}$. By triangle inequality, for every $l \geq \ell_{\delta}$ we have
\begin{displaymath}
\| u_{m}(l) - u_{\infty}(l) \|_{H^{1}} \leq \|u_{m} (l) - \overline{u}_{m} \|_{H^{1}} + \|\overline{u}_{m} - \overline{u}_{\infty} \|_{H^{1}} + \|\overline{u}_{\infty} - u_{\infty} (l) \|_{H^{1}} \leq \tfrac{\delta}{2} + \|\overline{u}_{m} - \overline{u}_{\infty} \|_{H^{1}}\,,
\end{displaymath}
from which we deduce that there exists~$m_{\delta}\in\mathbb{N}$ such that
\begin{displaymath}
\| u_{m}(l) - u_{\infty}(l) \|_{H^{1}}\leq \delta \qquad \text{for every~$m\geq m_{\delta}$ and every~$l\geq \ell_{\delta}$}\,.
\end{displaymath}
Combining the previous estimate with the uniform convergence of~$u_{m}$ to~$u_{\infty}$ on compact subsets of~$[0,+\infty)$ we conclude that $u_{m}\to u_{\infty}$ uniformly in $[0,+\infty)$.

Finally, the last part of the thesis follows from~\eqref{e.11.38} and from the convergence of~$\overline{u}_{m}$ to~$\overline{u}_{\infty}$ in~$H^{1}(\Om;\R^{2})$. \qed
\end{proof}

As a corollary of Proposition~\ref{p.continuitygradflow}, we deduce a convergence result for the reparametrized functions defined in Corollary~\ref{c.gfu}.

\begin{corollary}\label{c.5}
Let $(t_{m}, u^{0}_{m}, z_{m})$, $(t_{\infty}, u^{0}_{\infty}, z_{\infty})$,~$u_{m}$,~$u_{\infty}$, and $\overline{u}_{\infty}$ be as in Proposition~\ref{p.continuitygradflow}. Let $L(u_{m})$, $\omega_{m}$, and~$\rho_{m}$ be as in Corollary~\ref{c.gfu}. Then, for every $s_{m}\in [0,L(u_{m})]$ such that $\rho_{m}(s_{m}) \to \overline{\rho} \in[0,+\infty]$, we have that $\omega_{m}(s_{m})\to u_{\infty}(\overline{\rho})$ in~$H^1(\Omega; \R^2)$, where we intend $u_{\infty}(+\infty) = \overline{u}_{\infty}$.
\end{corollary}

\begin{proof}
Let~$\rho_{m}$ and $\rho_{\infty}$ be as in Corollary~\ref{c.gfu}. For every~$m$ let $\ell_{m}:=\rho_{m}(s_{m})$, so that $\omega_{m}(s_{m})=u_{m}\circ \rho_{m} (s_{m}) = u_{m}(\ell_{m})$. By assumption $\ell_{m}\to \overline{\rho}$. Hence, the thesis follows by applying Proposition~\ref{p.continuitygradflow}. \qed
\end{proof}

\subsection[A unilateral $L^2$-grandient flow for the phase field]{A unilateral \boldmath{$L^2$}-grandient flow for the phase field}


A result similar to Theorem~\ref{t.2} holds also for the phase field~$z$ when we consider the time~$t\in[0,T]$ and the displacement $u\in\U$ as fixed parameters. In this case, however, we will need a unilateral gradient flow in the topology of~$L^{2}(\Om)$, mainly to take care of the irreversibility condition imposed on the phase field. For this reason, the following result, similar in nature to Theorem~\ref{t.2}, needs to be proven.

\begin{theorem}\label{t.3}
Let $\tilde{p}\in(2,+\infty)$ be as in Lemma~\ref{l.HMWTh}, and let $(t, u, z^{0}) \in [0,T] \times W^{1,\tilde{p}}(\Om;\R^{2}) \times \Z$ with $0\leq z^0\leq 1$. Then, there exists an evolution $z\colon[0,+\infty)\to\Z$ satisfying the following conditions:
\begin{itemize}

\item[$(a)$] $z\in L^{\infty}([0,+\infty);H^{1}(\Om))$ and $z'\in L^{1}([0,+\infty); L^{2}(\Om))\cap L^{2}([0,+\infty); L^{2}(\Om))$;

\item[$(b)$] $z(0) = z^0$, $z$ is non-increasing, $0 \le z \le 1$,  $\|z'(l)\|_{L^2}=|\partial_{z}^{-}\F|(t,u,z(l))$ for a.e.~$l\in[0,+\infty)$; 

\item[$(c)$] for every $\ell \in[0,+\infty)$ it holds
\begin{equation}\label{e.42}
\F(t,u,z(\ell))=\F(t,u,z^0)-\tfrac{1}{2}\int_{0}^{\ell} |\partial_{z}^{-}\F|^{2}(t,u,z(l))+\|z'(l)\|_{L^2}^{2}\,\di l\,;
\end{equation}

\item[$(d)$] $z(l)$ converges to $\overline{z}$ strongly in $H^1 (\Omega)$ as $l \to+\infty$, where $\overline{z}=\argmin\,\{\F(t,u,z):\, z\in\Z,\, z\leq z^0\}$. Moreover,
\begin{align}
\F(t,u,\overline{z})&=\F(t,u,z^0)-\tfrac{1}{2}\int_{0}^{+\infty}|\partial_{z}^{-}\F|^{2}(t,u,z(l))+\|z'(l)\|_{L^2}^{2}\,\di l \,; \label{e.44}
\end{align}

\item[$(e)$] there exists $\overline{\ell}\in [0, +\infty]$ such that $\| z'(\ell) \|_{L^2} \neq 0$ for a.e.~$\ell <  \overline{\ell}$ and $\| z'(\ell)\|_{L^2}=0$ for a.e.~$\ell \geq \overline{\ell}$;

\item[$(f)$] there exists a constant $\overline{C}>0$ such that
\begin{equation}\label{e.lengthz}
\int_{0}^{+\infty}\|z'(l)\|_{L^2}\,\di l\leq \overline{C}(1+\|u+g(t)\|_{W^{1,\tilde{p}}}) \|z^0-\overline{z}\|_{H^{1}}\,.
\end{equation}
\end{itemize}
\end{theorem}

\begin{proof}
We set $\overline{z}=\argmin\,\{\F(t,u,z):\,z\in\Z,\,z\leq z^0\}$. 
In order to construct a gradient flow $l \mapsto z(l)$ as in the statement of the theorem, we proceed by time-discretization. For $k \in \mathbb{N}\setminus\{0\}$, and every $i\in\mathbb{N}$ we set $l^{k}_{i}\coloneq i / k$ and we solve iteratively the minimum problem
\begin{equation} \label{e.1306}
\min\,\{\F(t,u,z)+\tfrac{k}{2}\|z-z^{k}_{i}\|_{L^2}^{2}:\,z\in \Z,\, z\leq z^{k}_{i}\}\,,
\end{equation}
where $z^{k}_{0}\coloneq z^0$. First, let us prove that $z^{k}_{i}\geq\overline{z}$ for every $k,i$. We proceed by induction w.r.t.~$i$. A similar proof is contained in~\cite{NegriKimura}. By definition $z^0 \ge \overline{z}$. Let $z^k_{i} \ge \overline{z}$. Let us introduce the sets $\Omega^+ = \{ z^k_{i+1} \ge \overline{z} \}$, $\Omega^- = \{  z^k_{i+1} < \overline{z} \}$, 
and the corresponding energies
$$
	\F_{|\Omega^\pm} ( t, u , z ) = \int_{\Omega^\pm} W\big(z,\strain(u+g(t))\big)\,\di x   +   \int_{\Omega^\pm} | \nabla z |^2 + f(z) \, \di x . \\
$$
Let 
$$	\hat{z} := \max \{   z^k_{i+1}  , \overline{z} \}	= \begin{cases}   \overline{z}   & \text{ in $\Omega^-$} , \\
 z^k_{i+1} & \text{ in $\Omega^+$,} \end{cases} 
 \qquad 
 \check{z} := \min \{   z^k_{i+1}  , \overline{z} \}	= \begin{cases}  z^k_{i+1}     & \text{ in $\Omega^-$} , \\
  \overline{z} & \text{ in $\Omega^+$.} \end{cases} 
$$
By minimality of $\overline{z}$ we can write 
$$
  \F ( t , u , \check{z} ) = \F_{|\Omega^+} ( t , u , \overline{z} )  + \F_{|\Omega^-} ( t , u , z^k_{i+1} ) \ge \F ( t, u , \overline{z}) = \F_{|\Omega^+} ( t, u , \overline{z}) + \F_{|\Omega^-} ( t, u , \overline{z}) \, ,
$$
from which we deduce that  $\F_{|\Omega^-} ( t , u , z^k_{i+1} ) \ge \F_{|\Omega^-} ( t, u , \overline{z})$. Since $z^k_{i+1} < \overline{z} \le z^k_i$ in the set $\Omega^-$ we can write
\begin{align*}
	 \F ( t , u , \hat{z} ) + \tfrac{k}{2}  \| \hat{z} - z^k_i \|^2_{L^{2}} & =  \F_{|\Omega^+} ( t , u , z^k_{i+1} ) + \F_{|\Omega^-} ( t , u , \overline{z} ) + \tfrac{k}{2}  \| \hat{z} - z^k_i \|^2_{L^2} \\
	& \le \F ( t , u , z^k_{i+1} )  + \tfrac{k}{2}  \| z^k_{i+1} - z^k_i \|^2_{L^2} .
\end{align*}
Hence $\hat{z}$ is the minimizer of \eqref{e.1306}. By uniqueness it implies that $z^k_{i+1} = \hat{z} \ge \overline{z}$. 

Defining the usual piecewise affine interpolant $z^k$, we get a sequence $z^k$ bounded in $H^{1}_{\mathrm{loc}}([0,+\infty),L^{2}(\Om))$ and in $L^{\infty}([0,+\infty);H^1(\Omega))$ with $z^k (l) \ge \overline{z}$ for every $l \in [0,+\infty)$. Passing to the limit (up to subsequences) we identify  a limit function $z \in H^{1}_{\mathrm{loc}}([0,+\infty);L^{2}(\Om))\cap L^{\infty}([0,+\infty); H^{1}(\Om))$, satisfying  $z (l) \ge \overline{z}$ for every $l \in [0,+\infty)$ and 
\begin{equation}\label{e.60}
\F(t,u,z(\ell))\leq \F(t,u,z^0)-\tfrac{1}{2}\int_{0}^{\ell}|\partial_{z}^{-}\F|^{2}(t,u,z(l))+\|z'(l)\|_{L^2}^{2}\,\di l 
\end{equation}
for every $\ell \in[0,+\infty)$. With the usual Riemann sum argument we can show that in~\eqref{e.60} the equality holds, see e.g.~\cite{Negri_ACV}, we deduce that
\begin{displaymath}
\F(t,u,z(\ell))\geq \F(t,u,z^0)-\tfrac{1}{2}\int_{0}^{\ell}|\partial_{z}^{-} \F | ( t, u, z(l) ) \|z'(l)\|_{L^2} \,\di l \,,
\end{displaymath}
which implies, by Young inequality, the energy equality in~\eqref{e.42} and the following identities, valid for a.e.~$l\in[0,+\infty)$:
\begin{equation}\label{e.66}
\|z'(l)\|_{L^2}=|\partial_{z}^{-}\F|(t,u,z(l)) \qquad \text{and} \qquad \frac{\di}{\di l}\F(t,u,z(l))=-|\partial_{z}^{-}\F|(t,u,z(l)) \|z'(l)\|_{L^2}\,.
\end{equation}

Since $l \mapsto z(l)$ is decreasing, there exists a limit $\widetilde{z}\in\Z$, as $l \to +\infty$, weakly in $H^{1}(\Om)$ and strongly in~$L^{2}(\Om)$. In particular, $\widetilde{z}\geq\overline{z}$; we want to show that equality holds. To this aim, passing to the liminf as $l \to+\infty$ in~\eqref{e.42} we easily obtain that
\begin{equation}\label{e.6.16}
\F(t,u,\widetilde{z})\leq\F(t,u,z^0)-\tfrac{1}{2}\int_{0}^{+\infty}|\partial_{z}^{-}\F|^{2}(t,u,z(l))+\|z'(l)\|_{L^2}^{2}\,\di l\,.
\end{equation}
Coupling~\eqref{e.66} with~\eqref{e.6.16} we get that $z'\in L^{2}([0,+\infty); L^{2}(\Om))$. Moreover, being~$\F\geq0$, from~\eqref{e.6.16} we obtain that
\begin{equation}\label{e.62}
\liminf_{l \to+\infty}\,|\partial^{-}_{z}\F|(t,u,z(l))=0\,.
\end{equation}
By Lemma~\ref{l.2} we have that $|\partial_{z}^{-}\F|(t,u,\widetilde{z})=0$, that is,~$\widetilde{z}$ is a solution of
\begin{equation}\label{e.61}
\min\,\{\F(t,u,z):\,z\in\Z, \, z\leq\widetilde{z}\}\,.
\end{equation}
Since $\overline{z}\leq\widetilde{z}$, by uniqueness of solution of~\eqref{e.61} we get that $\overline{z}=\widetilde{z}$.

Now, we show that $z(l)\to\overline{z}$ strongly in~$H^1(\Omega)$. Indeed, for every $l \in[0,+\infty)$ we have, by convexity of $z\mapsto\F(t,u,z)$,
\begin{equation}\label{e.64}
0 \le \F(t,u,z(l))-\F(t,u,\overline{z})\leq -\partial_{z}\F(t,u,z(l))[\overline{z}-z(l)]\leq|\partial_{z}^{-}\F|(t,u,z(l))\|z(l)-\overline{z}\|_{L^2}\,,
\end{equation}
where, in the last inequality, we have used the characterization~\eqref{e.slope2} of the slope w.r.t.~$z$. By~\eqref{e.62},  we know that along a suitable subsequence $l_{j}\to+\infty$ we have $|\partial_{z}^{-}\F|(t,u,z(l_{j}))\to0$, so that $\F(t,u,z(l_{j}))\to\F(t,u,\overline{z})$. By monotonicity of $l\mapsto \F(t,u,z(l))$, we therefore get that $\F(t,u,z(l)) \to \F(t,u,\overline{z})$ as $l\to+\infty$. Hence, we deduce that $\| \nabla z(l) \|_{L^2} \to \| \nabla \overline{z} \|_{L^2}$, which in turn implies the convergence of~$z(l)$ to~$\overline{z}$ in~$H^1(\Omega)$. 

In order to prove~$(e)$, we define $\overline{\ell}:=\inf\,\{l \geq 0:\, \F(t,u,z(l))= \F(t,u, \overline{z})\}$. If $\overline{\ell} < +\infty$ then, being $\overline{z}$ the unique minimizer of $\{\F(t,u,z):\, z\in\Z,\, z\leq z^0\}$, we have $z(l) = \overline{z}$ for every $l \ge \ell$. 
In general, for a.e.~$l < \overline{\ell}$, we claim that $|\partial_{z}^{-} \F|(t,u, z(l))\neq 0$. By contradiction, if $|\partial_{z}^{-} \F|(t,u, z(l)) = 0$, then $z(l) = \argmin\,\{\F(t,u,z):\, z\in\Z,\, z\leq z(l)\}$. Since $\overline{z}\leq z(l)$, we would get that $z(l)=\overline{z}$, which contradicts the assumption $l< \overline{\ell}$. Therefore, $|\partial_{z}^{-}\F| (t, u, z(l)) \neq0$ for a.e.~$l<\overline{\ell}$. This implies, together with~\eqref{e.66}, that $\| z'(l) \|_{L^2}\neq 0$ for a.e.~$l<\overline{\ell}$.

The proof of property~$(f)$ is similar to the proof of~\eqref{e.lengthu} in Theorem~\ref{t.2,5}, but we have to take care of the monotonicity of~$l\mapsto z(l)$ and of the different norm of the gradient flow. By strong convexity, see~\eqref{e.strconv-z}, there exists a positive constant~$c$ independent of~$z$,~$u$, and~$t$, such that
\begin{displaymath}
\begin{split}
c\| \overline{z} - z(l) \|_{H^{1}}^{2}&\leq (\partial_{z}\F(t,u,\overline{z})-\partial_{z}\F(t,u,z(l)))[\overline{z}-z(l)] \le -\partial_{z}\F(t,u,z(l))[\overline{z}-z(l)]\\
&\leq|\partial_{z}^{-}\F|(t,u,z(l))\| \overline{z} - z(l) \|_{L^2} \leq|\partial_{z}^{-}\F|(t,u,z(l))\| \overline{z} - z(l) \|_{H^{1}}\,,
\end{split}
\end{displaymath}
which implies 
\begin{equation}\label{e.63}
\| \overline{z} - z(l) \|_{H^{1}}\leq C|\partial_{z}^{-}\F|(t,u,z(l))
\end{equation}
for some positive constant~$C$. Combining~\eqref{e.63} with~\eqref{e.64} we get
\begin{equation}\label{e.65}
(\F(t,u,z(l))-\F(t,u,\overline{z}))^{1/2}\leq C |\partial_{z}^{-}\F|(t,u,z(l))\,.
\end{equation}

Exploiting~\eqref{e.65}, we can now perform a \L ojasiewicz argument: by~\eqref{e.66},~\eqref{e.65}, and by the monotonicity and absolute continuity of $l\mapsto \F(t,u, z(l))$, for a.e.~$l \in[0, \overline{\ell})$ we have
\begin{equation}\label{e.67}
\begin{split}
-2\frac{\di}{\di l}\big( \F(t,u,z(l)) - \F(t,u,\overline{z})\big)^{\frac12} & =\big( \F(t,u,z(l)) - \F(t,u,\overline{z})\big)^{-\frac12} |\partial_{z}^{-}\F|(t,u,z(l))\|z'(l)\|_{L^2}\\
& \geq C \|z'(l)\|_{L^2}\,.
\end{split}
\end{equation}
Therefore, inequality~\eqref{e.67} implies that for every $\ell\in[0,\overline{\ell})$
\begin{displaymath}
\int_{0}^{ \ell }\|z'(l)\|_{L^2}\,\di l \leq -2 C \int_{0}^{ \ell }\frac{\di}{\di l}\big( \F(t,u,z(l)) - \F(t,u,\overline{z})\big)^{\frac12}\,\di l \leq 2C \big( \F(t,u,z^0) - \F(t,u,\overline{z})\big)^{\frac12}\,.
\end{displaymath}
In the limit as~$\ell\to+\infty$, from the previous inequality we get
\begin{equation}\label{e.68}
\int_{0}^{ +\infty }\|z'(l)\|_{L^2}\,\di l \leq  2C \big( \F(t,u,z^0) - \F(t,u,\overline{z})\big)^{\frac12}\,.
\end{equation}
By convexity, we have that 
\begin{align}
 \big( \F(t,u,z^0) - \F(t,u,\overline{z})\big)^{\frac12} & \leq \big( -\partial_{z}\F(t,u,z^0)[\overline{z}-z^0]\big)^{\frac12} \nonumber \\ 
	& \le \big( \partial_{z}\F(t,u,\overline{z})[\overline{z}-z^0] -\partial_{z}\F(t,u,z^0)[\overline{z}-z^0]\big)^{\frac12}\,,\label{e.69}
\end{align}
where, in the last inequality, we have used the fact that $\partial_{z}\F(t,u,\overline{z})[z^0 - \overline{z}] = 0$, by minimality of~$\overline{z}$.

\separe

The right hand side of~\eqref{e.69} is
\begin{displaymath}
\begin{split}
\big(\partial_{z}\F(t,u,\overline{z})-\partial_{z}\F(t,u,z^0)\big)[\overline{z}-z^0]&=\int_{\Om}(h'(\overline{z})-h'(z^0))(\overline{z}-z^0)(\mu|\strain_{d}(u+g(t))|^{2}+\kappa|\strain_{v}^{+}(u+g(t))|^{2})\,\di x\\
&\qquad+\int_{\Om} \big(f' (\overline{z} ) - f'( z^{0} ) \big)( \overline{z} - z^{0} )\,\di x + \|\nabla {\overline{z}} - \nabla {z^0} \|_{L^2}^{2}\,.
\end{split}
\end{displaymath}
Applying H\"older inequality to the first term of the right-hand side of previous inequality with $\tfrac{1}{\nu} + \tfrac{2}{\tilde{p}}=1$ 
and recalling that $0 \le z(l) \le z^0 \le 1$, and that $h,f \in C^{1,1}([0,1])$, we deduce that
\begin{equation}\label{e.71}
\begin{split}
\Big|\big(\partial_{z}\F&(t,u,\overline{z})-\partial_{z}\F(t,u,z^0)\big)[\overline{z}-z^0]\Big|\\
& \leq\int_{\Om}|h'(\overline{z})-h'(z^0)||\overline{z}-z^0|(\mu|\strain_{d}(u+g(t))|^{2}+\kappa|\strain_{v}^{+}(u+g(t))|^{2})\,\di x + C\| \overline{z} - z^{0} \|_{H^{1}}^{2}\\
& \leq C \int_{\Om} | \overline{z} - z^0 |^{2} (\mu |\strain_{d}(u+g(t))|^{2} + \kappa |\strain_{v}^{+}(u+g(t))|^{2}) \,\di x + C\| \overline{z} - z^{0} \|_{H^{1}}^{2} \\
& \leq C \|\overline{z}-z^0\|_{2\nu}^{2} \|u+g(t)\|^{2}_{W^{1,\tilde{p}}} + C\| \overline{z} - z^{0} \|_{H^{1}}^{2} \leq C( 1 + \|u+g(t)\|^{2}_{W^{1,\tilde{p}}} )  \|\overline{z}-z^0\|_{H^{1}}^{2}  \,.
\end{split}
\end{equation} 
Thus, combining inequalities~\eqref{e.68}-\eqref{e.71} we get
\begin{displaymath}
\int_{0}^{+\infty}\|z'(l)\|_{L^2}\,\di s\leq C(1+\|u+g(t)\|^{2}_{W^{1,\tilde{p}}})^{\frac12}\|\overline{z}-z^0\|_{H^{1}}\,.
\end{displaymath}
This concludes the proof of the theorem. \qed
\end{proof}

As in Corollary~\ref{c.gfu}, we define here a reparametrization of~$l\in[0,+\infty)$ which makes the gradient flow of Theorem~\ref{t.3} $1$-Lipschitz. Again, this reparametrization will be used in Section~\ref{s.prooft1}.

\begin{corollary}\label{c.gfz}
Let $\tilde{p}\in(2,+\infty)$ be as in Lemma~\ref{l.HMWTh},  let $(t, u, z^{0}) \in [0,T]\times W^{1,\tilde{p}} (\Omega; \mathbb{R}^2) \times \Z$ with $0\leq z^{0} \leq 1$, let $\overline{z}:=\argmin\,\{\F(t,u,z):\, z\in\Z, \, z\leq z^{0}\}$, and let~$z$ be the gradient flow computed in Theorem~\ref{t.3} with initial condition~$z^{0}$ and parameters~$t$ and~$u$. Given~$\overline{\ell}\in[0,+\infty]$ as in~$(e)$ of Theorem~\ref{t.3}, let us set
\begin{displaymath}
L(z) := \int_{0}^{\overline{\ell}} \| z'(l) \|_{L^2}\,\di l \qquad \text{and} \qquad \lambda(\ell) := \int_{0}^{\ell} \| z'(l) \|_{L^2}\,\di l \quad \text{for $\ell \in [ 0, \overline{\ell} ]$}\,.
\end{displaymath}
Moreover, let $\rho \colon [0, L(z)]\to[0,\overline{\ell}]$ be defined by $\rho := \lambda^{-1}$.  Then, the function $\zeta:= z\circ \rho$ belongs to the space $W^{1,\infty} ([0,L(z)]; L^{2}(\Om) )$ with $\| \zeta' (s) \|_{L^2} =1$ a.e.~in $[0, L(z)]$, $\zeta(0)=z^{0}$, $\zeta(L(z)) = \overline{z}$, and
\begin{equation}\label{e.enbal2}
\F(t,u, \zeta( s) ) = \F ( t, u, z^{0} ) - \int_{0}^{s} |\partial_{z}^{-} \F| (t, u, \zeta(\sigma )) \|\zeta'(\sigma)\|_{L^{2}} \, \di \sigma \,.
\end{equation}
\end{corollary}

\begin{proof}
We notice that~$\rho=\lambda^{-1}$ is well defined in view of~$(e)$ of Theorem~\ref{t.3}. As a consequence, also $\zeta\colon[0,L(z)]\to \Z$ is well defined and satisfies $\zeta(0)=z^{0}$, $\zeta(L(z)) = \overline{z}$, and $\|\zeta'(s) \|_{L^2}=1$ for a.e.~$s\in[0,L(z)]$. By $(b)$-$(d)$ of Theorem~\ref{t.3} we have that
\begin{displaymath}
\F(t, u, z(\ell) ) = \F(t, u, z^{0}) - \int_{0}^{\ell} |\partial_{z}^{-} \F|(t, u, z(l)) \|z'(l)\|_{L^2}\,\di l\qquad\text{for $\ell\in[0,\overline{\ell}]$}\,.
\end{displaymath}
By the change of coordinate~$l= \rho(\sigma)$ for $\sigma\in[0,L(z)]$ we deduce~\eqref{e.enbal2}.\qed
\end{proof}

\begin{remark}\label{r.4}
In the notation of Corollary~\ref{c.gfz}, we notice that, as a consequence of Theorem~\ref{t.3}, 
\begin{displaymath}
L(z)\leq \overline{C}(1+\|u+g(t)\|_{W^{1,\tilde{p}}})\| z^{0} - \overline{z} \|_{H^{1}}\,.
\end{displaymath}
\end{remark}

\section{Proof of the convergence result} \label{s.prooft1}

We develop in this section the proof of Theorem~\ref{t.1}.  We follow the main structure of~\cite{KneesNegri_M3AS17}. We start with constructing a time-discrete evolution by an alternate minimization algorithm. Next we interpolate between all the steps of the scheme w.r.t.~an arc-length parameter in a suitable norm. Since the energy~$\F$ is not separately quadratic, in this context there are no intrinsic norms stemming out from the functional, as it happens in~\cite{KneesNegri_M3AS17}; in our framework, instead, it is natural to use the $H^{1}$-norm for the displacement field~$u$ and the $L^{2}$-norm for the phase field~$z$. The latter technical choice is due to the existence of a unilateral $L^{2}$-gradient flow (see Theorem~\ref{t.3}) which in turn is related to the irreversibility of~$z$ along the whole algorithm.

In Proposition~\ref{p.compactness} we prove compactness of the discrete parametrized evolutions. We characterize the limit evolution in terms of \emph{equilibrium} and \emph{energy-dissipation balance} (see~$(d)$ and~$(e)$ of Theorem~\ref{t.1}). The proof of equilibrium and of the lower energy-dissipation inequality follows from lower semicontinuity of the functional~$\F$ and of the slopes~$|\partial_{u} \F|$ and~$|\partial_{z}^{-} \F|$. The technically hard part comes with the upper energy-dissipation inequality (see Section~\ref{s.inequality}). Comparing with~\cite{KneesNegri_M3AS17}, here we can not employ a chain rule argument, since the evolution~$z$ is qualitatively the reparametrization of an $L^{2}$-gradient flow, instead of an $H^{1}$-gradient flow. For this reason, we need to exploit a Riemann sum argument (see, e.g.,~\cite{MR2186036, Negri_ACV}). In this respect, the starting point would be the summability of $|\partial_{z}^{-} \F| (t(\cdot), u(\cdot), z(\cdot))$, which does not follow from the energy estimates, since we are only able to control $|\partial_{z}^{-} \F| (t(\cdot), u(\cdot), z(\cdot)) \| z'(\cdot) \|_{L^{2}}$. Nevertheless, we can show that~$|\partial_{z}^{-}\F|(t(\cdot), u(\cdot), z(\cdot)) $ belongs to $L^1$ in the set where~$\|z' \|_{L^{2}} \neq 0$. At this point, we can apply a Riemann sum argument in an auxiliary reparametrized setting which, roughly speaking, concentrates the intervals where~$\|z'\|_{L^{2}}= 0$ to an at most countable set of points, at the price of introducing discontinuities in the displacement evolution, which, however, can be controlled a posteriori via chain rule.


\subsection{Parametrization and discrete energy estimate} \label{s.4.1}

For $k \in \mathbb{N}$, $k \neq 0$ let $\tau_{k}:=T/k$ and $t^{k}_{i}:= i\tau_{k}$ for $i=0,\ldots, k$. We define the discrete evolutions~$u^k_i$ and~$v^k_i$ (in the time nodes $t^k_i$) by induction. We set $u^{k}_{0}\coloneq u_{0}$ and $z^{k}_{0}\coloneq z_{0}$. Given $u^k_{i-1}$ and $z^k_{i-1}$ we define~$u^k_{i}$ and~$z^k_{i}$ with the aid of two auxiliary sequences~$u^{k}_{i,j}$ and~$z^{k}_{i,j}$ defined as follows:  let $u^{k}_{i,0}:= u^{k}_{i-1}$ and $z^{k}_{i,0}:=z^{k}_{i-1}$, then for every $j\in\mathbb{N}$ let 
\begin{eqnarray}
&&\displaystyle u^{k}_{i,j+1}:=\argmin\,\{\F(t^{k}_{i},u,z^{k}_{i,j}):\,u\in \U \}\,,\label{e.minu}\\[2mm]
&&\displaystyle z^{k}_{i,j+1}:=\argmin\,\{\F(t^{k}_{i},u^{k}_{i,j + 1},z):\, z\in \Z,\, z\leq z^{k}_{i,j}\}\,.\label{e.minz}
\end{eqnarray}
Note that~$0\leq z^{k}_{i,j} \leq 1$ for every $k,i,j$ and that the sequence $z^k_{i,j}$ is bounded in $H^1(\Om)$ and non-increasing w.r.t.~$j$; hence, in the limit as $j \to\infty$, $z^{k}_{i,j} \rightharpoonup z^{k}_{i}$ weakly in~$H^{1}(\Om)$ and, by Proposition~\ref{p.6}, $u^{k}_{i,j} \to u^{k}_{i}$ in~$W^{1,\beta}(\Om;\R^{2})$ for $\beta\in[2,\tilde{p})$, where $u^k_i$ solves 
\begin{equation}\label{e.minu2}
\min\,\{\F(t^{k}_{i},u,z^{k}_{i}):\, u\in\U\}\,.
\end{equation}
Moreover, being $g\in W^{1,q} ([0,T];W^{1,p}(\Om;\R^{2}))$, by Lemma~\ref{l.HMWTh} we deduce that $u^{k}_{i,j}$ is bounded in~$W^{1,\tilde{p}}(\Om;\R^{2})$, uniformly w.r.t.~$k,i,j$. By \eqref{e.minz} we know that $| \partial_z^- \F | (t^{k}_{i},u^{k}_{i,j},z^k_{i,j}) = 0$ for $j\geq 1$. Since $u^k_{i,j} \to u^k_i$ in~$W^{1,\beta}(\Om;\R^{2})$ and $z^{k}_{i,j}\rightharpoonup z^{k}_{i}$ in $H^1(\Om)$ , as a consequence of Lemma~\ref{l.2} we deduce that $| \partial_z^- \F | (t^{k}_{i},u^{k}_{i},z^k_{i}) = 0$. Hence, $z^{k}_{i}$ is the solution of
\begin{equation}\label{e.minz2}
\min\,\{\F(t^{k}_{i},u^{k}_{i},z):\, z\in\Z,\, z\leq z^{k}_{i}\}\,.
\end{equation}

\begin{remark}
In general it may happen that the alternate minimization algorithm~\eqref{e.minu}-\eqref{e.minz} converges after a finite number of iterations. This case is anyway a special case of the above scheme and it will not be treated separately. 
\end{remark}

\separe

Recalling the results and the notation of Theorem~\ref{t.2} and Corollary~\ref{c.gfu}, for every $k,i,j$ there exists an auxiliary parametrized gradient flow $\omega^{k}_{i,j}\in W^{1,\infty}([0,L( \omega^{k}_{i,j})]; H^{1}(\Om;\R^{2}))$ such that $\omega^{k}_{i,j}(0) = u^{k}_{i,j}$, $\omega^{k}_{i,j}(L( \omega^{k}_{i,j}) ) = u^{k}_{i,j+1}$, $\| (\omega^{k}_{i,j})'(s) \|_{H^{1}} = 1$ for a.e.~$s\in [0,L( \omega^{k}_{i,j})]$, and
\begin{equation}\label{e.72}
\F(t^{k}_{i}, \omega^{k}_{i,j} (s), z^{k}_{i,j} ) = \F ( t^{k}_{i}, u^{k}_{i,j}, z^{k}_{i,j} ) - \int_{0}^{s} | \partial_{u} \F | ( t^{k}_{i}, \omega^{k}_{i,j} ( \sigma ), z^{k}_{i,j} ) \| (\omega^{k}_{i,j})'(\sigma) \|_{H^{1}}\, \di \sigma 
\end{equation}
 for every $s\in[0,L(w^{k}_{i,j})]$. Moreover, in view of Theorem~\ref{t.2,5} and Remark~\ref{r.3},
 \begin{equation}\label{e.15.40}
  L( \omega^{k}_{i,j}) \leq \bar{C} \|u^{k}_{i,j} - u^{k}_{i,j+1}\|_{H^{1}}
 \end{equation}
 for some positive constant~$\bar{C}$ independent of~$k,i,j$.

 In a similar way, by Theorem~\ref{t.3} there exists an auxiliary parametrized gradient flow $\zeta^{k}_{i,j}$ belonging to $W^{1,\infty}([0, L( \zeta^{k}_{i,j})]; L^{2}(\Om) )$ such that $\zeta^{k}_{i,j}(0) = z^{k}_{i,j}$, $\zeta^{k}_{i,j}( L( \zeta^{k}_{i,j}) ) = z^{k}_{i,j+1}$, $ (\zeta^{k}_{i,j})'(s) \le0 $ and $\| (\zeta^{k}_{i,j})'(s)\|_{L^2} = 1$ for a.e.~$s\in[0, L(\zeta^{k}_{i,j})]$, and 
 \begin{equation}\label{e.73}
 	\F ( t^{k}_{i}, u^{k}_{i,j+1}, \zeta^{k}_{i,j} (s)) = \F ( t^{k}_{i}, u^{k}_{i,j+1}, z^{k}_{i,j} ) - \int_{0}^{s} | \partial_{z}^{-} \F | ( t^{k}_{i}, u^{k}_{i,j+1}, \zeta^{k}_{i,j}( \sigma )) \|(\zeta^{k}_{i,j})'(\sigma) \|_{L^{2}} \,\di \sigma 
\end{equation}
for every $s\in [0, L(\zeta^{k}_{i,j})]$. Furthermore, by Theorem~\ref{t.3} and Remark~\ref{r.4} we have
\begin{equation}\label{e.15.52}
 L(\zeta^{k}_{i,j}) \leq \bar{C} (1 + \|u^{k}_{i,j} + g(t^{k}_{i}) \|_{1,\tilde{p}} ) \|z^{k}_{i,j} - z^{k}_{i,j+1}\|_{H^{1}}\,,
\end{equation}
for some positive constant~$\bar{C}$ independent of~$k,i,j$. In view of the uniform boundedness of~$u^{k}_{i,j}$ in~$W^{1,\tilde{p}}(\Om;\R^{2})$ (see Corollary~\ref{c.3}) and of the regularity of the boundary datum~$g$, inequality~\eqref{e.15.52} can be rewritten as
\begin{equation}\label{e.16.06}
L( \zeta^{k}_{i,j}) \leq \tilde{C} \|z^{k}_{i,j} - z^{k}_{i,j+1}\|_{H^{1}}
\end{equation}
for some positive constant~$\tilde{C}$ independent of~$k,i,j$.
 


We now start showing a uniform bound on the arc-length of the alternate minimization scheme~\eqref{e.minu}-\eqref{e.minz}. This is done by estimating the term
\begin{equation}\label{e.Sk}
S_{k}\coloneq \sum_{i=1}^{k}\sum_{j=0}^{\infty} L ( \omega^{k}_{i,j}) + L ( \zeta^{k}_{i,j})
\end{equation}
uniformly w.r.t.~$k\in\mathbb{N}$.

\begin{proposition}\label{p.finitelength}
There exists $\overline{S}\in(0,+\infty)$ such that $S_{k}\leq \overline{S}$ for every index~$k$. 
\end{proposition}

\begin{proof} In this proof we denote with~$C$ a generic positive constant, which could change from line to line.

Thanks to \eqref{e.15.40} and \eqref{e.16.06} 
we deduce that
\begin{equation}\label{e.74}
S_{k}\leq C\sum_{i=1}^{k}\sum_{j=0}^{\infty}\big(\|z^{k}_{i,j}-z^{k}_{i,j+1}\|_{H^1}+\|u^{k}_{i,j}-u^{k}_{i,j+1}\|_{H^{1}}\big)\,,
\end{equation}
for~$C$ independent of~$k$. Therefore, it is sufficient to prove that the right-hand side of~\eqref{e.74} is uniformly bounded.

For $j=0$, applying Corollary~\ref{c.3}, recalling that $z^k_{i,0}=z^k_{i-1}$ and that the boundary datum varies, 
we have that
\begin{displaymath}
\|u^{k}_{i,1}-u^{k}_{i,0}\|_{H^{1}} \leq C \|g(t^{k}_{i})-g(t^{k}_{i-1})\|_{W^{1,p}} \,.
\end{displaymath}
Moreover, by Proposition~\ref{p.5},
\begin{displaymath}
\|z^{k}_{i,1}-z^{k}_{i,0}\|_{H^{1}} \leq C( \|u^{k}_{i,1}-u^{k}_{i,0}\|_{H^{1}} + \|g(t^{k}_{i})-g(t^{k}_{i-1})\|_{H^{1}} ) \leq C \|g(t^{k}_{i})-g(t^{k}_{i-1})\|_{W^{1,p}}\,.
\end{displaymath}

For $j\geq 1$, in view of Proposition~\ref{p.5}, we have
\begin{equation}\label{e.2.07}
\|z^{k}_{i,j}-z^{k}_{i,j+1}\|_{H^{1}} \leq C \|u^{k}_{i,j}-u^{k}_{i,j+1}\|_{H^{1}}\,.
\end{equation}
By~\eqref{e.25} in Corollary~\ref{c.3}, recalling that for fixed $k$ the boundary datum does not change, we can continue in~\eqref{e.2.07} with
\begin{equation}\label{e.2.09}
\|z^{k}_{i,j}-z^{k}_{i,j+1}\|_{H^{1}} \leq C \|z^{k}_{i,j}-z^{k}_{i,j-1}\|_{L^\nu} ,
\end{equation}
for some exponent $\nu \gg 1$. 
The rest of the proof works as in~\cite[Theorem~4.1]{KneesNegri_M3AS17}. \qed
\end{proof}

At this point we are ready to define a new parametrization of the graph of the evolution in terms of the arc-length of the curves $\omega^k_{i,j}$ and $\zeta^k_{i,j}$, connecting all the intermediate steps of the alternate minimization scheme \eqref{e.minu}-\eqref{e.minz}.

We set~$s^{k}_{0}\coloneq 0$, $t_{k}(0)\coloneq 0$, $u_{k}(0)\coloneq u_{0}$, and $z_{k}(0)\coloneq z_{0}$. For every $i\geq 1$, assume to know $s^{k}_{i-1}$ and let us construct~$s^{k}_{i}$. We define $s^{k}_{i,-1}\coloneq s^{k}_{i-1}$ and, for $j\geq0$,
\begin{equation}\label{e.esse}
s^{k}_{i,0}\coloneq s^{k}_{i,-1}+\tau_{k}\,,\qquad s^{k}_{i,j+\frac{1}{2}}\coloneq s^{k}_{i,j}+L ( \omega^{k}_{i,j}) \,, \qquad s^{k}_{i,j+1} \coloneq s^{k}_{i,j+\frac{1}{2}}+L(\zeta^{k}_{i,j}) \,.
\end{equation}
In view of Proposition~\ref{p.finitelength}, we have that there exists a finite limit~$s^{k}_{i}$ of the sequence~$s^{k}_{i,j}$ as $j\to \infty$. For every $s\in[s^{k}_{i,-1},s^{k}_{i,0}]$ we define
\begin{equation}\label{e.interpolantt}
 t_{k}(s) \coloneq t^{k}_{i-1}+ (s - s^{k}_{i,-1})\,,\qquad u_{k}(s)\coloneq u^{k}_{i-1}\,,\qquad z_{k}(s)\coloneq z^{k}_{i-1}\,.
\end{equation}
For $j\geq 0$ and $s\in[s^{k}_{i,j}, s^{k}_{i,j+\frac{1}{2}}]$ we set
\begin{eqnarray} 
&& t_{k}(s)\coloneq t^{k}_{i} = t_{k}(s^{k}_{i,j})\,, \nonumber \\[1mm]
&&u_{k}(s)\coloneq \left\{\begin{array}{ll}
 \omega^{k}_{i,j} (s-s^{k}_{i,j})  & \text{if $u^{k}_{i,j}\neq u^{k}_{i,j+1}$}\,, \label{e.interpolantu} \\ [2mm]
 u^{k}_{i,j}=u^{k}_{i,j+1} &\text{otherwise}\,,
\end{array}\right.\\[1mm]
&& z_{k}(s) \coloneq  z^{k}_{i, j} = z_{k}(s^{k}_{i,j}) \,.\nonumber 
\end{eqnarray}
Finally, for $s \in [ s^{k}_{i,j+\frac{1}{2}}, s^{k}_{i,j+1} ]$ we define
\begin{eqnarray}
&& t_{k}(s) \coloneq t^{k}_{i} = t_{k} ( s^{k}_{i,j+\frac{1}{2}} )  \,, \nonumber \\ [1mm]
&& u_{k}(s) \coloneq u^{k}_{i,j+1} = u_{k} ( s^{k}_{i,j+\frac{1}{2}})  \,, \label{e.interpolantz} \\ [1mm]
&& z_{k}(s) \coloneq \left\{ \begin{array}{ll}
\zeta^{k}_{i,j} ( s - s^{k}_{i,j+\frac{1}{2}} ) & \text{if $z^{k}_{i,j}\neq z^{k}_{i,j+1}$} \,, \\ [2mm]
z^{k}_{i,j}=z^{k}_{i,j+1} & \text{otherwise}\,. \nonumber
\end{array}\right.
\end{eqnarray}
In the limit as~$s \to s^k_i$, we have that $t_{k}(s) \to t_{k}(s^{k}_{i})=t^{k}_{i}$ and $z_{k} (s) \weakto z^{k}_{i}=: z_{k}(s^{k}_{i})$ in $H^1(\Om)$. As for~$u_{k}$, by Proposition~\ref{p.6} we know that $u^{k}_{i,j} \to u^{k}_{i}$ in~$W^{1,\beta}(\Om;\R^{2})$ for every $\beta\in[2,\tilde{p})$. As a consequence of the exponential decay~\eqref{e.40} in Theorem~\ref{t.2}, we also deduce that $u_k ( s) \to u^{k}_{i} =: u_{k}(s^{k}_{i})$ in~$H^{1}(\Om;\R^{2})$.

In view of Proposition~\ref{p.finitelength}, we may assume that there exists~$0\leq S <+\infty$ such that, up to a constant extension, for every~$k\in\mathbb{N}$ the triple~$(t_{k},u_{k},z_{k})$ is well defined on the interval $[0,S]$, takes values in $[0,T]\times \U\times \Z$, and satisfies $t_{k}(S)=T$. We notice that since~$u^{k}_{i,j}$ are uniformly bounded in~$W^{1,\tilde{p}}(\Om;\R^{2})$, Theorem~\ref{t.2} implies that $u_{k}(s)$ is bounded in~$H^{1}(\Om;\R^{2})$ uniformly w.r.t.~$k$ and~$s$. It follows that also~$z^{k}_{i,j}$ are bounded in~$H^{1}(\Om)$ and hence, by Theorem~\ref{t.3},~$z_{k}(s)$ is bounded in~$H^{1}(\Om)$ uniformly w.r.t.~$k$ and~$s$. Moreover, as a consequence of Corollaries~\ref{c.gfu} and~\ref{c.gfz} and of the above construction, we have
\begin{equation}\label{e.boundlip}
t'_{k}(s) + \| u_{k}'(s) \|_{H^{1}} + \| z_{k}'(s) \|_{L^2} \leq 1\qquad\text{for a.e.~$s\in[0,S]$}\,,
\end{equation}
so that the triple $(t_{k},u_{k},z_{k})\in W^{1,\infty}([0,S]; [0,T]\times H^{1}(\Om;\R^{2}) \times L^{2}(\Om))$ is bounded. We notice that~$t_{k}$, $u_{k}$, and $z_{k}$ coincide with their Lipschitz continuous representatives.

We collect in the following proposition the equilibrium properties and a discrete energy-dissipation inequality satisfied by the triple~$(t_{k},u_{k},z_{k})$.

\begin{proposition}\label{p.7}
For every $k,i$ it holds
\begin{equation}\label{e.77}
|\partial_{u}\F|(t_{k}(s^{k}_{i}),u_{k}(s^{k}_{i}),z_{k}(s^{k}_{i}))=0\qquad\text{and}\qquad |\partial_{z}^{-}\F|(t_{k}(s^{k}_{i}),u_{k}(s^{k}_{i}),z_{k}(s^{k}_{i}))=0\,.
\end{equation}
Moreover, for every $s\in[0,S]$ we have
\begin{equation}\label{e.78}
\begin{split}
\F(t_{k}(s),&u_{k}(s), z_{k}(s))\leq\F(0,u_{0},z_{0})-\int_{0}^{s}|\partial_{u}\F|(t_{k}( \sigma ),u_{k}( \sigma ),z_{k}( \sigma ))\,\| u'_k (s) \|_{H^1} \, \di  \sigma \\
&\quad-\int_{0}^{s} |\partial_{z}^{-}\F|(t_{k}( \sigma ),u_{k}( \sigma ),z_{k}( \sigma ))\,\| z'_k(\sigma) \|_{L^2}\,\di  \sigma + \int_{0}^{s} \P(t_{k}( \sigma ),u_{k}( \sigma ),z_{k}( \sigma )) \,t'_{k}( \sigma )\,\di  \sigma \,,
\end{split}
\end{equation}
where we intend $ |\partial_{z}^{-}\F|(t_{k}( \sigma ),u_{k}( \sigma ),z_{k}( \sigma ))\,\| z'_k(\sigma) \|_{L^2}=0$ whenever $\|z'_{k}(\sigma) \|_{L^{2}}=0$.
\end{proposition}

\begin{proof}
The equilibrium equalities in~\eqref{e.77} follow from the construction \eqref{e.interpolantt}-\eqref{e.interpolantz} of the interpolation functions~$t_{k}$,~$u_{k}$, and~$z_{k}$ and from the minimality properties of~$u_{k}(s^{k}_{i}) = u^k_i$ and $z_{k}(s^{k}_{i})=z^k_i$ summarized in~\eqref{e.minu2}-\eqref{e.minz2}.

Let us show~\eqref{e.78}. Without loss of generality, we consider $s\in[0,S_{k}]$, where~$S_{k}$ is defined in~\eqref{e.Sk}. Let~$k$ and~$i\in\{1,\ldots,k\}$ be fixed. For $j=-1$, for every $s\in[s^{k}_{i,-1},s^{k}_{i,0}] = [s^{k}_{i-1},s^{k}_{i,0}]$ we have $u_{k}'(s)= z_{k}'(s)=0$ and, therefore,
\begin{align}
\F(t_{k}(s),& \,u_{k}(s),z_{k}(s)) = \F(t_{k}(s^{k}_{i-1}),u_{k}(s^{k}_{i-1}),z_{k}(s^{k}_{i-1}))+\int_{s^{k}_{i-1}}^{s} \!\!\! \partial_{t}\F(t_{k}( \sigma ),u_{k}( \sigma ),z_{k}( \sigma )) \, t'_{k}( \sigma )\,\di \sigma \nonumber \\
& =   \F(t_{k}(s^{k}_{i-1}),u_{k}(s^{k}_{i-1}),z_{k}(s^{k}_{i-1}))+\int_{s^{k}_{i-1}}^{s} \!\!\! \P (t_{k}( \sigma ),u_{k}( \sigma ),z_{k}( \sigma ))\,t'_{k}( \sigma )\,\di \sigma \label{e.80} \\
& \quad   - \int_{s^{k}_{i-1}}^{s} \!\!\! |\partial_{u}\F|(t_{k}( \sigma ),u_{k}( \sigma ),z_{k}( \sigma ))\,\| u'_k(s) \|_{H^1}\,\di  \sigma - \int_{s^{k}_{i-1}}^{s} \!\!\! |\partial_{z}^{-}\F|(t_{k}( \sigma ),u_{k}( \sigma ),z_{k}( \sigma ))\,\| z'_k(s) \|_{L^2}\,\di  \sigma \nonumber \,,
\end{align}
where, in the last equality, we have used the definition of the power functional~$\P$ in~\eqref{e.power} and~\eqref{e.pw}. For every $j\geq 0$, we distinguish between  $s\in[s^{k}_{i,j},s^{k}_{i,j+\frac{1}{2}}]$ and $s\in[s^{k}_{i,j+\frac{1}{2}},s^{k}_{i,j+1}]$. In the first case we have $t_{k}'(s)=z_{k}'(s)=0$ while $\| u'_k(s) \|_{H^1}=1$ for a.e.~$s \in[s^{k}_{i,j},s^{k}_{i,j+\frac{1}{2}}]$; then, in view of~\eqref{e.72},
\begin{align}
\F(t_{k}(s), & \, u_{k}(s),z_{k}(s)) = \F(t_{k}(s^{k}_{i,j}),u_{k}(s^{k}_{i,j}),z_{k}(s^{k}_{i,j}))-\int_{s^{k}_{i,j}}^{s} \!\! |\partial_{u}\F|(t_{k}( \sigma ),u_{k}( \sigma ),z_{k}( \sigma ))\, \| u'_k(s) \|_{H^1}\, \di  \sigma \nonumber \\
& = \F(t_{k}(s^{k}_{i,j}),u_{k}(s^{k}_{i,j}),z_{k}(s^{k}_{i,j}))  - \int_{s^{k}_{i,j}}^{s} \!\! |\partial_{u}\F|(t_{k}( \sigma ),u_{k}( \sigma ),z_{k}( \sigma )) \,\| u'_k(s) \|_{H^1} \,\di  \sigma \label{e.81} \\
& \quad   - \int_{s^{k}_{i,j}}^{s} \!\! |\partial_{z}^{-}\F|(t_{k}( \sigma ),u_{k}( \sigma ),z_{k}( \sigma ))\,\| z'_k(s) \|_{L^2} \,\di  \sigma +\int_{s^{k}_{i,j}}^{s} \!\!\! \P (t_{k}( \sigma ),u_{k}( \sigma ),z_{k}( \sigma )) \, t'_{k}( \sigma )\,\di \sigma \nonumber \,.
\end{align}
In the second case we have $t_{k}'(s)=u_{k}'(s)=0$ while $\| z'_k(s) \|_{L^2}=1$ for a.e.~$s \in[s^{k}_{i,j+\frac12},s^{k}_{i,j+1}]$; then, by~\eqref{e.73},
\begin{align}
\F(t_{k}(s), & \, u_{k}(s),z_{k}(s))  = \F(t_{k}(s^{k}_{i,j+\frac{1}{2}}),u_{k}(s^{k}_{i,j+\frac{1}{2}}),z_{k}(s^{k}_{i,j+\frac{1}{2}})) - \int_{s^{k}_{i,j+\frac{1}{2}}}^{s} \!\!\!\!\!\!\! |\partial_{z}^{-}\F|(t_{k}( \sigma ),u_{k}( \sigma ),z_{k}( \sigma )) \,\| z'_k(s) \|_{L^2} \,\di  \sigma  \nonumber \\
& = \F(t_{k}(s^{k}_{i, j+\frac{1}{2}}),u_{k}(s^{k}_{i+ \frac{1}{2}}),z_{k}(s^{k}_{i+\frac{1}{2}})) - \int_{s^{k}_{i,j+\frac{1}{2}}}^{s} \!\!\!\!\!\!\! |\partial_{z}^{-}\F|(t_{k}( \sigma ),u_{k}( \sigma ),z_{k}( \sigma ))\,\| z'_k(s) \|_{L^2} \,\di  \sigma \label{e.82} \\
& \quad   - \int_{s^{k}_{i,j+\frac{1}{2}}}^{s} \!\!\!\! |\partial_{u}\F|(t_{k}( \sigma ),u_{k}( \sigma ),z_{k}( \sigma ))\, \| u'_k(s) \|_{H^1} \,\di  \sigma  +\int_{s^{k}_{i, j+\frac{1}{2}}}^{s} \!\!\!\! \P (t_{k}( \sigma ),u_{k}( \sigma ),z_{k}( \sigma )) \, t'_{k}( \sigma )\,\di \sigma \nonumber  \,.
\end{align}
Summing up~\eqref{e.80}-\eqref{e.82}, we deduce that for every $s\in[s^{k}_{i-1},s^{k}_{i})$ it holds
\begin{equation*}\label{e.83}
\begin{split}
\F(t_{k}(s), & \, u_{k}(s),z_{k}(s)) =\F(t_{k}(s^{k}_{i-1}),u_{k}(s^{k}_{i-1}),z_{k}(s^{k}_{i-1}))-\int_{s^{k}_{i-1}}^{s} \!\!\! |\partial_{u}\F|(t_{k}( \sigma ),u_{k}( \sigma ),z_{k}( \sigma )) \, \| u'_k(s) \|_{H^1} \,\di  \sigma  \\
& - \int_{s^{k}_{i-1}}^{s} \!\!\!\! |\partial_{z}^{-}\F|(t_{k}( \sigma ),u_{k}( \sigma ),z_{k}( \sigma ))\ \, \| z'_k(s) \|_{L^2} \,\di  \sigma  + \int_{s^{k}_{i-1}}^{s} \!\!\!\! \P (t_{k}( \sigma ),u_{k}( \sigma ),z_{k}( \sigma ))\, t_{k}' \, ( \sigma )\,\di  \sigma \,.
\end{split}
\end{equation*}
Passing to the limit as $s\to s^{k}_{i}$ by Lemma \ref{l.lscFE} we get 
\begin{equation*}\label{e.84}
\begin{split}
\F(t_{k}(s^{k}_{i}), & \, u_{k}(s^{k}_{i}),  z_{k}(s^{k}_{i})) \leq \F(t_{k}(s^{k}_{i-1}),u_{k}(s^{k}_{i-1}),z_{k}(s^{k}_{i-1}))-\int_{s^{k}_{i-1}}^{s^{k}_{i}} \!\!\! |\partial_{u}\F|(t_{k}( \sigma ),u_{k}( \sigma ),z_{k}( \sigma ))\, \| u'_k(s) \|_{H^1} \,\di  \sigma  \\
& - \int_{s^{k}_{i-1}}^{s^{k}_{i}} \!\!\! |\partial_{z}^{-}\F|(t_{k}( \sigma ),u_{k}( \sigma ),z_{k}( \sigma )) \, \| z'_k(s) \|_{L^2} \,\di  \sigma  + \int_{s^{k}_{i-1}}^{s^{k}_{i}} \!\!\! \P ( t_{k}( \sigma ), u_{k}( \sigma ), z_{k}( \sigma ) ) \, t_{k}'( \sigma ) \, \di  \sigma \,,
\end{split}
\end{equation*}
observing that the passage to the limit in the power integral is straightforward since~$t'_{k}(\sigma)=0$ for $\sigma\in(s^{k}_{i,0}, s^{k}_{i})$. 

Finally, iterating the previous estimate w.r.t.~$i$ and combining again~\eqref{e.80}-\eqref{e.82} we deduce~\eqref{e.78}. 
\qed
\end{proof}

\subsection{Compactness and lower energy inequality}

In the following proposition we show the compactness of the sequence $(t_{k},u_{k},z_{k})$.

\begin{proposition}\label{p.compactness}
There exist a subsequence of~$(t_{k}, u_{k}, z_{k})$ and a triple $(t,u,z)\in W^{1,\infty} ( [0,S] ; [0,T] \times H^{1}(\Om;\R^{2}) \times L^{2}(\Om) )$ such that for every sequence~$s_{k}$ converging to~$s\in[0,S]$ we have
\begin{displaymath}
t_{k}(s_{k})\to t(s)\,,\qquad u_{k}(s_{k})\to u(s) \text{ in~$H^1(\Omega; \R^2)$,}\qquad z_{k}(s_{k}) \rightharpoonup z(s) \text{ weakly in~$H^{1}(\Om)$}.
\end{displaymath}
 Moreover,
\begin{equation}\label{e.boundlip2}
|t'(s)|+\|u'(s)\|_{H^{1}}+\|z'(s)\|_{L^2}\leq 1\qquad\text{for a.e.~$s\in[0,S]$}\,.
\end{equation}
In particular, $s \mapsto t(s)$ is non-decreasing and $t(S) = T$.
\end{proposition}

\begin{proof}
In view of~\eqref{e.boundlip}, we have that there exists a triple $(t,u,z)\in W^{1,\infty}( [0,S] ; [0,T] \times H^{1}(\Om;\R^{2}) \times L^{2}(\Om))$ such that, up to a subsequence, $(t_{k},u_{k},z_{k}) \rightharpoonup (t,u,z)$ weakly* in $W^{1,\infty} ( [0,S] ; [0,T] \times H^{1}(\Om;\R^{2}) \times L^{2}(\Om))$. In particular,  for every $s\in[0,S]$ if $s_k \to s$ we have
\begin{equation}\label{e.1106}
t_{k}(s_k)\to t(s)\,, \qquad u_{k}(s_k)\rightharpoonup u(s) \text{ in~$H^1(\Omega;\R^2)$}\,,\qquad z_{k}(s_k)\rightharpoonup z(s) \text{ in~$H^{1}(\Om)$}\,,
\end{equation} 
the latter being a consequence of the boundedness of~$z_{k}(\sigma)$ in~$H^{1}(\Om)$ uniformly for~$\sigma\in[0,S]$. Inequality~\eqref{e.boundlip2} can be obtained from~\eqref{e.boundlip} by integration and by weak lower semicontinuity of the norms. 
It is easy to check that $s \mapsto t(s)$ is non-decreasing and that $t_k(S) = T \to t(S) = T$.

It remains to show that, along the same subsequence,~$u_{k}$ converges strongly in~$H^1(\Omega; \R^2)$ pointwise  in $[0,S]$. Let us fix~$s\in[0,S]$. For every~$k$, let $i_{k}\in\{1,\ldots,k\}$ and~$j_{k}\in\mathbb{N} \cup \{-1\}$ be such that $s \in [s^{k}_{i_{k},j_{k}},s^{k}_{i_{k},j_{k}+1})$. We have to distinguish between three different cases: up to a further (non-relabelled) subsequence, either $s\in[s^{k}_{i_{k},-1},s^{k}_{i_{k},0})$, or $s\in[s^{k}_{i_{k},j_{k}},s^{k}_{i_{k},j_{k}+\frac{1}{2}})$, or $s\in[s^{k}_{i_{k},j_{k}+\frac{1}{2}},s^{k}_{i_{k},j_{k}+1})$ for every~$k$, see~\eqref{e.interpolantt}-\eqref{e.interpolantz}.

In the first case,  for every index~$k$ we have that $u_{k}(s)=u^{k}_{i_{k}-1}= u_{k}(s^{k}_{i_{k}-1})$ and $z_{k}(s)=z^{k}_{i_{k}-1} = z_{k}(s^{k}_{i_{k}-1})$. If $i_k =1$ for infinitely many~$k$, then $u_{k}(s^{k}_{i_{k}-1})= u_{k}(0)=u_{0}$ and there is nothing to show. Let us therefore assume that $i_{k}\geq 2$ for every~$k$. Hence, by \eqref{e.minu2} we have
\begin{displaymath}
u_{k}(s) = u^k_{i_k-1}=\argmin\,\{\F(t^{k}_{i_{k}-1},u,z^{k}_{i_k-1}):\, u\in\U\} = \argmin\,\{\F(t^{k}_{i_{k}-1},u,z_{k}(s)):\, u\in\U\} .
\end{displaymath}
Since $s^{k}_{i_{k},0}-s^{k}_{i_{k}-1}=\tau_{k}\to0$ as $k\to \infty$ and $t_{k}(s)=t^{k}_{i_{k}-1}+(s-s^{k}_{i_{k}-1})$, we have that $t_{i_{k}-1}\to t(s)$. Moreover, $z_{k}(s)\rightharpoonup z(s)$ weakly in~$H^{1}(\Om)$ by~\eqref{e.1106}. Thus, applying Proposition~\ref{p.6} we deduce that $u_{k}(s) \to \bar{u}$ strongly in~$H^1(\Omega;\R^2)$ where $\bar{u} \in \argmin\,\{\F(t(s) ,u,z(s)):\, u\in\U\}$. Since $u_k (s) \weakto u(s)$ by \eqref{e.1106}, it follows that $\bar{u} = u (s)$ and that the $u_{k}(s) \to u (s)$ strongly in~$H^1(\Omega;\R^2)$.

In the second case we have $s\in[s^{k}_{i_{k},j_{k}},s^{k}_{i_{k},j_{k}+\frac{1}{2}})$ for every~$k$. Here we want to apply Proposition~\ref{p.continuitygradflow}, and we use explicitly the parametrization~$\rho_{k}$ of the gradient flow~$\omega^{k}_{i_{k},j_{k}}$ from Corollary~\ref{c.gfu}. 
As a first step, we show that, up to a subsequence, the initial condition $\omega^{k}_{i_{k},j_{k}} (0) = u^k_{i_k,j_k}$ converges to some $u^*$ strongly in $H^1(\Omega;\R^2)$. If, up to a further subsequence, $j_{k}=0$ for every $k$, we have $\omega^{k}_{i_{k}, 0} (0) = u^k_{i_k,0} =  u^k_{i_k-1}$, thus by \eqref{e.minu2} we know that 
\begin{equation*} 
	u^{k}_{i_{k}-1} = \argmin\,\{ \F ( t^{k}_{i_{k}-1}, u, z^{k}_{i_{k}-1}):\, u\in\U\} \,.
\end{equation*} 
Being $ t^{k}_{i_{k}-1} = t_{k}(s) - \tau_k$, we have that $t^{k}_{i_{k}-1}\to t(s)$ as $k\to\infty$, while, along a subsequence, we have $z^{k}_{i_{k}-1} \rightharpoonup z^*$ weakly in~$H^{1}(\Om)$ for some $z^*\in \Z$. Therefore, again by Proposition~\ref{p.6} we get that $u^{k}_{i_{k}, 0} \to u^*$ in~$H^{1}(\Om;\R^{2})$. 
%
%
In a similar way, if~$j_{k}\geq 1$ for every $k$ large enough, we have~$\omega^{k}_{i_{k}, j_{k}}(0)= u^{k}_{i_{k}, j_{k}}$ and, using~\eqref{e.minu}, we get 
\begin{equation*} 
u^{k}_{i_{k},j_{k}} = \argmin\,\{ \F ( t^{k}_{i_{k}}, u, z^{k}_{i_{k},j_{k}-1}):\, u\in\U\} \,.
\end{equation*} 
Then $t^{k}_{i_{k}} =t_{k}(s) \to t(s)$ as $k\to\infty$, while, up to a subsequence,  $z^{k}_{i_{k},j_{k}-1} \rightharpoonup z^*$ weakly in~$H^{1}(\Om)$ for some $z^*\in \Z$. As above, by Proposition~\ref{p.6} we conclude that $u^{k}_{i_{k}, j_{k}} \to u^{*}$ in~$H^{1}(\Om;\R^{2})$. 
In all the cases, we have that the initial condition of the reparametrized gradient flow~$\omega^{k}_{i_{k},j_{k}}$ converges in~$H^1(\Omega;\R^2)$ to some~$u^*$. Now, let us consider the parametrization $\rho_{k}(s - s^{k}_{i_{k},j_{k}})\in [0,+\infty)$. Up to a subsequence, we may assume that~$\rho_{k}(s - s^{k}_{i_{k},j_{k}})\to \overline{\rho}\in[0,+\infty]$. Since $t_{k}(s)\to t(s)$, $z_{k}(s)\rightharpoonup z(s)$ weakly in~$H^{1}(\Om)$, and $\omega^{k}_{i_{k},j_{k}}(0) \to u^*$ in~$H^1(\Omega;\R^2)$, from Corollary~\ref{c.5} we deduce that $u_{k}(s) = \omega^{k}_{i_{k},j_{k}}(s - s^{k}_{i_{k},j_{k}})$ admits a strong limit~$\widetilde{u}$ in~$H^{1}(\Om;\R^{2})$. In view of~\eqref{e.1106} $\widetilde{u}= u(s)$ and $u_{k}(s) \to u (s)$ strongly in~$H^1(\Omega;\R^2)$.

Finally, let us consider the case $s\in[s^{k}_{i_{k},j_{k}+\frac{1}{2}},s^{k}_{i_{k},j_{k}+1})$ for every~$k$. Then, $t_{k}(s)=t^{k}_{i_{k}}$ and $u_{k}(s)=u^{k}_{i_{k},j_{k}+1}$. By construction of~$u^{k}_{i_{k},j_{k}+1}$ in~\eqref{e.minu}, we have that
\begin{displaymath}
u_{k}(s)=u^{k}_{i_{k},j_{k}+1}=\argmin\,\{\F(t^{k}_{i_{k}},u,z^{k}_{i_{k},j_{k}}):\, u\in\U\}\,.
\end{displaymath}
Again, we know that $t^{k}_{i_k} \to t(s)$ and that, up to subsequence, $z^{k}_{i_{k},j_{k}}\rightharpoonup z^*$ weakly in~$H^{1}(\Om)$ for some~$z^* \in\Z$. We are in a position to apply again Proposition~\ref{p.6}, which implies the strong convergence of~$u_{k}(s)$ to~$u(s)$ in~$H^1(\Omega;\R^2)$.

Combining the three cases described above, we have shown that every subsequence of~$u_{k}(s)$ admits a further subsequence converging to~$u(s)$ in~$H^1(\Omega;\R^2)$. Hence, the whole sequence~$u_{k}(s)$ converges to~$u(s)$ in~$H^1(\Omega;\R^2)$ for every $s\in [0,S]$. Noticing that, by~\eqref{e.boundlip}, $\| u_{k}( s_{k}) - u_{k}(s) \|_{H^{1}}  \leq |s_{k} - s|$, we conclude the proof.\qed

\end{proof}

%

We are now in a position to prove the lower energy-dissipation inequality for the triple~$(t,u,z)$.

\begin{proposition}\label{p.3}
Let $g$, $u_{0}$, and~$z_{0}$ be as in Theorem \ref{t.1}. Let $(t,u,z)\colon[0,S]\to[0,T]\times\U\times\Z$ be given by Proposition \ref{p.compactness}. Then $(t,u,z)$ satisfies~$(a)$-$(d)$ of Theorem \ref{t.1} and for every~$s\in[0,S]$ it holds 
\begin{equation}\label{e.3e}
\begin{split}
\F(t(s),u(s),z(s)) \leq&  \,\F(0,u_{0},z_{0})-\int_{0}^{s}|\partial_{u}\F|(t(\sigma),u(\sigma),z(\sigma))\, \| u' (s) \|_{H^1} \,\di\sigma \\
&-\int_{0}^{s}|\partial^-_{z}\F|(t(\sigma),u(\sigma),z(\sigma))\, \| z' (s) \|_{L^2} \,\di\sigma +\int_{0}^{s}\P(t(\sigma),u(\sigma),z (\sigma)) \,t'(\sigma)\,\di\sigma \,.
\end{split}
\end{equation}
\end{proposition}

\begin{proof}
We have already seen that the function~$s\mapsto t(s)$ is non decreasing and such that~$t(S)=T$. Thus, condition~$(b)$ is satisfied.

From inequality~\eqref{e.boundlip2} we get~$(a)$. Moreover, being $s\mapsto z_{k}(s)$ non-increasing for every $k\in\mathbb{N}$, it is clear that the pointwise limit $s\mapsto z(s)$ is non-increasing, so that~$(c)$ holds.

Let us now show property~$(d)$. For every $s\in[0,S]$ of continuity for~$(t,u,z)$ we can find a sequence $s_{m}\in[0,S]$ such that $s_{m}\to s$ and $t(s_{m})\neq t(s)$ for every~$m$. Without loss of generality, we may assume that $s_{m} \leq s$. Since~$t_{k}$ converges pointwise to~$t$, we can construct a subsequence~$k_{m}$ such that~$t_{k_{m}}(s_{m}) \neq t_{k_{m}}(s)$ for every~$m$. By construction of the interpolation functions~$t_{k_{m}}$ (see~\eqref{e.interpolantt}-\eqref{e.interpolantz}), there exists a sequence of indexes~$i_{m}\in\{1,\ldots , k_{m}\}$ such that, up to a further subsequence, one of the following conditions is satisfied:
\begin{displaymath}
s \in [s^{k_{m}}_{i_{m}-1} , s^{k_{m}}_{i_{m}, 0}) \qquad \text{or}\qquad 
s_{m} \leq s^{k_{m}}_{i_{m}-1} < s^{k_{m}}_{i_{m}, 0} \leq s \qquad\text{or}\qquad
 s^{k_{m}}_{i_{m}-1} < s_{m}\leq s^{k_{m}}_{i_{m}, 0} \leq s\,.
\end{displaymath}
In any case, since $|s^{k_{m}}_{i_{m}, 0} - s^{k_{m}}_{i_{m}-1}| \leq \tau_{k_{m}}$ and $s_{m} \to s$, we have that $s^{k_{m}}_{i_{m}-1} \to s$ as $m\to\infty$. In view of~\eqref{e.77} of Proposition~\ref{p.7}, we know that
\begin{equation}\label{e.88}
\begin{split}
& |\partial_{u}\F| \big (t_{k_{m}}(s^{k_{m}}_{i_{k_{m}-1}}), u_{k_{m}}(s^{k_{m}}_{i_{k_{m}-1}}),z_{k_{m}}(s^{k_{m}}_{i_{k_{m}-1}}) \big ) = 0 \,,\\
& |\partial_{z}^{-}\F| \big( t_{k_{m}}(s^{k_{m}}_{i_{k_{m}-1}}), u_{k_{m}}(s^{k_{m}}_{i_{k_{m}-1}}),z_{k_{m}}(s^{k_{m}}_{i_{k_{m}-1}}) \big) = 0 \,.
\end{split}
\end{equation}
By Proposition~\ref{p.compactness} we know that $t_{k_{m}}(s^{k_{m}}_{i_{m}-1}) \to t(s)$ in~$[0,T]$, $u_{k_{m}}(s^{k_{m}}_{i_{m}-1})\to u(s)$ in~$H^1(\Omega; \R^2)$, and $z_{k_{m}}(s^{k_{m}}_{i_{k_{m}}-1})\rightharpoonup z(s)$ weakly in~$H^{1}(\Om)$. Hence, applying Lemmata~\ref{l.2} and~\ref{l.3} and passing to the limit in~\eqref{e.88} as~$m\to \infty$ we get the equilibrium conditions~$(d)$.


\separe

The proof of the lower energy-dissipation inequality~\eqref{e.3e} is divided into two steps.
Clearly, the starting point is \eqref{e.78}, i.e.,
\begin{equation}\label{e.78bis}
\begin{split}
\F(t_{k}(s),&u_{k}(s), z_{k}(s))\leq\F(0,u_{0},z_{0})-\int_{0}^{s}|\partial_{u}\F|(t_{k}( \sigma ),u_{k}( \sigma ),z_{k}( \sigma ))\,\| u'_k (s) \|_{H^1} \, \di  \sigma \\
&\quad-\int_{0}^{s}|\partial_{z}^{-}\F|(t_{k}( \sigma ),u_{k}( \sigma ),z_{k}( \sigma ))\,\| z'_k(s) \|_{L^2}\,\di  \sigma + \int_{0}^{s} \P(t_{k}( \sigma ),u_{k}( \sigma ),z_{k}( \sigma )) \,t'_{k}( \sigma )\,\di  \sigma \,.
\end{split}
\end{equation}

{\bf Step 1: Slopes.}
By \eqref{e.1106} and Lemma \ref{l.lscFE} we get 
$$
	\F ( t(s) , u(s) , z(s) ) \le \liminf_{k \to \infty}  \F ( t_k (s) , u_k(s) , z_k(s)) .
$$
Let us take the limsup in the right hand side of \eqref{e.78bis}. 
The inequality
\begin{align}
	\int_{0}^{s}|\partial_{u}\F|& (t(\sigma),u(\sigma),z(\sigma))\, \| u' (s) \|_{H^1} +  |\partial^-_{z}\F|(t(\sigma),u(\sigma),z(\sigma))\, \| z' (s) \|_{L^2}   \,\di\sigma \label{e.balder4} \\
	& \le
	\liminf_{k \to \infty}
	\int_{0}^{s}|\partial_{u}\F|(t_k (\sigma),u_k(\sigma),z_k(\sigma))\, \| u'_k (s) \|_{H^1}  + |\partial^-_{z}\F|(t_k(\sigma),u_k(\sigma),z_k(\sigma))\, \| z'_k (s) \|_{L^2} \,\di\sigma \nonumber
\end{align}
follows, for instance, from \cite[Theorem 3.1]{Balder_RCMP85}. Let us see how our setting fits into the 
framework and the notation of \cite{Balder_RCMP85}. We set $X = [0,T] \times H^1(\Om;\R^{2})
\times \{z\in H^{1}(\Om;[0,1]):\, \|z\|_{H^{1}}\leq R\}$ and $\Xi = H^1  (\Omega ; \R^2)  \times L^2 (\Omega)$. 
The space~$X$ is endowed with the strong topology in~$[0,T]\times H^{1}(\Om;\R^{2})$ and the weak topology in the ball of~$H^{1}(\Om)$. Being the latter metrizable,~$X$ is a metric space. The space~$\Xi$ is endowed with the weak topology. 
For $x = ( t , u , z)$ and for $\xi = (  u'  , z' )$ the integrand is
$$
l ( x , \xi) = \left\{ \begin{array}{ll}
| \partial_u \F |  ( t , u ,z)  \ \| u' \|_{H^1}  + | 
\partial_z^- \F | ( t , u ,z) \ \| z' \|_{L^2} & \text{if $\|z'\|_{L^{2}} \neq 0$}\,,\\[1mm]
| \partial_u \F |  ( t , u ,z)  \ \| u' \|_{H^1} & \text{if $\|z'\|_{L^{2}} =0$}\,.
\end{array}\right.
$$
Clearly $l \ge 0$. Let us check that $l ( x ,\cdot)=l(t,u,z,\cdot ,\cdot)$ is convex in $\Xi = H^1(\Om;\R^2) \times L^2(\Om)$. If $| \partial_z^- \F| (t,u,z)=+\infty$, then
$$
l ( x , \xi)  = \left\{ \begin{array}{ll}
+\infty & \text{if $\|z'\|_{L^{2}} \neq 0$}\,,\\[1mm]
| \partial_u \F |  ( t , u ,z)  \ \| u' \|_{H^1} & \text{if $\|z'\|_{L^{2}} =0$}
\end{array}\right. = | \partial_u \F |  ( t , u ,z)  \ \| u' \|_{H^1} + \chi_{\{ z' = 0\} }\,,
$$
where~$\chi$ denotes the indicator function. Hence, $l(x,\cdot)$ is convex, since it is the sum of convex functions. If~$|\partial_{z}^{-}\F|(t,u,z)<+\infty$, then
$$
l ( x , \xi) = | \partial_u \F |  ( t , u ,z)  \ \| u' \|_{H^1}  + | \partial_z^- \F | ( t , u ,z) \ \| z' \|_{L^2} 
$$
which is convex w.r.t~$(u',z')$. We now show that $l ( \cdot , \cdot)$ is sequentially lower semicontinuous in $X \times \Xi$. Let $(u'_{k}, z'_{k}) \rightharpoonup (u', z')$ (weakly) in~$\Xi$ and let $(t_{k},u_{k}, z_{k}) \to (t, u, z)$ in the metric of~$X$, that is, $t_{k}\to t$, $u_{k} \to u$ in~$H^{1}(\Om;\R^{2})$, and $z_{k}\rightharpoonup z$ in~$H^{1}(\Om)$. We notice that by Lemma~\ref{l.3} and the fact that~$|\partial_{u} \F|<+\infty$ on~$X$,
\begin{equation}\label{e.balder1}
|\partial_{u} \F| (t,u,z)\, \|u'\|_{H^{1}} \leq \liminf_{k\to\infty} \, |\partial_{u} \F| (t_{k}, u_{k}, z_{k}) \, \|u'_{k}\|_{H^{1}}\,.
\end{equation}
 If~$\| z' \|_{L^{2}} =0$, then~\eqref{e.balder1} is enough to show lower semicontinuity of~$l$. If~$\|z'\|_{L^{2}} \neq 0$ and~$|\partial_{z}^{-} \F|( t,u,z)=+\infty$, by the lower semicontinuity of the~$L^{2}$-norm we have that $\|z'_{k}\|_{L^{2}}> \delta>0$ for some positive~$\delta$ and for every~$k$ sufficiently large. Thus,
\begin{equation}\label{e.balder2}
\liminf_{k\to\infty} \, |\partial_{z}^{-} \F| (t_{k}, u_{k}, z_{k})\, \|z'_{k}\|_{L^{2}} \geq \delta \liminf_{k\to\infty} \,  |\partial_{z}^{-} \F| (t_{k}, u_{k}, z_{k}) =+\infty\,,
\end{equation}
where the last equality follows from Lemma~\ref{l.2}. If, instead,~$\|z'\|_{L^{2}} \neq 0$ and $|\partial_{z}^{-} \F|( t,u,z) < +\infty$, then Lemma~\ref{l.2} implies
\begin{equation}\label{e.balder3}
|\partial_{z}^{-} \F| (t,u,z)\, \|z'\|_{L^{2}} \leq \liminf_{k\to\infty} \, |\partial_{z}^{-} \F| (t_{k}, u_{k}, z_{k}) \, \|z'_{k}\|_{L^{2}} \,.
\end{equation}
Collecting inequalities~\eqref{e.balder1}-\eqref{e.balder3} we deduce the lower semicontinuity of~$l$.

By Proposition \ref{p.compactness} we know that $x_k = ( t_k , u_k , z_k)$ converges pointwise in~$[0,S]$ to $x = ( t , u , z)$ w.r.t.~the metric of $X$, and thus in 
measure. Moreover, again by Proposition~\ref{p.compactness}, we have 
that $\xi_k = ( u'_k , z'_k ) $ converges to $ \xi = (u' , z')$ weakly* in 
$L^\infty ( (0,S) ; H^1 (\Omega; \R^2)  \times L^2 (\Omega) )$ and thus weakly in~$L^1 ( (0,S) ; H^1 (\Omega; \R^2)  \times L^2 (\Omega) )$. Hence,~\eqref{e.balder4} holds.

{\bf Step 2: Power.} We claim that 
\begin{equation}\label{e.91}
\int_{0}^{s} \P ( t( \sigma ), u( \sigma ), z( \sigma ) )\, t'( \sigma )\,\di \sigma = \lim_{k\to\infty}\,\int_{0}^{s} \P ( t_{k}( \sigma ), u_{k}( \sigma ), z_{k}( \sigma ) )\, t'_{k}( \sigma )\,\di \sigma \,.
\end{equation}


Let us fix $s\in[0,S]$. By definition~\eqref{e.power} of~$\P$ we have that
\begin{displaymath}
\int_{0}^{s} \P(t_{k}(\sigma), u_{k}(\sigma), z_{k}(\sigma) )\, t'_{k}(\sigma) \,\di\sigma = \int_{0}^{s}  \int_{\Om} \partial_{\strain} W \big( z_{k}(\sigma) , \strain (u_{k}(\sigma) + g(t_{k}(\sigma)) ) \big) {\,:\,} \strain ( \dot{g}(t_{k} (\sigma)) )\, t'_{k}(\sigma) \,\di x\,\di\sigma\,.
\end{displaymath}
In order to show~\eqref{e.91}, we will prove that $\strain ( \dot{g}(t_{k} (\cdot)) ) t'_{k}(\cdot) \rightharpoonup \strain ( \dot{g}(t (\cdot)) ) t'(\cdot)$ in $L^{q}([0,S]; L^{2}(\Om;\mathbb{M}^{2}_{s}))$ and that $  \partial_{\strain} W ( z_{k}(\cdot) , \strain (u_{k}(\cdot) + g(t_{k}(\cdot)) )) \to \partial_{\strain} W ( z(\cdot) , \strain (u(\cdot) + g(t(\cdot)) ))$ in $L^{q'}([0,S]; L^{2}(\Om;\mathbb{M}^{2}_{s}))$.

Let us start with the latter. Remember that 
$$ \partial_{\strain} W \big( z , \strain (u + g ) \big)  =  2 h(z) \big(  \mu \strain_{d} ( u +g) + \kappa \strain_{v}^{+} (u+g) \big) - 2\kappa \strain_{v}^{-} ( u + g) .  $$ 
For every $\sigma\in[0,S]$, by Lemma~\ref{l.HMWw} and since~$t_{k}\to t$,~$u_{k}\to u$ in~$H^{1}(\Om;\R^{2})$, and~$z_{k} \to z$ in~$L^{2}(\Om)$, we have that, up to a not relabelled subsequence, $ \partial_{\strain} W ( z_{k}(\sigma) , \strain (u_{k}(\sigma) + g(t_{k}(\sigma)) )) \to \partial_{\strain} W ( z(\sigma) , \strain (u(\sigma) + g(t(\sigma)) ))$ a.e.~in~$\Om$. Since~$z_{k}(\sigma)$ takes values in~$[0,1]$, we can apply~$(c)$ of Lemma~\ref{l.HMWw} to deduce that 
\begin{equation}\label{e.pow2}
	\big| \partial_{\strain} W \big( z_{k}(\sigma) , \strain (u_{k}(\sigma) + g(t_{k}(\sigma)) )\big)\big|\leq C |\strain (u_{k}(\sigma)+ g(t_{k}(\sigma)) ) |
\end{equation}
for some positive constant~$C$ independent of~$k$. Therefore, by dominated convergence we get that $\partial_{\strain} W ( z_{k}(\sigma) , \strain (u_{k}(\sigma) + g(t_{k}(\sigma)) ))$ converges to~$\partial_{\strain} W ( z(\sigma) , \strain (u(\sigma) + g(t(\sigma)) ))$ strongly in~$L^{2}(\Om;\mathbb{M}^{2}_{s})$. Moreover, being~$u_{k}$ and~$g\circ t_{k}$ bounded in $L^{\infty}([0,S];H^{1}(\Om;\R^{2}))$, in view of~\eqref{e.pow2} and of the previous convergence we deduce that $ \partial_{\strain} W ( z_{k}(\cdot) , \strain (u_{k}(\cdot) + g(t_{k}(\cdot)) ))$ converges to $ \partial_{\strain} W ( z(\cdot) , \strain (u(\cdot) + g(t(\cdot)) ))$ strongly in $L^{q'}([0,S]; L^{2}(\Om;\mathbb{M}^{2}_{s}))$ (actually, in $L^{\nu}([0,S]; L^{2}(\Om;\mathbb{M}^{2}_{s}))$ for every $\nu<+\infty$).

As for $\strain ( \dot{g}(t_{k} (\cdot)) ) \,t'_{k}(\cdot)$, we proceed by a density argument. Indeed, by density of $C^{1}_{c}([0,T]; L^{2}(\Om;\mathbb{M}^{2}_{s}))$ in $L^{q}([0,T]; L^{2}(\Om;\mathbb{M}^{2}_{s}))$, for every $\delta>0$ there exists~$E\in C^{1}_{c}([0,T]; L^{2}(\Om;\mathbb{M}^{2}_{s}))$ such that 
\begin{equation}\label{e.pow3}
\| E - \strain ( \dot{g} ) \|_{L^{q}([0,T]; L^{2}(\Om;\mathbb{M}^{2}_{s}))}\leq \delta\,.
\end{equation}
Using a change of variables, that~$t'_{k}(\sigma)\leq 1$ for a.e.~$\sigma\in[0,S]$, and~\eqref{e.pow3}, we have that
\begin{equation}\label{e.pow4}
\begin{split}
\int_{0}^{S}  \| E (t_{k} (\sigma))\,t'_{k}(\sigma) - \strain (\dot{g} ( t_{k}(\sigma)))\, t'_{k}(\sigma)\|_{L^2}^{q}\,\di\sigma &\leq
\int_{0}^{S} \| E (t_{k} (\sigma))  - \strain (\dot{g} ( t_{k}(\sigma)))\|_{L^2}^{q} \, t'_{k}(\sigma)\,\di\sigma\\
& \leq \int_{0}^{T} \|E (t) - \strain (\dot{g}(t))\|_{L^2}^{q}\,\di t \leq \delta^{q}\,.
\end{split}
\end{equation}
The same inequality holds for $E(t(\cdot))\,t'(\cdot) - \strain(\dot{g}(t(\cdot)))\, t'(\cdot)$.

Let us now fix $\varphi \in L^{q'}([0,S]; L^{2}(\Om;\mathbb{M}^{2}_{s}))$. Simply by adding and subtracting $(E\circ t_{k})\, t'_{k}$ and~$(E\circ t)\, t'$, we have that
\begin{equation}\label{e.pow5}
\begin{split}
\int_{0}^{S} \int_{\Om}   \Big( \strain (\dot{g} ( t_{k}(\sigma)))\, t'_{k}(\sigma) & - \strain (\dot{g} ( t(\sigma)))\, t'(\sigma) \Big) {\,:\,} \varphi\,\di x\,\di\sigma \\ 
& = \int_{0}^{S} \int_{\Om} \Big( \strain (\dot{g} ( t_{k}(\sigma)))\, t'_{k}(\sigma)  - E (t_{k} (\sigma))\,t'_{k}(\sigma)\Big) {\,:\,} \varphi\,\di x\,\di\sigma\\
& \quad + \int_{0}^{S}\int_{\Om} \Big( E (t_{k} (\sigma))\,t'_{k}(\sigma) - E(t(\sigma))\,t'(\sigma) \Big ) {\,:\,} \varphi\,\di x\,\di\sigma  \\
& \quad + \int_{0}^{S} \int_{\Om} \Big( E(t(\sigma))\,t'(\sigma) - \strain (\dot{g} ( t(\sigma)))\, t'(\sigma) \Big) {\,:\,} \varphi\,\di x\,\di\sigma\,.
\end{split}
\end{equation}
By~\eqref{e.pow4}, the first term on the right-hand side of~\eqref{e.pow5} can be estimated by
\begin{displaymath}
 \int_{0}^{S} \int_{\Om} \Big( \strain (\dot{g} ( t_{k}(\sigma)))\, t'_{k}(\sigma)  - E (t_{k} (\sigma))\,t'_{k}(\sigma)\Big) {\,:\,} \varphi\,\di x\,\di\sigma \leq \delta\,  \|\varphi\|_{L^{q'}([0,S]; L^{2}(\Om;\mathbb{M}^{2}_{s}))}\,.
\end{displaymath}
The same estimate holds for the third term on the right-hand side of~\eqref{e.pow5}. Recalling that $t_{k} \rightharpoonup t$ weakly* in~$W^{1,\infty}(0,S)$ and that $E\in C^{1}_{c}([0,T]; L^{2}(\Om;\mathbb{M}^{2}_{s}))$, we get that $E \circ t_{k} \to E\circ t$ strongly in $L^{q}([0,S]; L^{2}(\Om;\mathbb{M}^{2}_{s}))$ and $t'_{k}\varphi \rightharpoonup t' \varphi$ weakly in $L^{q'}([0,S]; L^{2}(\Om;\mathbb{M}^{2}_{s}))$, so that
\begin{displaymath}
\lim_{k\to\infty}\, \int_{0}^{S}\int_{\Om} \Big( E (t_{k} (\sigma))\,t'_{k}(\sigma) - E(t(\sigma))\,t'(\sigma) \Big ) {\,:\,} \varphi\,\di x\,\di\sigma = 0\,.
\end{displaymath}
Collecting the above inequalities, taking the modulus of~\eqref{e.pow5} and passing to the limsup as $k\to \infty$ we obtain
\begin{displaymath}
\limsup_{k\to\infty}\, \left| \int_{0}^{S} \int_{\Om}   \Big( \strain (\dot{g} ( t_{k}(\sigma)))\, t'_{k}(\sigma)  - \strain (\dot{g} ( t(\sigma)))\, t'(\sigma) \Big) {\,:\,} \varphi\,\di x\,\di\sigma \right| \leq 2\delta\,  \|\varphi\|_{L^{q'}([0,S]; L^{2}(\Om;\mathbb{M}^{2}_{s}))}\,.
\end{displaymath}
Hence, passing to the limit as $\delta\to 0$, by the arbitrariness of~$\varphi\in L^{q'}([0,S]; L^{2}(\Om;\mathbb{M}^{2}_{s}))$ we deduce that $\strain ( \dot{g}(t_{k} (\cdot)) ) \,t'_{k}(\cdot)$ converges to $\strain ( \dot{g}(t (\cdot)) ) \,t'(\cdot)$ weakly in $L^{q}([0,S]; L^{2}(\Om;\mathbb{M}^{2}_{s}))$, and this concludes the proof of~\eqref{e.91}.\qed
\end{proof}

\subsection{Upper energy-dissipation inequality} \label{s.inequality}

This section is devoted to the proof of the inequality
\begin{displaymath}
\begin{split}
\F(t(s),u(s),z(s)) \geq &  \,\F(0,u_{0},z_{0})-\int_{0}^{s}|\partial_{u}\F|(t(\sigma),u(\sigma),z(\sigma))\, \| u' (s) \|_{H^1} \,\di\sigma \\
&-\int_{0}^{s}|\partial^-_{z}\F|(t(\sigma),u(\sigma),z(\sigma))\, \| z' (s) \|_{L^2} \,\di\sigma +\int_{0}^{s}\P(t(\sigma),u(\sigma),z (\sigma)) \,t'(\sigma)\,\di\sigma 
\end{split}
\end{displaymath}
for the triple~$(t,u,z)$ defined in Proposition~\ref{p.compactness}.

The function~$z$ belongs to~$W^{1,\infty}([0,S]; L^{2}(\Om)) \cap L^{\infty}([0,S];H^{1}(\Om))$. Therefore,~$z$ is differentiable a.e.~in~$(0,S)$ with~$z'(s)\in L^{2}(\Om)$. We set $z'(s)=0$ for every~$s\in(0,S]$ of non-differentiability for~$z$. Clearly, this does not change the differentiability properties of~$z$, the representation
\begin{displaymath}
z(s)=z_{0}+\int_{0}^{s} z'(\sigma)\,\di\sigma,
\end{displaymath}
and the energy-dissipation inequality above.
In what follows we need the following auxiliary piecewise constant interpolation functions 
\begin{eqnarray}
&& \underline{t}_{k}(\sigma):= \left\{
\begin{array}{ll}
t^{k}_{i-1} & \text{if $\sigma \in [s^{k}_{i,-1}, s^{k}_{i,\frac{1}{2}})$}\,, \label{e.undert}\\ [2mm]
t^{k}_{i} &\text{if $\sigma \in [ s^{k}_{i,\frac{1}{2}}, s^{k}_{i} )$} \,,
\end{array} \right. \\[2mm]
&& \underline{u}_{k}(\sigma):= \left\{
\begin{array}{lll}
u^{k}_{i-1} & \text{if $\sigma\in[s^{k}_{i,-1},s^{k}_{i,\frac{1}{2}})$}\,,\\[2mm] 
u^{k}_{i,j+1} & \text{if $\sigma\in[s^{k}_{i,j+\frac{1}{2}}, s^{k}_{i,j+\frac{3}{2}})$, $j\geq 0$}\,,
\end{array}\right. \label{e.underu} 
\end{eqnarray}
where~$t^{k}_{i}$,~$u^{k}_{i,j}$,~$s^{k}_{i,j}$, and~$s^{k}_{i}$, $j, k\in\mathbb{N}$, $i=0,\ldots, k$, have been defined in~Section~\ref{s.4.1}.

We now discuss the convergence of~$\underline{t}_{k}$ and~$\underline{u}_{k}$.

\begin{lemma}\label{r.t}
The sequence~$\underline{t}_{k}$ converges pointwise in $[0,S]$ and weakly* in $L^{\infty}(0,S)$  to the function~$t(\cdot)$ defined in Proposition~\ref{p.compactness}.
\end{lemma}

\begin{proof}
Recalling the definition of the affine interpolation function~$t_{k}$ in~\eqref{e.interpolantt}-\eqref{e.interpolantz}, we have that $\underline{t}_{k}(\sigma) = t_{k}(\sigma)$ if~$\sigma\in [s^{k}_{i,\frac{1}{2}}, s^{k}_{i}]$, while $|\underline{t}_{k}(\sigma) - t_{k}(\sigma)| \leq \tau_{k}$ if $\sigma \in [s^{k}_{i,-1}, s^{k}_{i,\frac{1}{2}}]$.\qed
\end{proof}

\begin{lemma}\label{p.nonstationaryz}
Let $\sigma\in (0,S)$ be such that $z'(\sigma)\neq 0$. Then, $\underline{u}_{k}(\sigma)\to u(\sigma)$ in $H^{1}(\Om;\R^{2})$.
\end{lemma}

\begin{proof}
We actually show that there exist a subsequence~$k_{m}$ and a sequence $\sigma_{m} \to \sigma$ such that $\sigma_{{m}} \leq \sigma$ and $u_{k_m}(\sigma_{m})=\underline{u}_{k_{m}}(\sigma)$. From this property and Proposition~\ref{p.compactness} the thesis follows.

Since~$ z'(\sigma) \neq 0$, there exists~$\delta>0$ such that $z(s)\neq z(\sigma)$ for every~$s\in[\sigma-\delta, \sigma)$. Let us fix a sequence $\delta_{m}\searrow 0$. Since, by Proposition~\ref{p.compactness}, $z_{k}$ converges to~$z$ in~$L^{2}(\Om)$ pointwise in~$[0,S]$, for every~$m$ we can find~$k_{m}> k_{m-1}$  such that
\begin{displaymath}
z_{k_{m}}(\sigma - \delta_{m}) \neq z_{k_{m}}(\sigma). 
\end{displaymath}
We deduce that there exists a point~$\sigma^{-}_{m}\in (\sigma - \delta_{m}, \sigma)$ with $z'_{k_{m}}(\sigma_{m}^{-})\neq 0$. Moreover, by definition of the interpolation function~$z_{k}$ in~\eqref{e.interpolantt}-\eqref{e.interpolantz}, $z_{k_{m}}$ changes only in intervals of the form~$[s^{k_{m}}_{i,j+\frac{1}{2}} , s^{k_{m}}_{i,j+1})$. Therefore, there exist suitable indexes $i_{m},j_{m}$ such that $\sigma_{m}^{-}\in [s^{k_{m}}_{i_{m}, j_{m}+\frac{1}{2}} , s^{k_{m}}_{i_{m}, j_{m}+1})$. 

For every~$m$, there exists two indexes~$\lambda_{m}, \gamma_{m}$ such that $\sigma\in[s^{k_{m}}_{\lambda_{m},\gamma_{m}}, s^{k_{m}}_{\lambda_{m}, \gamma_{m}+1})$. We now distinguish three different cases, according to the value of~$\gamma_{m}$ (along an infinite sequence of indexes $m_{n}$ not explicitly indicated):
\begin{itemize}
\item if~$\gamma_{m} = -1$, then $\sigma\in [s^{k_{m}}_{\lambda_{m}-1}, s^{k_{m}}_{\lambda_{m},0})$ and $\underline{u}_{k_{m}}(\sigma)= u^{k_{m}}_{\lambda_{m}-1} = u_{k_{m}}(\sigma)$, so that we could simply set $\sigma_{{m}}:=\sigma$;

\item if~$\gamma_{m}\geq 0$ and~$\sigma\in[s^{k_{m}}_{\lambda_{m},\gamma_{m}+\frac{1}{2}}, s^{k_{m}}_{\lambda_{m},\gamma_{m}+1})$, then $\underline{u}_{k_{m}}(\sigma) = u^{k_{m}}_{\lambda_{m},\gamma_{m}+1} = u_{k_{m}}(\sigma)$, and, as before, we set $\sigma_{m} := \sigma$;

\item if ~$\gamma_{m}\geq 0$ and~$\sigma\in [ s^{k_{m}}_{\lambda_{m},\gamma_{m}}, s^{k_{m}}_{\lambda_{m}, \gamma_{m}+\frac{1}{2}})$, then~$\underline{u}_{k_{m}}(\sigma) = u^{k_{m}}_{\lambda_{m},\gamma_{m}} = u_{k_{m}}(s^{k_{m}}_{\lambda_{m},\gamma_{m}})$. Since~$\sigma_{m}^{-} < \sigma$ with $\sigma_{m}^{-}\in [s^{k_{m}}_{i_{m},j_{m} + \frac{1}{2}}, s^{k_{m}}_{i_{m},j_{m}+1})$, we have that either~$i_{m}<\lambda_{m}$ or~$i_{m}=\lambda_{m}$ and $j_{m}<\gamma_{m}$; in any case
\begin{displaymath}
\sigma_{m}^{-} \leq s^{k_{m}}_{i_{m},j_{m}+1} \leq s^{k_{m}}_{\lambda_{m}, \gamma_{m}} \leq \sigma\,.
\end{displaymath}
Since~$\sigma_{m}^{-}\to \sigma$, we also deduce that $s^{k_{m}}_{\lambda_{m},\gamma_{m}} \to \sigma$, so that we set $\sigma_{m} := s^{k_{m}}_{\lambda_{m},\gamma_{m}}$.\qed
\end{itemize}
\end{proof}


\begin{lemma} \label{p.almoststationaryz}
Let $\sigma \in (0,S]$. Assume that there exists a sequence $\sigma_{m}\nearrow \sigma$ such that $\sigma_{m}< \sigma$ and $z'(\sigma_{m})\neq 0$ for every~$m$. Then, $\underline{u}_{k}(\sigma)\to u(\sigma)$ in~$H^{1}(\Om;\R^{2})$.
\end{lemma}

\begin{proof}
For every~$\sigma_{m}$ we have $\underline{u}_{k}(\sigma_{m}) \to u(\sigma_{m})$ in~$H^{1}(\Om;\R^{2})$ for every $m\in\mathbb{N}$. Hence, we can extract a subsequence~$k_{m}$ such that $\underline{u}_{k_{m}}(\sigma_{m}) \to u(\sigma)$ in~$H^{1}(\Om;\R^{2})$ as $m\to\infty$.


To conclude that also $\underline{u}_{k_{m}}(\sigma)\to u(\sigma)$ we discuss the mutual position of~$\sigma_{m}$ and~$\sigma$. As in the previous Lemma, we could have:
\begin{itemize}
\item if $\sigma \in[s^{k_{m}}_{\lambda_{m},-1},s^{k_{m}}_{\lambda_{m},0})$, then $\underline{u}_{k_{m}}(\sigma) = u^{k_{m}}_{\lambda_{m}-1}= u_{k_{m}}(\sigma)$ and $u_{k_{m}}(\sigma)\to u(\sigma)$ in~$H^{1}(\Om;\R^{2})$;

\item if $\sigma \in[s^{k_{m}}_{\lambda_{m},\gamma_{m}+\frac{1}{2}},s^{k_{m}}_{\lambda_{m},\gamma_{m}+1})$, then $\underline{u}_{k_{m}}(\sigma)=u^{k_{m}}_{\lambda_{m}, \gamma_{m}+1} = u_{k_{m}}(\sigma)$ and $u_{k_{m}}(\sigma)\to u(\sigma)$ in~$H^{1}(\Om;\R^{2})$;

\item if $\sigma_{m},\sigma \in[s^{k_{m}}_{\lambda_{m},\gamma_{m}},s^{k_{m}}_{\lambda_{m},\gamma_{m}+\frac{1}{2}})$, then $\underline{u}_{k_{m}}(\sigma)=\underline{u}_{k_{m}}(\sigma_{m}) \to u(\sigma)$ in~$H^{1}(\Om;\R^{2})$;

\item if $\sigma \in[s^{k_{m}}_{\lambda_{m},\gamma_{m}}, s^{k_{m}}_{\lambda_{m},\gamma_{m}+\frac{1}{2}})$ and $\sigma_{m}\in [s^{k_{m}}_{i_{m},j_{m}}, s^{k_{m}}_{i_{m},j_{m}+1})$ with~$(i_{m},j_{m})\neq (\lambda_{m},\gamma_{m})$ then, being $\sigma_{m}<\sigma$, we have that either~$i_{m}<\lambda_{m}$ or~$i_{m}=\lambda_{m}$ and $j_{m}<\gamma_{m}$. In both cases we have $\sigma_{m} \leq s^{k_{m}}_{i_{m}, j_{m}+1} \leq s^{k_{m}}_{\lambda_{m},\gamma_{m}} \leq \sigma$. Thus, the sequence of nodes~$s^{k_{m}}_{\lambda_{m},\gamma_{m}}$ converges to~$\sigma$. We deduce that $\underline{u}_{k_{m}}(\sigma)= u_{k_{m}}(s^{k_{m}}_{\lambda_{m},\gamma_{m}}) \to u(\sigma)$ in~$H^{1}(\Om;\R^{2})$.
\end{itemize}

Repeating the above argument for any subsequence~$k_{j}$ of~$k$ we conclude the thesis. \qed
\end{proof}

Let us define the set
\begin{equation}\label{e.U}
U := \{ \sigma\in(0,S] : \,\text{there exists a sequence $\sigma_{m}\nearrow \sigma$ such that $\sigma_{m} \leq \sigma$ and $z'(\sigma_{m})\neq 0$} \} \,.
\end{equation}
In view of Lemmata~\ref{p.nonstationaryz} and~\ref{p.almoststationaryz}, for every~$\sigma\in U$ we have $\underline{u}_{k}(\sigma)\to u(\sigma)$ in~$H^{1}(\Om;\R^{2})$. Viceversa, we still have no information on the set $U^{c}:= (0,S]\setminus U$. In the following lemma we show the structure of~$U^{c}$.

\begin{lemma}\label{l.Uc}
%
%
There exist countably many $s^{-}_{i} < s^{+}_{i}$ in $[0,S]$, such that
\begin{displaymath}
U^{c} = \bigcup_{i\in\mathbb{N}} (s^{-}_{i},s^{+}_{i}]\,,
\end{displaymath} 
where the intervals $(s^{-}_{i},s^{+}_{i}]$ are pairwise disjoint.
\end{lemma}

\begin{proof} 
First, note that $$U^{c} = \{ \sigma\in (0,S] : \, z'(\sigma)=0 \text{ and there is no sequence $\sigma_{m}\nearrow \sigma$ such that $z'(\sigma_{m}) \neq 0$} \}. $$
Clearly~$\sigma\in U^{c}$ if and only if there is no sequence~$\sigma_{m}\nearrow \sigma$ such that $z'(\sigma_{m})\neq 0$. This implies that $z$ is constant in a left neighborhood of $\sigma$. Then, if~$z$ is differentiable in~$\sigma$ we have~$z'(\sigma)=0$, if it is not differentiable in~$\sigma$ then~$z'(\sigma)=0$ by convention.

It follows that for every $\sigma\in U^{c}$ there exists a left-neighborhood~$U_{\sigma}$ of~$\sigma$ in~$(0,S]$ such that $U_{\sigma} \subseteq U^{c}$.
Indeed, for every~$\sigma \in U^{c}$ we have that~$z'$ has to vanish in a left-neighborhood of~$\sigma$ in~$(0,S]$. We denote this left-neighborhood with~$U_{\sigma}\subseteq U^{c}$.

We first write~$U^{c}$ as the union of its connected components
\begin{displaymath}
U^{c} = \bigcup_{ \alpha\in A} I_{\alpha}\,,
\end{displaymath}
where~$A$ is some set of indexes. From what we have seen above, each~$I_{\alpha}$ contains at least an interval. Therefore,~$U^{c}$ can be actually written as the union of countably many connected components:
\begin{displaymath}
U^{c} = \bigcup_{i\in\mathbb{N}} I_{i}\,.
\end{displaymath}
For every~$i\in\mathbb{N}$ there exist $s^{-}_{i}< s^{+}_{i}$ such that $(s^{-}_{i},s^{+}_{i})\subseteq I_{i}\subseteq [s^{-}_{i},s^{+}_{i}]$. Since every point in $U^c$ admits a left neighborhood contained in $U^c$, we deduce that~$s^{-}_{i}\notin I_{i}$. On the other hand, $s^{+}_{i} \in I_{i}$. Indeed, $z'(\sigma)=0$ for every~$\sigma\in I_{i}$, so that~$z$ is constant on~$I_{i}$. Hence, if~$z$ is differentiable in~$s^{+}_{i}$ we get~$z'(s^{+}_{i})=0$, if~$z$ is not differentiable in~$s^{+}_{i}$ then~$z'(s^{+}_{i})=0$ by convention. This implies that~$s^{+}_{i}\in I_{i}$. All in all, we have proved that each connected component~$I_{i}$ is of the form~$(s^{-}_{i},s^{+}_{i}]$ for suitable~$s^{-}_{i} < s^{+}_{i} \in [0,S]$.\qed 

\end{proof}

%

We set $R:= S-|U^c|$ and define the absolutely continuous function
\begin{equation}\label{e.beta}
\beta(s) := \int_{0}^{s} \mathbf{1}_{U}( \sigma ) \, \di \sigma\qquad \text{for every $s\in[0,S]$} \,.
\end{equation}

\begin{lemma}\label{l.beta}
The following facts hold:
\begin{itemize}
\item[$(a)$] $\beta \colon [0,S]\to[0,R]$ is $1$-Lipschitz continuous, non-decreasing, and surjective;
\item[$(b)$] let $B:=\{\sigma\in[0,S]:\, \text{$\beta$ is not differentiable in~$\sigma$ or $\beta'(\sigma)\neq \mathbf{1}_{U}(\sigma)$} \}\cup U^{c}$, then $|\beta(B)|=0$;
\item[$(c)$] if $\beta$ is constant in~$[a,b]$, then~$z$ is constant in~$[a,b]$ and~$(a,b]\subseteq U^{c}$.
\end{itemize}
\end{lemma}

\begin{proof}
It is clear that $\beta$ is non-decreasing and 1-Lipschitz continuous. Moreover,~$\beta(0)=0$ and~$\beta(S)=|U|= S - |U^{c}| = R$, so that $\beta$ is onto~$[0,R]$.

Since $|\{\sigma\in[0,S]:\, \text{$\beta$ is not differentiable in~$\sigma$ or $\beta'(\sigma)\neq \mathbf{1}_{U}(\sigma)$} \}|=0$ and~$\beta$ is Lipschitz, we have that
\begin{displaymath}
|\beta ( \{\sigma\in[0,S]:\, \text{$\beta$ is not differentiable in~$\sigma$ or $\beta'(\sigma)\neq \mathbf{1}_{U}(\sigma)$} \})|=0\,.
\end{displaymath}
As for $\beta(U^{c})$, we have that if~$s\in U^{c}$, then there exist $i\in\mathbb{N}$ and~$s^{-}_{i} < s^{+}_{i}$ such that $s\in (s^{-}_{i},s^{+}_{i}]\subseteq U^{c}$. Hence, $\beta(s)=\beta(s^{-}_{i})= \beta(s^{+}_{i})$ and $\beta (U^c) = \{  \beta ( s_i^-) \}_{i \in \N}$, thus $|\beta(U^{c})|=0$. All in all, we have shown that~$|\beta(B)|=0$, so that~$(b)$ holds.

As for~$(c)$, we have that if~$\beta$ is constant in~$[a,b]$ then~$\beta'=0$ in~$(a,b)$. Being~$\beta'(\sigma)=\mathbf{1}_{U}(\sigma)$ a.e.~in~$(0,S]$, we deduce that $|(a,b)\setminus U^{c}|=0$. Therefore, for a.e.~$\sigma\in(a,b)$ we have~$z'(\sigma)=0$, which, together with the continuity of~$z$ in~$[0,S]$, implies that~$z$ is constant on~$[a,b]$. From this we get that~$z'(\sigma)=0$ for every~$\sigma\in (a,b)$. 
Hence~$(a,b)\subseteq U^{c}$. As for the point~$b$, the only possibility to have~$b\in U$ is that~$z$ is differentiable in~$b$ with~$z'(b) \neq 0$. But this can not happen, since~$z$ is constant on~$[a,b]$. Thus,~$b\in U^{c}$.  \qed
\end{proof}

We now introduce the ``right-inverse'' of~$\beta$:
\begin{equation}\label{e.alpha}
\alpha(r):= \min\,\{ s\in[0,S]:\beta(s)=r\}\,.
\end{equation}
The function~$\alpha$ is well-defined on~$[0,R]$ because of the continuity of~$\beta$. 
Its main properties are listed in the next lemma, where we denote with~$\alpha^{\pm}$ the left and right limit of~$\alpha$, where they exist.

\begin{lemma}\label{l.alpha}
The following facts hold:
\begin{itemize}
\item[$(a)$] $\alpha$ is strictly increasing and left-continuous;

\item[$(b)$]  $\beta( \alpha (r)) = \beta(\alpha^{+}(r)) = r$ for every~$r\in[0,R]$, $\alpha(\beta(s))\leq s$ for every $s\in[0,S]$ and $\alpha(\beta(s))=s$ for every $s\in U$;

\item[$(c)$]$\alpha$ is differentiable in $(0,R)\setminus \beta(B)$ with $\alpha'(r) = 1$.
\end{itemize}
\end{lemma}

\begin{proof}
The function~$\alpha$ is strictly increasing since~$\beta$ is increasing. Hence, left and right limits of $\alpha$ exist in every point of $(0,R)$. 

In order to prove the left-continuity of~$\alpha$, we first notice that, by construction, we have $\beta(\alpha(r))=r$ for $r\in[0,R]$. Since~$\alpha$ is strictly increasing, it is clear that $\alpha^{-}( \bar{r} ) = \lim_{r\nearrow \bar{r}} \alpha(r) \leq \alpha(\bar{r})$. To show the opposite inequality, we consider the equality $\beta(\alpha(r)) = r$ and pass to the limit as~$r\nearrow \bar{r}$, which gives $\beta(\alpha^{-}( \bar{r})) = \bar{r}$. From the definition of~$\alpha$ we deduce that $\alpha(\bar{r})\leq \alpha^{-}(\bar{r})$. Therefore, $\alpha^{-}(\bar{r}) = \alpha(\bar{r})$ and~$\alpha$ is left-continuous.

Let us now prove~$(b)$. The equality~$\beta(\alpha(r))=r$ has been already shown while the equality~$\beta(\alpha^{+}(r))=r$ follows by definition of~$\alpha^{+}(r)$ and the continuity of~$\beta$. For every~$s\in[0,S]$, it is clear by construction that~$\alpha(\beta(s)) \leq s$. Let us now consider $s\in U$. By contradiction, let us assume that $\alpha(\beta(s))<s$. Then, the function~$\beta$ is constant in the interval~$[ \alpha(\beta(s)), s]$. By~(c) of Lemma~\ref{l.beta} we have that $z$ is constant in the interval~$[\alpha(\beta(s)), s]$ and $(\alpha(\beta(s)), s]\in U^{c}$, which is a contradiction. Therefore, it has to be~$\alpha(\beta(s))=s$.

Let us now show~$(c)$.
%
We start by proving that every~$\bar{r} \in (0,R)\setminus \beta(U^{c})$ is of continuity for~$\alpha$. In view of~$(a)$, we only have to show that $\alpha(r) \to \alpha(\bar{r})$ for $r\searrow \bar{r}$. By contradiction, let us assume that
\begin{displaymath}
\alpha( \bar{r}) < \lim_{r\searrow \bar{r}}\, \alpha(r) = \alpha^{+}( \bar{r} ) \,.
\end{displaymath}
Then by~$(b)$ and by monotonicity of~$\beta$ we have that~$\beta$ is constant in the interval $[\alpha(\bar{r}) , \alpha^{+} (\bar{r} )]$. From~$(c)$ of Lemma~\ref{l.beta}, we deduce that~$z$ is constant on the same interval and $(\alpha( \bar{r} ) , \alpha^{+} (\bar{r} ) ] \subseteq U^{c}$. Therefore, $\bar{r}\in \beta(U^{c})$, which is a contradiction. Hence,~$\alpha$ is continuous in~$\bar{r}$. 
In view of~$(a)$, we already know that $\alpha\in BV(0,R)$. We now prove that every $\bar{r} \in (0,R) \setminus \beta ( B )$ is of differentiability for~$\alpha$ with~$\alpha'( \bar{r} )=1$. By the previous argument,~$\bar{r}$ is of continuity for~$\alpha$. For $h\in\R$ with~$|h|$ small enough, let us write
\begin{displaymath}
\frac{ \alpha(\bar{r}+h)-\alpha( \bar{r})}{h} = \frac{\alpha(\bar{r}+h)-\alpha( \bar{r})}{(\bar{r} + h) - \bar{r}} = \frac{\alpha(\bar{r}+h)-\alpha( \bar{r})}{\beta(\alpha(\bar{r}+h)) - \beta(\alpha( \bar{r}))}\,.
\end{displaymath}
As~$h\to0$ we have, by continuity of~$\alpha$ in~$\bar{r}$, that~$\alpha (\bar{r}+h) \to \alpha(\bar{r})$. Hence, passing to the limit in the previous equality we get
\begin{displaymath}
\lim_{h\to 0}\, \frac{ \alpha(\bar{r}+h)-\alpha( \bar{r})}{h}  = \frac{1}{\beta'(\alpha(\bar{r}))} = 1\,,
\end{displaymath}
where we have used the fact that $\bar{r} \notin \beta(B)$, so that $\alpha( \bar{r}) \notin B$ and~$\beta$ is differentiable in~$\alpha( \bar{r} )$ with~$\beta'(\alpha(\bar{r})) = \mathbf{1}_{U}( \alpha( \bar{r} ))=1$. Thus, we have proved that~$\alpha$ is differentiable at every~$\bar{r}\in(0,R)\setminus\beta(B)$ and~$\alpha'(\bar{r})=1$. \qed

\end{proof}

We now consider the reparametrized functions
\begin{displaymath}
\tilde{t}:= t\circ \alpha\,,\qquad \tilde{z}:=z\circ \alpha\,, \qquad \tilde{u}:=u\circ\alpha \,.
\end{displaymath}

\begin{lemma}\label{l.tildez}
$\tilde{z}\in W^{1,\infty}([0,R]; L^{2}(\Om))$ with Lipschitz constant~$1$. Moreover, $\tilde{z}'=z'\circ\alpha$ a.e.~in $[0,R]$.
\end{lemma}

\begin{proof} 
Let $\rho<r\in[0,R]$. Being~$z\in W^{1,\infty}([0,S]; L^{2}(\Om))$, we have that
\begin{displaymath}
\| \tilde{z} (r) - \tilde{z} (\rho) \|_{L^2}  = \|z( \alpha (r) ) - z ( \alpha ( \rho ) ) \|_{L^2} \leq \int_{ \alpha ( \rho ) }^{\alpha ( r ) } \| z'( \sigma ) \|_{L^2} \, \di \sigma\,.
\end{displaymath}
Since~$z' = 0$ in~$U^{c}$ and~$\|z'(s)\|_{L^2} \leq 1$ for a.e.~$s\in[0,S]$ by~\eqref{e.boundlip2}, we can continue in the previous chain of inequalities with
\begin{displaymath}
\| \tilde{z} (r) - \tilde{z} (\rho) \|_{L^2}  \leq \int_{ \alpha ( \rho ) }^{ \alpha ( r ) } \| z'( \sigma ) \|_{L^2} \mathbf{1}_{U}( \sigma ) \, \di \sigma
\leq \int_{ \alpha ( \rho ) }^{ \alpha ( r ) } \mathbf{1}_{ U } ( \sigma ) \, \di \sigma = \beta(\alpha( r) ) - \beta ( \alpha( \rho) ) =  r - \rho\,,
\end{displaymath}
where we have used the definition~\eqref{e.beta} of~$\beta$ and~$(b)$ of Lemma~\ref{l.alpha}.

Let us denote with~$C:= \{s\in[0,S]:\, \text{$z$ is not differentiable in~$s$}\}$. Since~$|C|=0$ and~$\beta$ is Lipschitz continuous, we have that $|\beta ( C ) | = 0$. Let us show that~$\tilde{z}$ is differentiable in every~$\bar{r}\in(0,R)\setminus(\beta(B) \cup \beta(C))$. Indeed, we notice that for such~$\bar{r}$ we have, by~$(c)$ in Lemma~\ref{l.alpha}, that~$\alpha$ is differentiable in~$\bar{r}$ with~$\alpha' ( \bar{r} ) = 1$. Moreover, since~$\bar{r}\notin \beta(C)$, from the definition~\eqref{e.alpha} of~$\alpha$ we deduce that~$\alpha(\bar{r})\notin C$, so that~$z$ is differentiable in~$\alpha( \bar{r} )$. Therefore, for~$r \neq  \bar{r}$ we can write
\begin{equation} \label{e.tildezdiff}
\frac{ z ( \alpha ( r )) - z ( \alpha ( \bar{r} ) )} {r - \bar{r}} = \frac{ z ( \alpha ( r )) - z ( \alpha ( \bar{r} ) )} {\alpha(r) - \alpha(\bar{r})}\,\frac{\alpha(r) - \alpha(\bar{r})}{r - \bar{r}}\,,
\end{equation}
since~$\alpha$ is strictly increasing by~$(a)$ of Lemma~\ref{l.alpha}. In view of the previous considerations, we can pass to the limit in~\eqref{e.tildezdiff} as $r\to\bar{r}$ obtaining
\begin{displaymath}
\lim_{ r\to\bar{r}}\, \frac{ z ( \alpha ( r )) - z ( \alpha ( \bar{r} ) )} {r - \bar{r}} = z'(\alpha(\bar{r})) \, \alpha'(\bar{r}) = z'(\alpha(\bar{r}))\,.
\end{displaymath}
In conclusion, we have shown that~$\tilde{z}$ is differentiable in every~$\bar{r}\in(0,R)\setminus (\beta(B)\cup \beta(C))$ with~$\tilde{z}'(\bar{r}) = z'(\alpha(\bar{r}))$. Since~$|\beta(B) \cup \beta(C)|=0$, we get that~$\tilde{z}' = z' \circ \alpha$ a.e.~in~$[0,R]$, and this concludes the proof of the proposition. \qed
\end{proof}

We now go back to the proof of the upper-energy inequality. In the following two lemmata we further investigate the summability of the unilateral slope~$|\partial_{z}^{-}\F|$ on the set~$U$.

\begin{lemma}\label{l.L1-1} 
The function $s \mapsto | \partial_z^- \F | ( t(s), u(s) , z(s) ) \, \mathbf{1}_{U} (s)$ belongs to $L^1(0,S)$.
 \end{lemma}
 
 \begin{proof}
 To prove this property, we slightly modify the energy inequality~\eqref{e.78} of Proposition~\ref{p.7} making use of the piecewise constant interpolation function~$\underline{u}_{k}$ defined in~\eqref{e.underu} on the interval~$[0,S]$.
 
 Let $k$ and~$i\in\{1,\ldots,k\}$ be fixed. For $j=-1$, for every $s\in[s^{k}_{i-1},s^{k}_{i,0}]$ we have $u_{k}'(s)= z_{k}'(s)=0$ and, by~\eqref{e.77},  $|\partial_{z}^{-}\F|(\underline{t}_{k}(s), \underline{u}_{k}(s), z_{k}(s)) = 0$. Therefore,
\begin{align}
\F(t_{k}(s), u_{k}(s), & \ z_{k}(s)) =
 \F(t_{k}(s^{k}_{i-1}), u_{k}(s^{k}_{i-1}),z_{k}(s^{k}_{i-1})) - \int_{s^{k}_{i-1}}^{s} \!\!\! |\partial_{u}\F|(t_{k}(\sigma), u_{k}( \sigma ),z_{k}( \sigma ))\,\| u'_k(\sigma) \|_{H^1}\,\di  \sigma \nonumber \\
&    - \int_{s^{k}_{i-1}}^{s} \!\!\! |\partial_{z}^{-}\F|( \underline{t}_{k}(\sigma), \underline{u}_{k}( \sigma ),z_{k}( \sigma )) \,\di  \sigma + \int_{s^{k}_{i-1}}^{s} \mathcal{P} ( t_{k}(\sigma), u_{k}(\sigma), z_{k}(\sigma) ) \, t'_{k}(\sigma) \, \di \sigma  \,. \label{e.100} 
\end{align}
For every $j\geq 0$, we distinguish between  $s\in[s^{k}_{i,j},s^{k}_{i,j+\frac{1}{2}}]$ and $s\in[s^{k}_{i,j+\frac{1}{2}},s^{k}_{i,j+1}]$. In the first case we have $t_{k}'(s)=z_{k}'(s)=0$ and $\| u'_k(s) \|_{H^1}=1$ for a.e.~$s \in[s^{k}_{i,j},s^{k}_{i,j+\frac{1}{2}}]$. For $j=0$ we have $\underline{t}_{k}(s) = t^{k}_{i-1}$, $\underline{u}_{k}(s)= u^{k}_{i-1}$, $z_{k}(s) = z^{k}_{i-1}$, and $|\partial_{z}^{-} \F| ( \underline{t}_{k}(s), \underline{u}_{k}(s), z_{k}(s))=0$ again by~\eqref{e.77}. If $j\geq 1$, then $\underline{t}_{k}(s) = t^{k}_{i}$, $\underline{u}_{k}(s) = u^{k}_{i,j}$, $z_{k}(s)= z^{k}_{i,j}$, and $|\partial_{z}^{-} \F | ( \underline{t}_{k}(s), \underline{u}_{k}(s), z_{k}(s))=0$ by~\eqref{e.minz}. Hence, we rewrite~\eqref{e.72} as
\begin{align}
\F  & (t_{k}(s), u_{k}(s), z_{k}(s))  = \F ( t_{k}(s^{k}_{i,j}), u_{k} ( s^{k}_{i,j} ) , z_{k} ( s^{k}_{i,j}) ) - \int_{s^{k}_{i,j}}^{s} \!\! | \partial_{u} \F | ( t_{k}(\sigma), u_{k}( \sigma ), z_{k}( \sigma ) ) \, \| u'_{k}(\sigma) \|_{H^{1}}\, \di   \sigma  \nonumber  \\
& = \F ( t_{k}(s^{k}_{i,j}), u_{k}(s^{k}_{i,j}),z_{k}(s^{k}_{i,j}))  - \int_{s^{k}_{i,j}}^{s} \!\! |\partial_{u} \F | ( t_{k}(\sigma), u_{k} ( \sigma ), z_{k} ( \sigma ) ) \, \| u'_k(\sigma) \|_{H^1} \,\di  \sigma \label{e.101}  \\
&\quad - \int_{s^{k}_{i,j}}^{s} \!\! | \partial_{z}^{-} \F |( \underline{t}_{k}(\sigma), \underline{u}_{k}( \sigma ),z_{k}( \sigma ))  \,\di  \sigma + \int_{s^{k}_{i,j}}^{s} \mathcal{P} ( t_{k}(\sigma), u_{k}(\sigma), z_{k}(\sigma)) \, t'_{k}(\sigma) \,\di\sigma \nonumber \,. 
\end{align}
In the case $s\in [s^{k}_{i,j+\frac{1}{2}}, s^{k}_{i, j+1})$ we have $t_{k}'(s)=u_{k}'(s)=0$, $\| z'_k(s) \|_{L^2}=1$ for a.e.~$s \in[s^{k}_{i,j+\frac12},s^{k}_{i,j+1}]$, $\underline{t}_{k}(s) = t^{k}_{i}$, and~$\underline{u}_{k}(s) = u^{k}_{i,j+1}$. Then, we rewrite~\eqref{e.73} as
\begin{align}
\F & ( t_{k}(s), u_{k}(s), z_{k}(s))  = \F( t_{k}(s^{k}_{i,j+\frac{1}{2}}), u_{k}(s^{k}_{i,j+\frac{1}{2}}), z_{k}(s^{k}_{i,j+\frac{1}{2}}) ) - \! \int_{s^{k}_{i,j+\frac{1}{2}}}^{s} \!\!\!\!\!\!\!\!\! | \partial_{z}^{-} \F |( \underline{t}_{k}(\sigma), \underline{u}_{k}( \sigma ), z_{k}( \sigma ) ) \,\| z'_{k}(\sigma)\|_{L^{2}}\, \di  \sigma  \nonumber \\
& = \F( t_{k}( s^{k}_{i,j+\frac{1}{2}}), u_{k}(s^{k}_{i, j + \frac{1}{2}}), z_{k}(s^{k}_{i, j +\frac{1}{2}})) - \int_{s^{k}_{i,j+\frac{1}{2}}}^{s} \!\!\!\!\!\!\!\!\! |\partial_{u}\F| ( t_{k}(\sigma), u_{k}( \sigma ), z_{k}( \sigma ) ) \, \| u'_k(\sigma) \|_{H^1} \,\di  \sigma \label{e.102} \\
& \quad   -   \int_{s^{k}_{i,j+\frac{1}{2}}}^{s} \!\!\!\!\!\!\!\!\! | \partial_{z}^{-} \F |( \underline{t}_{k}(\sigma), \underline{u}_{k}( \sigma ),z_{k}( \sigma ) )\,\di  \sigma + \int_{s^{k}_{i, j+\frac{1}{2}}}^{s} \!\!\!\!\!\!\!\! \P ( t_{k}( \sigma ), u_{k}( \sigma ), z_{k}( \sigma ) ) \, t'_{k}( \sigma )\,\di \sigma \nonumber  \,.
\end{align}
Summing up~\eqref{e.100}-\eqref{e.102}, we deduce that for every $s\in[s^{k}_{i-1},s^{k}_{i})$ it holds
\begin{align*}
\F( t_{k}(s), u_{k}(s),  z_{k}(s)) = & \ \F( t_{k}(s^{k}_{i-1}), u_{k}(s^{k}_{i-1}), z_{k}(s^{k}_{i-1}) ) - \int_{s^{k}_{i-1}}^{s} \!\!\! |\partial_{u} \F | ( t_{k}(\sigma), u_{k}( \sigma ), z_{k}( \sigma )) \, \| u'_k(\sigma) \|_{H^1} \,\di  \sigma \\ 
& - \int_{s^{k}_{i-1}}^{s} \!\!\!\! |\partial_{z}^{-} \F | ( \underline{t}_{k}(\sigma), \underline{u}_{k}( \sigma ), z_{k}( \sigma )) \,\di  \sigma + \int_{s^{k}_{i-1}}^{s} \!\!\!\! \P ( t_{k}( \sigma ), u_{k}( \sigma ), z_{k}( \sigma ) ) \, t'_{k}( \sigma )\,\di \sigma  \,.
\end{align*}
Passing to the limit as $s\to s^{k}_{i}$ by Lemma \ref{l.lscFE} we get 
\begin{equation*}
\begin{split}
\F( t_{k}(s^{k}_{i}), u_{k}(s^{k}_{i}), & \ z_{k}(s^{k}_{i})) \leq \F ( t_{k}(s^{k}_{i-1}), u_{k}(s^{k}_{i-1}), z_{k}(s^{k}_{i-1})) - \int_{s^{k}_{i-1}}^{s^{k}_{i}} \!\!\! |\partial_{u} \F | ( t_{k}(\sigma), u_{k}( \sigma ), z_{k}( \sigma ))\, \| u'_k(s) \|_{H^1} \,\di  \sigma  \\
& - \int_{s^{k}_{i-1}}^{s^{k}_{i}} \!\!\! | \partial_{z}^{-} \F | ( \underline{t}_{k}(\sigma), \underline{u}_{k}( \sigma ), z_{k}( \sigma )) \,\di  \sigma + \int_{s^{k}_{i-1}}^{s^{k}_{i}} \!\!\!\! \P ( t_{k}( \sigma ), u_{k}( \sigma ), z_{k}( \sigma ) ) \, t'_{k}( \sigma )\,\di \sigma  \,.
\end{split}
\end{equation*}
Iterating the previous estimates we deduce for every $s\in[0,S]$
\begin{equation}\label{e.103}
\begin{split}
\F( t_{k}(s), u_{k}(s), & \ z_{k}(s)) \leq \F( 0, u_{0}, z_{0}) - \int_{0}^{s} \!\! |\partial_{u} \F | ( t_{k}(\sigma), u_{k}( \sigma ), z_{k}( \sigma ))\, \| u'_k(\sigma) \|_{H^1} \,\di  \sigma  \\
& - \int_{0}^{s} \!\! | \partial_{z}^{-} \F | ( \underline{t}_{k}(\sigma), \underline{u}_{k}( \sigma ), z_{k}( \sigma )) \,\di  \sigma + \int_{0}^{s} \!\! \P ( t_{k}( \sigma ), u_{k}( \sigma ), z_{k}( \sigma ) ) \, t_{k}'( \sigma ) \, \di  \sigma \,.
\end{split}
\end{equation} 
We take the liminf on the left-hand side of~\eqref{e.103} and use lower semicontinuity of the energy. We take the limsup on the right-hand side of~\eqref{e.103} and apply the same argument as in the proof of Proposition~\ref{p.3} for the first and the last integral, while we apply Fatou to the second integral. Thus we obtain
\begin{displaymath}
\begin{split}
\F( t(s), u (s),  z (s)) &\leq \F(0, u_{0}, z_{0}) - \liminf_{k\to\infty} \, \int_{0}^{s} \!\! |\partial_{u} \F | ( t_{k}(\sigma), u_{k}( \sigma ), z_{k}( \sigma ))\, \| u'_k(\sigma) \|_{H^1} \,\di  \sigma  \\
&\quad- \liminf_{k\to\infty}\, \int_{0}^{s} \!\!\! | \partial_{z}^{-} \F | (\underline{t}_{k}(\sigma), \underline{u}_{k}( \sigma ), z_{k}( \sigma )) \,\di  \sigma + \limsup_{k\to\infty} \, \int_{0}^{s} \!\! \mathcal{P}(t_{k}(\sigma), u_{k}(\sigma), z_{k}(\sigma))\,t'_{k}(\sigma) \, \di \sigma \\
& \leq \F( 0, u_{0}, z_{0}) - \int_{0}^{s} \!\! |\partial_{u} \F | ( t(\sigma), u ( \sigma ), z ( \sigma ))\, \| u' (\sigma) \|_{H^1} \,\di  \sigma \\
&\quad -  \int_{0}^{s} \!\!\! \liminf_{k\to\infty}\, | \partial_{z}^{-} \F | ( \underline{t}_{k}(\sigma), \underline{u}_{k}( \sigma ), z_{k}( \sigma )) \,\di \sigma  + \int_{0}^{s} \!\! \mathcal{P}(t (\sigma), u (\sigma), z (\sigma) )\,t' (\sigma) \, \di \sigma  \,.
\end{split}
\end{displaymath}
For every~$\sigma \in U$ we know thanks to Lemmata~\ref{p.nonstationaryz} and~\ref{p.almoststationaryz} that $\underline{u}_{k}(\sigma)\to u(\sigma)$ in~$H^{1}(\Om;\R^{2})$, while from Lemma~\ref{r.t} we get that $\underline{t}_{k} \to t$ pointwise in~$[0,S]$. Hence, by Lemma~\ref{l.2} we can continue in the previous inequality with
\begin{equation}\label{e.120}
\begin{split}
\F( t(s), & u (s),  z (s))  \leq  \F( 0, u_{0}, z_{0}) - \int_{0}^{s} \!\! |\partial_{u} \F | ( t (\sigma), u ( \sigma ), z ( \sigma ))\, \| u' (\sigma) \|_{H^1} \,\di  \sigma \\
& -  \int_{0}^{s} \!\!  | \partial_{z}^{-} \F | ( t(\sigma), u ( \sigma ), z( \sigma )) \, \mathbf{1}_{U}(\sigma) \, \di \sigma - \int_{0}^{s} \!\! \liminf_{k\to\infty} | \partial_{z}^{-} \F | ( \underline{t}_{k}(\sigma), \underline{u}_{k}( \sigma ), z_{k}( \sigma )) \, \mathbf{1}_{U^{c}} (\sigma) \, \di \sigma \\
& +\int_{0}^{s} \mathcal{P}(t(\sigma), u(\sigma), z(\sigma)) \, t'(\sigma) \, \di \sigma \,.
\end{split}
\end{equation}
Since $u \in W^{1,\infty} ([0,S]; H^{1}(\Om;\R^{2}))$, $g\in W^{1,q}([0,T]; W^{1,p}(\Om;\R^{2}))$ for some $p>2$ and $q>1$, and~$0 \leq z(s) \leq 1$ for every $s\in[0,S]$, the power functional $\mathcal{P} (t(\cdot), u(\cdot), z(\cdot))\, t'(\cdot)$ belongs to~$L^{1}(0,S)$. Therefore, being the energy functional~$\F$ and the slopes~$|\partial_{u} \F|$ and~$|\partial_{z}^{-} \F|$ positive, we deduce from~\eqref{e.120} that~$|\partial^{-}_{z} \F| ( t(\cdot), u(\cdot), z(\cdot))\,\mathbf{1}_{U}(\cdot) \in L^{1}(0,S)$.\qed
\end{proof}

\begin{lemma}{\bf (Riemann sum).} \label{l.L1-2} The function $r \mapsto | \partial_z^- \F | ( \tilde{t} (r), \tilde{u}(r) , \tilde{z}(r))$ belongs to $L^1(0,R)$. Moreover, for every $R' \in (0, R\,]$ there exists a sequence of subdivisions $\{ r^m_n , \text{ for $n=0,...,N_m$}\}$ with 
$$
	r^{m}_{0} = 0 , \quad  r^{m}_{N_{m}} = R' , \quad \lim_{m \to \infty} \big( \max_{n=0,...,N_m-1} \, ( r^m_{n+1} - r^m_{n} ) \,\big) \to 0 ,
$$
such that the simple functions
$$
	F^m (r) := \sum_{n=0}^{N_{m}-1} | \partial_z^- \F | ( \tilde{t}(r^m_n),  \tilde{u}(r^m_n) , \tilde{z}(r^m_n) ) 
	\left\| \frac{\tilde{z} (r^m_{n+1}) - \tilde{z} (r^m_n) }{r^m_{n+1} - r^m_{n}}   \right\|_{L^2} \! \mathbf{1}_{( r^m_n , \, r^m_{n+1})} (r)
$$
converge to $| \partial_z^- \F |( \tilde{t}(\cdot), \tilde{u}(\cdot) , \tilde{z}(\cdot)) \| \tilde{z}' (\cdot) \|_{L^2}$ strongly in $L^1(0,R')$ (as $m \to \infty$).
\end{lemma}

\begin{proof}
Since~$\beta\colon[0,S]\to[0,R]$ is Lipschitz continuous, surjective, and $\beta' = \mathbf{1}_{U}$ a.e.~in~$[0,S]$, by the change of variable formula we have for every Borel measurable function~$g\colon [0,S]\to[0,+\infty]$ 
\begin{displaymath}
\int_{[0,S]} g(\sigma) \, \mathbf{1}_{U} (\sigma)\,\di\sigma = \int_{[0,R]} \sum_{\sigma\in \beta^{-1} ( r)} g(\sigma)\,\di r\,.
\end{displaymath}
If~$r\in[0,R]\setminus\beta(U^{c})$, then $\{\sigma \in \beta^{-1}(r)\} = \{\alpha(r)\}$. Indeed, $\beta(\alpha(r))= r$ and if there exists $s > \alpha(r)$ such that $\beta(s)=r$, then $(\alpha(r), s]\subseteq U^{c}$ by~$(c)$ of Lemma~\ref{l.beta} and~$r\in\beta(U^{c})$, which is a contradiction. Since~$|\beta(U^{c})|=0$, from the previous equality we obtain
\begin{equation}\label{e.changeofvariable}
\int_{[0,S]} g(\sigma) \, \mathbf{1}_{U} (\sigma)\,\di\sigma =  \int_{[0,R]} g(\alpha(r))\,\di r\,.
\end{equation}

We now apply~\eqref{e.changeofvariable} to the function~$|\partial_{z}^{-}\F|(t(\cdot), u(\cdot), z(\cdot))$:
\begin{displaymath}
	  \int_{0}^{S} | \partial_z^- \F | ( t(s), u (s) , z (s)) \, \mathbf{1}_{U}(s) \, \di s = \int_{0}^{R} | \partial_z^- \F | ( \tilde{t}(r), \tilde{u}(r) , \tilde{z}(r)) \, \di r \,.
\end{displaymath}
Hence $| \partial_z^- \F |( \tilde{t}(\cdot), \tilde{u}(\cdot) , \tilde{z}(\cdot)) $ belongs to~$L^1(0,R')$ for every $R' \in (0,R\,]$ by Lemma~\ref{l.L1-1}. Thus, by classical results 
(see, e.g.,~\cite[Lemma~4.12]{MR2186036}) there exists a sequence of subdivisions $\{ r^m_n \}$ with 
$$
	r^{m}_{0} = 0 , \quad  r^{m}_{N_{m}} = R' , \quad \lim_{m \to \infty} \big( \max_{n=0,...,N_m-1} \, ( r^m_{n+1} - r^m_{n} ) \,\big) \to 0 ,
$$ 
such that the simple functions
$$
	F^m (r) := \sum_{n=0}^{N_{m}-1} | \partial_z^- \F | ( \tilde{t} (r^m_n), \tilde{u}(r^m_n) , \tilde{z}(r^m_n) ) \,   \mathbf{1}_{( r^m_n , \, r^m_{n+1})} (r) 
$$
converge to $| \partial_z^- \F |( \tilde{t} (\cdot), \tilde{u}(\cdot) , \tilde{z}(\cdot)) $ strongly in~$L^1 (0,R')$.

Invoking for instance \cite[Lemma D.1]{NegriKimura}, for a.e.~$r \in (0,R')$ it holds 
$$
    \sum_{n=0}^{N_{m}-1} \left\| \frac{\tilde{z} (r^m_{n+1}) - \tilde{z} (r^m_n) }{r^m_{n+1} - r^m_{n}}   \right\|_{L^2} \!  \mathbf{1}_{( r^m_n , \, r^m_{n+1})} (r) \to  \| \tilde{z}' (r) \|_{L^2} \, .
$$
The thesis follows by dominated convergence, since~$\| \tilde{z}' (r) \|_{2} \leq 1$ for a.e.~$r\in[0,R']$.\qed 
\end{proof}

We are now in a position to prove the upper energy-dissipation inequality.

\begin{proposition} \label{p.R-sum} Let $s \in (0,S]$ and $(t,u,z)$ be the triple obtained in Proposition~\ref{p.compactness}. Then,
\begin{displaymath}
\begin{split}
	\F ( t(s), u (s) , z(s) ) \ge & \ \F(0, u_0 , z_0) - \int_0^{s} | \partial^-_z \F | ( t(\sigma), u(\sigma) , z(\sigma) ) \| z' (\sigma) \|_{L^2} \, \di \sigma \\
	&-\int_{0}^{s}  | \partial_u \F | ( t(\sigma), u(\sigma) , z(\sigma) ) \| u' (\sigma) \|_{H^1} \, \di \sigma +\int_{0}^{s}\mathcal{P}(t(\sigma), u(\sigma), z(\sigma) ) \, t'(\sigma)\,\di \sigma \,.
\end{split}
\end{displaymath}
\end{proposition}

\proof We divide the proof in two steps.

{\bf Step 1: \boldmath{$s \in U $}.} Let $ R' = \beta(s)$. Since $s\in U$, then $R' > 0$. Let~$\{r^m_n\}$ a sequence of subdivision  of~$[0, R']$ provided by Lemma~\ref{l.L1-2}. We recall that $\tilde{u} (r) = u \circ \alpha( r)$ and $\tilde{z} (r) = z \circ \alpha( r)$. Thus, by the regularity of~$t$ and of~$u$ we can write by chain rule
\begin{align*}
	\F & ( \tilde{t}(r^{m}_{n+1}), \tilde{u} ( r^{m}_{n+1}) , \tilde{z} ( r^{m}_{n+1})   )  = \F ( t \circ \alpha (r^{m}_{n+1}), u \circ \alpha (r^{m}_{n+1}) , z \circ \alpha ( r^{m}_{n+1}) ) \\ 
	& = \F ( t \circ \alpha (r^m_n), u \circ \alpha (r^{m}_{n}) , z \circ \alpha ( r^{m}_{n+1}) )  + \int_{\alpha(r^m_n)}^{\alpha(r^{m}_{n+1})} \!\!\! \partial_u \F ( t(\sigma), u ( \sigma) , z \circ \alpha ( r^{m}_{n+1}  )  )  [ u' (\sigma) ] \, \di \sigma \\
	& \quad + \int_{\alpha(r^m_n)}^{\alpha(r^{m}_{n+1})} \mathcal{P}(t(\sigma), u(\sigma), z \circ \alpha (r^{m}_{n+1})) \, t'(\sigma) \, \di \sigma \\
	& = \F ( \tilde{t}(r^m_n), \tilde{u} (r^{m}_{n}) , \tilde{z} ( r^{m}_{n+1}) ) + \int_{\alpha(r^m_n)}^{\alpha(r^{m}_{n+1})} \!\!\! \partial_u \F ( t(\sigma), u ( \sigma) , z \circ \alpha ( r^{m}_{n+1}  )  )  [ u' (\sigma) ] \, \di \sigma\\
	& \quad + \int_{\alpha(r^m_n)}^{\alpha(r^{m}_{n+1})} \!\!\! \mathcal{P}(t(\sigma), u(\sigma), z \circ \alpha (r^{m}_{n+1})) \, t'(\sigma) \, \di \sigma \,.
\end{align*}
Using the convexity of the energy $\F ( t, u , \cdot)$ we can write 
\begin{align*}
	\F ( \tilde{t}(r^m_n), \tilde{u} ( r^{m}_{n}) , \tilde{z} ( r^{m}_{n+1}) ) 
			& \ge \F ( \tilde{t}(r^m_n), \tilde{u} ( r^m_n) , \tilde{z} ( r^{m}_{n}) ) 
			+ \partial_z \F ( \tilde{t}(r^m_n), \tilde{u} ( r^{m}_{n}) , \tilde{z} ( r^{m}_{n}) ) [  \tilde{z} (r^{m}_{n+1} )  -   \tilde{z} (r^{m}_{n} )  ] \\
			& \ge \F ( \tilde{t}(r^m_n), \tilde{u} ( r^m_n) , \tilde{z} ( r^{m}_{n}) ) 
			- | \partial^-_z \F | ( \tilde{t}(r^m_n), \tilde{u} ( r^{m}_{n}) , \tilde{z} ( r^{m}_{n}) ) \, \| \tilde{z} (r^{m}_{n+1} )  -   \tilde{z} (r^{m}_{n} ) \|_{L^2}  \\
			& = \F ( \tilde{t}(r^m_n), \tilde{u} ( r^m_n) , \tilde{z} ( r^{m}_{n}) ) 
			- \int_{r^m_n}^{r^{m}_{n+1}} \!\!\!\! | \partial^-_z \F | ( \tilde{t}(r^m_n), \tilde{u} ( r^{m}_{n}) , \tilde{z} ( r^{m}_{n}) ) \, \left\| \frac{ \tilde{z} (r^{m}_{n+1} )  -   \tilde{z} (r^{m}_{n} ) }{r^{m}_{n+1}-r^m_n} \right\|_{L^2} \!\!\!\! \di \rho  .
\end{align*}
In conclusion, for every index $n=0,...,N_{m}-1$ we have
\begin{align*}
   \F ( \tilde{t}(r^{m}_{n+1}),  \tilde{u} ( r^{m}_{n+1}) , \tilde{z} ( r^{m}_{n+1})   )  \ge & \ \F ( \tilde{t}(r^m_n), \tilde{u} ( r^m_n) , \tilde{z} ( r^{m}_{n}) ) \\
   & - \int_{r^m_n}^{r^{m}_{n+1}} \!\!\! | \partial^-_z \F | ( \tilde{t}(r^m_n), \tilde{u} ( r^{m}_{n}) , \tilde{z} ( r^{m}_{n}) ) \, \left\| \frac{ \tilde{z} (r^{m}_{n+1} )  -   \tilde{z} (r^{m}_{n} ) }{r^{m}_{n+1}-r^m_n} \right\|_{L^2} \!\!\!\! \di \rho \,  \\
	&  + \int_{\alpha(r^m_n)}^{\alpha(r^{m}_{n+1})} \!\! \partial_u \F ( t(\sigma), u ( \sigma) , z \circ \alpha ( r^{m}_{n+1}  )  )  [ u' (\sigma) ] \, \di \sigma \\
	&  + \int_{\alpha(r^m_n)}^{\alpha(r^{m}_{n+1})} \!\!\! \mathcal{P}(t(\sigma), u(\sigma), z \circ \alpha (r^{m}_{n+1})) \, t'(\sigma) \, \di \sigma  \,.
\end{align*}
Note that $\tilde{u} (r^m_0) = \tilde{u} (0) = u_0$ and that $\alpha(R') =  \alpha(\beta(s)) = s$ because $s \in U$. Thus, 
\begin{displaymath}
\tilde{u} (r^m_{N_{m}}) = \tilde{u} ( R' ) =  u \circ \alpha (R') = u ( s)\,.
\end{displaymath}
In a similar way, $\tilde{z} (r^m_0) =  z_0$ and $\tilde{z} ( r^m_{N_{m}})= z(s)$. Therefore, iterating the previous inequality for $n=0,...,N_{m}-1$ yields
\begin{align}
   \F ( t(s), u (s) , z(s) ) & \ge \F ( 0, u_0 , z_0 ) 
			- \sum_{n=0}^{N_{m}-1} \int_{r^m_n}^{r^{m}_{n+1}} | \partial^-_z \F | ( \tilde{t}(r^m_n), \tilde{u} ( r^{m}_{n}) , \tilde{z} ( r^{m}_{n}) ) \, \left\| \frac{ \tilde{z} (r^{m}_{n+1} ) -   \tilde{z} (r^{m}_{n} ) }{r^{m}_{n+1}-r^m_n} \right\|_{L^2} \!\!\!\! \di \rho  \nonumber \\
	& \quad + \sum_{n=0}^{N_{m}-1} \int_{\alpha(r^m_n)}^{\alpha(r^{m}_{n+1})} \!\! \partial_u \F ( t(\sigma), u ( \sigma) , z \circ \alpha ( r^{m}_{n+1}  )  )  [ u' (\sigma) ]  \, \di \sigma \label{e.marameo} \\
	& \quad + \sum_{n=0}^{N_{m}-1}  \int_{\alpha(r^m_n)}^{\alpha(r^{m}_{n+1})} \!\!\! \mathcal{P}(t(\sigma), u(\sigma), z \circ \alpha (r^{m}_{n+1})) \, t'(\sigma) \, \di \sigma \nonumber  \,.
\end{align}

We now pass to the limit as~$m\to\infty$. By Lemma \ref{l.L1-2} we know that the first sum in~\eqref{e.marameo} converges to 
$$
	\int_0^{R'} | \partial_z^- \F |( \tilde{t}(\rho), \tilde{u}( \rho ) , \tilde{z}(\rho)) \| \tilde{z}' (\rho) \|_{L^2} \, \di \rho .
$$
By the change of variable formula~\eqref{e.changeofvariable} with $g(\sigma)= |\partial_{z}^{-} \F| (t(\sigma), u(\sigma), z(\sigma)) \|z'(\sigma)\|_{L^{2}}$, and recalling the definition of~$\tilde{t}$,~$\tilde{u}$,~$\tilde{z}$ and that~$\tilde{z}'= z' \circ \alpha$ a.e.~in $[0,R]$ by Lemma~\ref{l.tildez}, we get
\begin{align}
	\int_0^{R'} | \partial_z^- \F |( \tilde{t}(\rho), \tilde{u}( \rho ) , \tilde{z}(\rho)) \| \tilde{z}' (\rho) \|_{L^2} \, \di \rho  & = \int_0^s | \partial_z^- \F |( t(\sigma), u ( \sigma ) , z (\sigma)) \| z' (\sigma) \|_{L^2} \,\mathbf{1}_{U} ( \sigma) \, \di \sigma  \nonumber \\
	 & = \int_0^s | \partial_z^- \F |( t(\sigma), u ( \sigma ) , z (\sigma)) \| z' (\sigma) \|_{L^2} \, \di \sigma \,,  \label{e.cucu}
\end{align}
where in the last equality we have used the fact that $z'=0$ in $[0,S] \setminus U$, and hence $| \partial_z^- \F | \| z' \|_{L^2}=0$. 

We claim that the second and the third sums in~\eqref{e.marameo} converge to
\begin{align}
	\int_0^s  \partial_u \F ( t(\sigma), u ( \sigma ) , z (\sigma)) [ u' (\sigma) ] \, \di \sigma \qquad \text{and} \qquad \int_{0}^{s} \mathcal{P}( t(\sigma), u(\sigma), z(\sigma))\, t'(\sigma) \, \di \sigma\,, \label{e.uccu}
\end{align}
respectively. We notice that if the claim holds, then, passing to the limit in~\eqref{e.marameo} as~$m\to\infty$ and using~\eqref{e.cucu} we would get
\begin{displaymath}
\begin{split}
	\F ( t(s), u (s) , z(s) ) \ge & \ \F( 0, u_0 , z_0) - \int_0^{s}  | \partial_u \F | ( t(\sigma), u(\sigma) , z(\sigma) ) \| u' (\sigma) \|_{H^1} \, \di \sigma \\
	& - \int_{0}^{s} | \partial^-_z \F | ( t(\sigma), u(\sigma) , z(\sigma) ) \|   z' (\sigma) \|_{L^2} \, \di \sigma + \int_{0}^{s} \mathcal{P}(t(\sigma), u(\sigma), z(\sigma)) \, t'(\sigma) \, \di \sigma \,.
\end{split}
\end{displaymath}
Let us prove the claim. Fix $\bar\sigma \in (0, s)$ and let $j$ (depending on $\bar{\sigma}$ and $k$) be such that $\bar\sigma \in [\alpha(r^m_n) , \alpha (r^{m}_{n+1}) )$. Note that, being $\alpha$ discontinuous, it may happen that $\alpha (r^{m}_{n+1}) - \alpha (r^m_n) \not\to 0$. However, we can write 
$$
	z ( \bar\sigma) = z ( \alpha (r^m_n) ) + \int_{\alpha(r^m_n)}^{\bar\sigma} z' ( \sigma) \, \di \sigma
$$
and thus 
\begin{align*}
	\| z ( \bar \sigma) - z ( \alpha (r^m_n) ) \|_{L^2} & \le \int_{\alpha(r^m_n)}^{\bar\sigma} \| z' ( \sigma) \|_{L^2} \, \di \sigma = \int_{\alpha(r^m_n)}^{\bar\sigma} \| z' ( \sigma) \|_{L^2}\, \mathbf{1}_{U}(\sigma) \, \di \sigma \\
	& \le \int_{\alpha(r^m_n)}^{\alpha(r^{m}_{n+1})} \,\mathbf{1}_{U}(\sigma) \, \di \sigma = \beta \circ \alpha(r^{m}_{n+1}) - \beta \circ \alpha(r^m_n)= r^{m}_{n+1} - r^m_n \to 0 ,
\end{align*}
where in the last limit we have used the property of the subdivision $r^m_n$. Arguing in the same way, we also prove that $\|z(\alpha(r^{m}_{n+1})) - z(\alpha(r^m_n))\|_{2} \to 0$ as $k\to \infty$, which implies that $\| z ( \bar{\sigma} ) - z ( \alpha ( r^{m}_{n+1} ) ) \|_{2} \to 0$ as well. Hence, we have shown that the sequence $\sum_{n=0}^{N_{m}-1} z(\alpha(r^{m}_{n+1}))\,\mathbf{1}_{[\alpha(r^m_n), \alpha(r^{m}_{n+1}))}$ converges pointwise to~$z$ in~$L^{2}(\Om)$.

Recall that 
$$
	\partial_{u} \F( t(\sigma), u (\sigma) , z \circ \alpha (r^{m}_{n+1})) [u'(\sigma)] = 
	\int_{\Om} \partial_{\strain} W \big(z \circ \alpha (r^{m}_{n+1}) ,\strain(u (\sigma) + g \circ t (\sigma) )  \big){\,:\,}\strain(u'(\sigma))\,\di x .
$$
By $(c)$ in Lemma \ref{l.HMWw} we have $| \partial_{\strain} W \big(z \circ \alpha (r^{m}_{n+1}),\strain(u (\sigma) + g \circ t (\sigma) )  \big) | \le  C | \strain(u (\sigma) + g \circ t (\sigma) )   | $. Let us consider a subsequence (not relabelled) such that $z \circ \alpha (r^{m}_{n+1}) \to z (\sigma)$ a.e.~in $\Omega$. Then, being $W$ of class $C^1$, 
$$
\partial_{\strain} W \big(z \circ \alpha (r^m_{n+1}) ,\strain(u (\sigma) + g \circ t (\sigma) )  \big) \to \partial_{\strain} W \big(z ( \sigma) ,\strain(u (\sigma) + g \circ t (\sigma) )  \big) \quad \text{a.e.~in $\Omega$.}
$$
By dominated convergence 
$$
\partial_{u}\F( t(\sigma), u (\sigma) ,z \circ \alpha (r^{m}_{n+1}) )[u'(\sigma)] \to \partial_{u}\F ( t(\sigma), u (\sigma) ,z (\sigma))[u'(\sigma)] \quad\text{for a.e.~$\sigma\in[0,s]$}.
$$

Applying again dominated convergence (in the integral over~$[0,s]$) we prove the first part of the claim~\eqref{e.uccu}, i.e.,
\begin{displaymath}
\lim_{m\to\infty}  \sum_{n=0}^{N_{m}-1} \int_{\alpha(r^m_n)}^{\alpha(r^{m}_{n+1})} \!\! \partial_u \F ( t(\sigma), u ( \sigma) , z \circ \alpha ( r^{m}_{n+1}  )  )  [ u' (\sigma) ]  \, \di \sigma = \int_{0}^{s} |\partial_{u} \F| (t(\sigma), u(\sigma), z(\sigma))\, [u'(\sigma)]\, \di \sigma\,.
\end{displaymath}

Following the proof of~\eqref{e.91} we also obtain the second part of the claim~\eqref{e.uccu}, that is,
\begin{displaymath}
\lim_{m\to\infty} \sum_{n=0}^{N_{m}-1}  \int_{\alpha(r^m_n)}^{\alpha(r^{m}_{n+1})} \!\!\! \mathcal{P}(t(\sigma), u(\sigma), z \circ \alpha (r^{m}_{n+1})) \, t'(\sigma) \, \di \sigma = \int_{0}^{s} \mathcal{P}(t(\sigma), u(\sigma), z(\sigma))\, t'(\sigma) \, \di \sigma \,.
\end{displaymath}

{\bf Step 2: \boldmath{ $s \in U^{c}$.}} In this case $s \in ( s^-_i , s^+_i]$ for some index $i \in \N$. 
In the interval $[ s^{-}_{i} , s]$ we have $z(\sigma) = z (s^-_i)$ and~$z'(\sigma)=0$, while~$t$ and~$u$ are of class $W^{1,\infty}$. Thus, we can write
\begin{align*}
	\F ( t(s), u ( s) , z (s)) & = \F ( t(s),  u(s) , z(s^-_i) ) = \F ( t(s^{-}_{i}) , u (s^-_i) , z (s^-_i) ) 
		+ \int_{s^-_i}^{s} \partial_u \F ( t(\sigma), u ( \sigma) , z(s^-_i)) [ u' ( \sigma) ] \, \di  \sigma  \\
		& \quad +\int_{s^{-}_{i}}^{s} \mathcal{P} (t(\sigma), u(\sigma), z(s^{-}_{i}))\, t'(\sigma)\, \di \sigma \\
		& \ge \F ( t(s^{-}_{i}), u (s^-_i) , z (s^-_i) ) 
		- \int_{s^-_i}^{s} | \partial_u \F | ( t(\sigma), u ( \sigma) , z(\sigma)) \, \| u' ( \sigma) \|_{H^1} \, \di \sigma \\
		&\quad -\int_{s^{-}_{i}}^{s} | \partial^-_z \F | ( t(\sigma), u(\sigma) , z(\sigma) ) \, \| z' (\sigma) \|_{L^2} \, \di \sigma + \int_{s^{-}_{i}}^{s} \mathcal{P}(t(\sigma), u(\sigma), z(\sigma)) \, t'(\sigma)\, \di \sigma \,. 
\end{align*}
Since $s^-_i \in U$ we can apply the previous step and we conclude the proof.
\qed


\appendix

\section{Comparing different parametrizations}

In this appendix we will compare, qualitatively, the evolutions of Theorem~\ref{t.1} with those of \cite[Theorem~4.2]{KneesNegri_M3AS17}, or, more precisely, we will compare the evolutions obtained here, employing $H^1$-norm for~$u$ and $L^2$-norm for~$z$, with those obtained employing energy norms. As we will see, the evolutions will be qualitatively the same (up to subsequences) even if these norms are not equivalent.

We need to consider the setting of~\cite{KneesNegri_M3AS17}, otherwise energy norms would not be defined. Let $\mathcal{J} \colon [0,T] \times H^1_0 (\Omega , \mathbb{R}^2) \times H^1 (\Omega; [0,1]) \to [0,+\infty)$ given by 
$$
	\mathcal{J} (t,u,z) = \tfrac12 \int_\Omega (z^2 + \eta) \stress (u + g(t)){\,:\,}  \strain(u + g(t)) \, \di x + \tfrac12 \int_\Omega |\nabla z |^2 + (z-1)^2 \, \di x ,
$$
where we assume that the boundary datum~$g$ belongs to $C^{1,1}([0,T]; W^{1,p}(\Om;\R^{2}))$, $p>2$.
Note that this energy is separately quadratic, thus it is natural, and technically convenient, to introduce a couple of energy (instrinsic) norms: 
$$
	\| u \|^2_{z} = \int_\Omega (z^2 + \eta) \stress(u){\,:\,} \strain(u)\, \di x ,
	\qquad
	\| z \|^2_{u} = \int_\Omega | \nabla z |^2 + z^2 ( 1 + \stress(u){\,:\,} \strain(u)) \, \di x ,
$$
which correspond, respectively, to the quadratic part of the energies $\mathcal{J}(t, \cdot, z)$ and $\mathcal{J}(t, u, \cdot)$. Accordingly, we employ the slopes
\begin{align*}
|\partial_{u} \mathcal{J}|_z (t,u,z) &= \max \,\{-\partial_{u}\mathcal{J}(t,u,z)[\varphi]:\,\varphi \in H^1_0(\Omega ; \R^2),\, \|\varphi\|_{z}\leq 1\}\,,\\[1mm]
|\partial_{z}^{-} \mathcal{J}|_u (t,u,z) &= \max \,\{-\partial_{z} \mathcal{J}(t,u,z)[\xi]:\,\xi \in H^1(\Omega),\,\xi\leq 0,\,\|\xi\|_{u}\leq 1\} \,.
\end{align*}

Let us consider again the alternate scheme (at time $t^k_i$)
\begin{eqnarray*}
&&\displaystyle u^{k}_{i,j+1}:=\argmin \{\mathcal{J}(t^{k}_{i},u,z^{k}_{i,j}):\,u\in \U \} , \\[2mm]
&&\displaystyle z^{k}_{i,j+1}:=\argmin \{\mathcal{J}(t^{k}_{i},u^{k}_{i,j + 1},z):\, z\in \Z,\, z\leq z^{k}_{i,j}\}.
\end{eqnarray*}
We remark that, given $\tau_k$, the families $u^{k}_{i,j}$ and $z^k_{i,j}$ are uniquely determined (by strict separate convexity of the energy). 
Following \cite{KneesNegri_M3AS17}, we interpolate and parametrize the discrete configurations $u^k_{i,j}$ and $z^k_{i,j}$ with respect to the energy norms $\| \cdot \|^2_{z}$ (for the displacement field) and $\| \cdot \|^2_{u}$ (for the phase field). We remark that in this case it is enough to consider piece-wise affine interpolation, which actually coincides, for both $u$ and $z$,  with the gradient flow in the energy norm. As a result, we get a sequence of arc-length parametrizations $ ( \bar{t}_k , \bar{u}_k, \bar{z}_k)$, bounded  in $W^{1,\infty} ( [ 0, R] ; [0,T] \times H^1_0(\Omega ; \mathbb{R}^2) \times H^1 (\Omega))$ 
and of uniformly finite length, i.e., with $R$ independent of $k \in \mathbb{N}$. 
Invoking \cite[Lemma 4.3]{KneesNegri_M3AS17}, there exists a subsequence (non relabelled) and a limit $( \bar{t}, \bar{u}, \bar{z})$ in $W^{1,\infty} ( [ 0, R] ; [0,T] \times H^1_0(\Omega ; \mathbb{R}^2) \times H^1 (\Omega))$ such that for every sequence~$r_{k}$ converging to~$r \in[0,R]$ we have
\begin{equation} \label{e.barconv}
\bar{t}_{k}(r_{k})\to \bar{t} (r)\,,\qquad \bar{u}_{k}(r_{k})\to \bar{u}(r) \text{ in~$H^1_0(\Omega; \R^2)$,}\qquad \bar{z}_{k}(r_{k}) \rightharpoonup \bar{z}(r) \text{ weakly in~$H^{1}(\Om)$}.
\end{equation}
Moreover, invoking~\cite[Theorem 4.2]{KneesNegri_M3AS17} the limit evolution satisfies the following properties:
\begin{itemize}
\item [$(\bar{a})$] \emph{Regularity}: $(\bar{t},\bar{u},\bar{z})\in W^{1,\infty}([0,R];[0,T]\times H^1_0 ( \Omega;\R^{2}) \times H^{1}(\Om; [0,1]))$, and for a.e.~$r\in[0,R]$ 
\begin{displaymath}
\bar{t\,}'(r)+\| \bar{u}'(r)\|_{ \bar{z}(r)}+\| \bar{z}'(r)\|_{\bar{u} (r)}\leq 1\,,
\end{displaymath}
here the symbol~$'$ denotes the derivative w.r.t.~the parametrization variable $r$;

\item [$(\bar{b})$] \emph{Time parametrization}: the function $\bar{t}\colon [0,R]\to[0,T]$ is non-decreasing and surjective; 

\item [$(\bar{c})$] \emph{Irreversibility}: the function $\bar{z}$ is non-increasing and $0 \le \bar{z}(r) \leq 1$ for every $0\leq r \leq R$;

\item [$(\bar{d})$] \emph{Equilibrium}: for every continuity point $r\in[0,R]$ of~$(\bar{t},\bar{u},\bar{z})$
\begin{displaymath}
|\partial_{u} \mathcal{J}|_{\bar{z}(r)} (\bar{t}(r),\bar{u}(r),\bar{z}(r))=0\qquad\text{and}\qquad|\partial_{z}^{-} \mathcal{J}|_{\bar{u}(r)} (\bar{t}(r),\bar{u}(r),\bar{z}(r))=0 \,;
\end{displaymath}

\item [$(\bar{e})$] \emph{Energy-dissipation equality}: for every $r \in[0,R]$
\begin{equation}\label{e.eneqR}
\begin{split}
\mathcal{J}(\bar{t}(r),& \ \bar{u}(r), \bar{z}(r))  =  \mathcal{J}(0,u_0,z_0)-\int_{0}^{r}|\partial_{u} \mathcal{J}|_{\bar{z}(\rho)} (\bar{t}(\rho),\bar{u}(\rho),\bar{z}(\rho))\, \| \bar{u}' (\rho) \|_{\bar{z}(\rho)} \di\rho\\
&-\int_{0}^{r} \!\! |\partial_{z}^{-} \mathcal{J}|_{\bar{u}(\rho)} (\bar{t}(\rho),\bar{u}(\rho),\bar{z}(\rho))\, \| \bar{z}' (\rho) \|_{\bar{u}(\rho)} \di\rho +\int_{0}^{r} \!\! \P( \bar{t}(\rho),\bar{u}(\rho), \bar{z}(\rho)) \,\bar{t\,} '(\rho)\,\di\rho\,.
\end{split}
\end{equation}
\end{itemize}
In~\cite{KneesNegri_M3AS17} the authors showed property~$(\bar{d})$ for every $r\in[0,R]$ with $\bar{t}'(r)>0$. However, it is not difficult to see that the same equilibrium condition is verified at continuity points.

\noindent Moreover, by \cite[Proposition 4.1]{KneesNegri_M3AS17} we have 
\begin{itemize}
\item [$(\bar{f})$] {\it Non-degeneracy:} there exists $C>0$ such that for a.e.~$r\in[0,R]$ 
\begin{equation} \label{e.nondegR}	
	C < \bar{t\,}'(r) + \| \bar{u}'(r) \|_{\bar{z}(r)} +  \| \bar{z}'(r) \|_{\bar{u}(r)} .
\end{equation}
\end{itemize}
Finally, note that, by the separate differentiability of the energy, the equilibrium conditions $(d)$ can be written in an equivalent ``norm-free'' fashion as 
\begin{itemize}
\item [$(\bar{d}')$] \emph{Equilibrium}: for every continuity point $r\in[0,R]$ of $(\bar{t},\bar{u},\bar{z})$
\begin{displaymath}
 	\partial_{u} \mathcal{J} (\bar{t}(r), \bar{u}(r), \bar{z}(r)) [ \varphi] = 0 , 
		\qquad \text{and} \qquad 
	\partial_{z} \mathcal{J} (\bar{t}(r),\bar{u}(r),\bar{z}(r)) [ \xi ] = 0 ,
\end{displaymath}
for every $\varphi \in H^1_0 ( \Omega ; \mathbb{R}^2)$ and every $\xi \in H^1(\Omega)$ with $\xi \leq 0$. 
\end{itemize}
At this point, consider the subsequence (not relabelled) converging to the limit~$(t,u,z)$ and let us re-interpolate the discrete configurations~$u^k_{i,j}$ and~$z^k_{i,j}$ with respect to the norms~$\| \cdot \|_{H^1}$ (for the displacement field) and $\| \cdot \|_{L^2}$ (for the phase field) as we did in Section~\ref{s.4.1}. In this way we get a new sequence of parametrizations~$(t_k , u_k ,z_k)$ bounded in $W^{1,\infty} ( [ 0, S] ; [0,T] \times H^1_0(\Omega ; \mathbb{R}^2) \times L^2 (\Omega))$. 
Clearly, we can apply Proposition~\ref{p.compactness} which provides (up to a further subsequence) a limit parametrization $(t , u ,z) \in W^{1,\infty} ( [ 0, S] ; [0,T] \times H^1_0(\Omega ; \mathbb{R}^2) \times L^2 (\Omega) )$ such that 
\begin{displaymath}
t_{k}(s_{k})\to t(s)\,,\qquad u_{k}(s_{k})\to u(s) \text{ in~$H^1(\Omega; \R^2)$,}\qquad z_{k}(s_{k}) \rightharpoonup z(s) \text{ weakly in~$H^{1}(\Om)$},
\end{displaymath}
for every sequence~$s_{k}$ converging to~$s\in[0,S]$. The limit $(t,u,z)$ satisfies properties $(a)$-$(e)$ of Theorem~\ref{t.1}. 

\separe

We recall that $(t_k, u_k, z_k)$ is defined in the points~$s^k_{i,j}$ and~$s^k_{i,j+\frac12}$ (see Section~\ref{s.4.1}). In a similar way, the interpolation $(\bar{t}_k, \bar{u}_k, \bar{z}_k)$ is defined in points of the form~$r^k_{i,j}$ and~$r^k_{i,j+\frac12}$ (see Section~4.3 in~\cite{KneesNegri_M3AS17}). Moreover, we notice that the interpolation nodes are different since the underlying parametrizations are different. However, the configurations computed by the alternate minimization scheme are the same. Therefore, we have that
\begin{displaymath}
t_{k}(s^{k}_{i,j}) = \bar{t}_{k}(r^{k}_{i,j})\, , \quad u_{k}(s^{k}_{i,j}) = \bar{u}_{k}(r^{k}_{i,j}) \,, \quad z_{k}(s^{k}_{i,j}) = \bar{z}_{k}(r^{k}_{i,j})\,.
\end{displaymath}
The same holds for nodes of the form $s^{k}_{i,j+\frac{1}{2}}$ and $r^{k}_{i,j+\frac{1}{2}}$.

Since~$(\bar{t}_k, \bar{u}_k , \bar{z}_k)$ is piecewise affine while $(t_k , u_k ,z_k)$ is not, a direct comparison of the triples~$(\bar{t}, \bar{u},\bar{z})$ and~$(t,u,z)$ is not immediate. Nevertheless, we can show the following ``equivalence'' of the reparametrizations.

\begin{lemma} \label{l.s<r}
There exist  two positive constants $C_{1}, C_{2}$ such that for every $k\in\mathbb{N}\setminus\{0\}$ and every $i\in\{1,\ldots, k\}$ 
\begin{equation} \label{e.s<r}
C_{1} (s^{k}_{i+1} - s^{k}_{i}) \leq r^{k}_{i+1} - r^{k}_{i} \leq C_{2} (s^{k}_{i+1} - s^{k}_{i}) \,.
\end{equation}
\end{lemma}

\begin{proof}
 Using the fact that $\| \cdot \|_z$ and $\| \cdot \|_{H^1}$ are equivalent, by Korn's inequality, while~$\|\cdot\|_{u}$ and~$\|\cdot\|_{H^{1}}$ are equivalent by~\cite[Lemma~2.3]{KneesNegri_M3AS17}, by \eqref{e.74} we can write 
\begin{align*}
	s^k_{i+1} - s^k_i 
		& = \tau_k + \sum_{j=0}^{\infty} L ( \zeta^k_{i,j} ) + L ( w^k_{i,j})  \le \tau_k + C \sum_{j=0}^{\infty} \| z^k_{i,j} - z^k_{i,j+1} \|_{H^1} + \| u^k_{i,j} - u^k_{i,j+1} \|_{H^1} \\
		& \le \tau_k + C' \sum_{j=0}^{\infty}  r^k_{i,j+1} - r^i_{i,j} \le C' ( r^k_{i+1} - r^k_i ) \,.
\end{align*}

On the other hand, by Proposition~\ref{p.5} and Corollary~\ref{c.3} we have that
\begin{displaymath}
\| z^k_{i,j} - z^k_{i,j+1} \|_{H^{1}} \leq \left\{ \begin{array}{ll}
 C \| u^k_{i,j} - u^k_{i,j+1} \|_{H^1} & \text{if $j\geq 1$}\,,\\[1mm]
 C\tau_{k} & \text{if $j=0$}\,.
\end{array}\right.
\end{displaymath} 
Hence, again by equivalence of norms we get
\begin{align*}
	r^k_{i+1} - r^k_i 
		& =  \tau_k + C \sum_{j=0}^{\infty} \| z^k_{i,j} - z^k_{i,j+1} \|_{u^{k}_{i,j+1}} + \| u^k_{i,j} - u^k_{i,j+1} \|_{z^{k}_{i,j}} \\
		& \le C' \Big( \tau_k + \sum_{j=0}^{\infty}  \| u^k_{i,j} - u^k_{i,j+1} \|_{H^{1}} \Big)\le C' ( s^k_{i+1} - s^k_i ) \,.
\end{align*}
This concludes the proof of~\eqref{e.s<r}.\qed
\end{proof}


 Let us consider an interval of the form $(s^k_{i_k,j_k+\frac12} , s^k_{i_k,j_k+1}) \subset (0,S)$ and the corresponding interval $(r^k_{i_k,j_k+\frac12} , r^k_{i_k,j_k+1}) \subset (0,R)$. By definition, we have 
$$
	t^k_{i_k} = \bar{t}_k (r) = t_k (s) 
	\quad \text{ and } \quad 
	u^k_{i_k,j_k+1} = \bar{u}_k (r) = u_k ( s) ,
$$
for every $r \in (r^k_{i_k,j_k+\frac12} , r^k_{i_k,j_k+1})$ and every $s \in (s^k_{i_k,j_k+\frac12} , s^k_{i_k,j_k+1})$. On the contrary, the phase field interpolations coincide only in the extrema, i.e.,
\begin{equation} \label{e.zuzzu}
	\bar{z}_k (r^k_{i_k,j_k+\frac12}) = z_k (s^k_{i_k,j_k+\frac12}) = z^k_{i,j}
	\quad \text{ and } \quad 
	\bar{z}_k (r^k_{i_k,j_k+1}) = z_k (s^k_{i_k,j_k+1}) = z^k_{i,j+1}.
\end{equation}

Now, up to subsequences (non relabelled), we can assume that, as $k \to \infty$,
$$
	s^k_{i_k,j_k+\frac12} \to s^- , \qquad s^k_{i_k,j_k+1} \to s^+ ,  \qquad r^k_{i_k,j_k+\frac12} \to r^- , \qquad r^k_{i_k,j_k+1} \to r^+ .
$$
Since parametrizations are different, in general we should distinguish between all the following cases: $s^- = s^+$, $s^- < s^+$, $r^- = r^+$, and $r^- < r^+$; however the situation is much simpler, thanks to the following lemma. 

\begin{lemma} \label{l.s=r} We have $r^- = r^+$ if and only if $s^- = s^+$.
\end{lemma}

\begin{proof} Assume that $r^- = r^+$. By compactness, we know that $\bar{z}_k ( r^k_{i_k,j_k+\frac12} ) = z^k_{i_k,j_k} \weakto z ( r^-)$ weakly in~$H^{1}(\Om)$ and that $\bar{z}_k ( r^k_{i_k,j_k+1} ) = z^k_{i_k,j_k+1} \weakto z ( r^+)$ weakly in~$H^{1}(\Om)$. Since $r^- = r^+$, we have $\bar{z}(r^-) = \bar{z}(r^+)$ and 
$$\| z^k_{i_k,j_k} - z^k_{i_k,j_k+1} \|_{H^1} = \big( r^k_{i_k,j_k+1}  - r^k_{i_k,j_k+\frac12}  \big) \to 0 . $$
By \eqref{e.16.06} we know that 
\begin{align*}
   \big( s^k_{i_k,j_k+1}  - s^k_{i_k,j_k+\frac12} \big) = L (\zeta^k_{i,j})  \le
C \| z^k_{i_k,j_k} - z^k_{i_k,j_k+1} \|_{H^1} .
\end{align*}
Thus $s^- = s^+$. 

Assume that $s^- = s^+$. Hence, arguing as above, by \eqref{e.zuzzu} we have $z (s^-) = \bar{z} (r^-) = z(s^+) = \bar{z} (r^+)$. Moreover, being $\bar{t}_k$, $t_k$, $\bar{u}_k$ and $u_k$ constant in the corresponding intervals, in the limit we have $t (s^-) = \bar{t} (r^-) = t(s^+) = \bar{t} (r^+)$ and $u (s^-) = \bar{u} (r^-) = u(s^+) = \bar{u} (r^+
)$. Then, if $r^- < r^+$ we would contradict the non-degeneracy condition \eqref{e.nondegR}. \qed \end{proof} 


As a consequence of Lemma~\ref{l.s<r}, the solutions $(\bar{t}, \bar{u}, \bar{z})$ and $(t,u,z)$ coincide in continuity points.


\begin{proposition} \label{p.z=barz} Let $r$ be a continuity point for~$(\bar{t}, \bar{u}, \bar{z})$. Then, there exists a continuity point~$s$ for~$(t,u,z)$ such that $(\bar{t} (r), \bar{u} (r), \bar{z}(r) ) = (t (s), u (s), z (s))$.

Viceversa, if $s$ is a continuity point for~$(t,u,z)$, then there exists a continuity point~$r$ for~$(\bar{t},\bar{u},\bar{z})$ such that $(t (s), u (s), z (s)) = (\bar{t} (r), \bar{u} (r), \bar{z}(r) )$.
\end{proposition}

\begin{proof} Fix $\delta >0$. Since $r$ is a continuity point, by monotonicity of~$\bar{t}$ we have $\bar{t}(r + \delta) > \bar{t} (r)$. Since~$\bar{t}_k$ converges pointwise to~$\bar{t}$, we have 
$ \bar{t}_k (r + \delta ) - \bar{t}_k (r) \ge \tfrac{1}{2} ( \bar{t}(r + \delta) - \bar{t} (r)) > 0$ for every $k$ sufficiently large. As~$\bar{t}_k$ changes only in parametrization intervals of the form $(r^k_{i,-1} , r^k_{i,0})$ and since $\tau_k \to 0$, there exist two  indexes~$i_k< i'_{k}$ such that, 
\begin{equation}\label{e.stolemma}
	r^{k}_{i_{k}-1}\leq r < r^k_{i_k} < r^k_{i_k,0}  \leq \ldots \leq r^{k}_{i'_{k}} \leq  r + \delta \leq  r^{k}_{i'_{k}+1}\,.
\end{equation}
Hence for every index $j \in \N$, we have 
$$
	r < r^k_{i_k,j+ \frac12} \leq  r^k_{i_k,j+1} \leq r^k_{i_k+1,0} < r + \delta .
$$
Since $\delta$ can be arbitrarily small, we can find a sequence $(i_{k},j_{k})$ such that
$$
	r^k_{i_k,j_k+ \frac12} \to r \quad \text{ and } \quad r^k_{i_k,j_k + 1} \to r \,.
$$
Then $\bar{z}_{k} ( r^k_{i_k,j_k + 1} ) \weakto \bar{z} (r)$ weakly in $H^1(\Omega)$. Since, by construction, $\bar{z}_{k} ( r^k_{i_k,j_k + 1} ) = z_{k} ( s^k_{i_k,j_k + 1} ) $ we have $z _{k}( s^k_{i_k,j_k + 1} ) \weakto \bar{z} (r)$ weakly in~$H^1(\Omega)$. Up to a subsequence (not relabelled) $s^k_{i_k,j_k + 1} \to s$ and thus $z_{k} ( s^k_{i_k,j_k + 1} ) \weakto z (s)$. We conclude that $z(s) = \bar{z} (r)$. In a similar way $\bar{t} (r) = t(s)$ and $\bar{u} ( r) = u (s)$. 

It remains to show that~$s$ is a continuity point for~$(t,u,z)$. To this aim, let us set, up to subsequence, $s_{\delta}\coloneq \lim_{k} s^{k}_{i'_{k}}$, where the indexes~$i'_{k}$ have been defined in~\eqref{e.stolemma}. Hence, applying Lemma~\ref{l.s<r} we deduce that
\begin{displaymath}
|s - s_{\delta}|  = \lim_{k\to\infty}\, |s^{k}_{i_{k}, j_{k}+1} - s^{k}_{i'_{k}}| \leq \lim_{k\to\infty}\, |s^{k}_{i_{k}} - s^{k}_{i'_{k}}| \leq C \lim_{k\to\infty}\, |r^{k}_{i_{k}} - r^{k}_{i'_{k}}| \leq C\delta\,.
\end{displaymath} 
By definition of~$s_{\delta}$ we have that $t(s_{\delta})= \lim_{k\to\infty} t_{k}(s^{k}_{i'_{k}}) = \lim_{k\to\infty}  \bar{t}_{k}(r^{k}_{i'_{k}})$. By~\eqref{e.stolemma} we get that $| \bar{t}_{k}(r^{k}_{i'_{k}}) - \bar{t}_{k}(r+\delta)|\leq \tau_{k}$, from which we deduce that $t(s_{\delta}) = \bar{t} ( r + \delta) > \bar{t}(r) = t(s)$. This implies that~$s$ is of continuity for~$(t,u,z)$.

The viceversa can be shown in a similar way. \qed \end{proof}

On the contrary, in discontinuity points the evolution $(t,u,z)$ and $(\bar{t}, \bar{u}, \bar{z})$ interpolate the same configurations but with different paths. To better understand, let us consider an interval of the form $(r^k_{i_{k},j_{k}+\frac12}, r^k_{i_{k},j_{k}+1})$ such that $r^k_{i_{k},j_{k}+\frac12} \to r^-$, $r^k_{i_{k},j_{k}+1} \to r^+$ with $r^- < r^+$.
As a consequence both $\bar{t}$ and $\bar{u}$ are constant in $(r^-, r^+)$ and thus every $r \in (r^-, r^+)$ is not a continuity point. In this case, we have 
$$
	\bar{z}_{k} ( r^k_{i_{k},j_{k}+\frac12} ) \weakto \bar{z} ( r^-) 
		\quad \text{ and } \quad  
	\bar{z}_{k} ( r^k_{i_{k},j_{k}+1} ) \weakto \bar{z} ( r^+) \,.
$$
By the non-degeneracy property of $(\bar{t}, \bar{u}, \bar{z})$, we deduce that $\bar{z} ( r^-) \neq \bar{z} ( r^+)$. Moreover, up to subsequence we have
$$
	z_{k} ( s^k_{i_{k},j_{k}+\frac12} ) \weakto z ( s^-) 
		\quad \text{ and } \quad  
	z_{k} ( s^k_{i_{k},j_{k}+1} ) \weakto z ( s^+) \,.
$$
Since $z\in W^{1,\infty}([0,S]; L^{2}(\Om))$, we get that $s^- \neq s^+$. Recalling that $t$ and $u$ are constant in $(s^-,s^+)$, the energy balance of Theorem \ref{t.1} reads
$$
	\mathcal{J} ( t(s^-) , u(s^-) , z(s^+) ) = \mathcal{J} ( t(s^-) , u(s^-) , z(s^-) ) - \int_{s^-}^{s^+} \| z' (\sigma) \|_{L^2} \, | \partial^-_v \mathcal{J}| ( t(s^-) , u(s^-) , z(\sigma) )  \, \di \sigma .
$$
Thus $z$ is a (normalized) unilateral gradient flow in $L^2$, with initial datum $z(s^-)$. On the contrary, $\bar{z}$ is the affine interpolation of $\bar{z} ( r^-)$ and $\bar{z} ( r^+)$. Thus, in general, $\bar{z}$ and $z$ do not coincide in the corresponding intervals $(r^-,r^+)$ and $(s^-,s^+)$ even if they coincide in the extrema.

\section{On the alternate behavior in discontiuity points \label{AppB}} 

Let us consider the set $U$ defined in \eqref{e.U}
and denote by $\mathbf{1}_{U}$ its characteristic function. From \eqref{e.120} we know that
\begin{align*}
\F( t(s), u (s),  z (s))  \leq  & \ \F( 0, u_{0}, z_{0}) - \int_{0}^{s} \!\! |\partial_{u} \F | ( t (\sigma), u ( \sigma ), z ( \sigma ))\, \| u' (\sigma) \|_{H^1} \,\di  \sigma \\
& -  \int_{0}^{s} \!\!  | \partial_{z}^{-} \F | ( t(\sigma), u ( \sigma ), z( \sigma )) \, \mathbf{1}_{U}(\sigma) \, \di \sigma - \int_{0}^{s} \!\! \liminf_{k\to\infty} | \partial_{z}^{-} \F | ( \underline{t}_{k}(\sigma), \underline{u}_{k}( \sigma ), z_{k}( \sigma )) \, \mathbf{1}_{U^{c}} (\sigma) \, \di \sigma .
\end{align*}
Note that $z' (\sigma) = 0$ for every $\sigma \in U^c$ (see Lemma \ref{l.Uc}), thus, being $\| z' \|_{L^2} \le 1$ a.e.~in $[0,S]$, we have $\| z' \|_{L^2} \le \mathbf{1}_{U}$ a.e.~in $[0,S]$; as a consequence the above estimate together with \eqref{e.eneq} yield
\begin{align*}
\F( t(s), u (s),  z (s)) 
 \leq  & \ \F( 0, u_{0}, z_{0}) - \int_{0}^{s} \!\! |\partial_{u} \F | ( t (\sigma), u ( \sigma ), z ( \sigma ))\, \| u' (\sigma) \|_{H^1} \,\di  \sigma \\
& -  \int_{0}^{s} \!\!  | \partial_{z}^{-} \F | ( t(\sigma), u ( \sigma ), z( \sigma )) \, \mathbf{1}_{U}(\sigma) \, \di \sigma  +\int_{0}^{s} \mathcal{P}(t(\sigma), u(\sigma), z(\sigma)) \, t'(\sigma) \, \di \sigma \,.  \\
 \leq & \ \F( 0, u_{0}, z_{0}) - \int_{0}^{s} \!\! |\partial_{u} \F | ( t (\sigma), u ( \sigma ), z ( \sigma ))\, \| u' (\sigma) \|_{H^1} \,\di  \sigma \\
& -  \int_{0}^{s} \!\!  | \partial_{z}^{-} \F | ( t(\sigma), u ( \sigma ), z( \sigma )) \, \| z' (\sigma) \|_{L^2} \, \di \sigma  +\int_{0}^{s} \mathcal{P}(t(\sigma), u(\sigma), z(\sigma)) \, t'(\sigma) \, \di \sigma   \\
= & \ \F ( t(s), u (s),  z (s)) .
\end{align*}
Therefore, all inequalities turn into equalitites and thus 
$$
0 \le  \int_{0}^{s} \!\!  | \partial_{z}^{-} \F | ( t(\sigma), u ( \sigma ), z( \sigma )) \, ( \mathbf{1}_{U}(\sigma) - \| z' (\sigma) \|_{L^2} ) \, \di \sigma = 0 ;
$$
hence $| \partial_{z}^{-} \F | ( t(\sigma), u ( \sigma ), z( \sigma )) \, ( \mathbf{1}_{U}(\sigma) - \| z' (\sigma) \|_{L^2} ) =0$ a.e.~in $[0,S]$. 

Therefore, if $| \partial_{z}^{-} \F | ( t(\sigma), u ( \sigma ), z( \sigma )) \neq 0$ for $\sigma \in (\sigma_1, \sigma_2) \subset U$ then $\| z' (\sigma) \|_{L^2} = 1$ a.e.~in $(\sigma_1, \sigma_2)$ and then both $t'(\sigma)=0$ and $u'(\sigma)=0$ a.e.~in $(\sigma_1, \sigma_2)$, by (a) in Theorem \ref{t.1}. 
This means that in the discontinuity interval $(\sigma_1, \sigma_2)$ only $z$ changes, following a normalized, unilateral $L^2$ gradient flow. In view of this observation, excluding the cases in which $| \partial_{z}^{-} \F | ( t(\sigma), u ( \sigma ), z( \sigma )) = 0$, we expect that in the presence of time-discontinuties the limit evolution is still alternate, and thus it is not simultaneous in $u$ and $z$. More precisely, consider a discontinuity at time $t$, with a transition from $(u^-,z^-)$ to $(u^+,z^+)$, parametrized in the interval $(\sigma^-, \sigma^+)$.  Being~$t$ the limit of continuity points, the left limit $(u^-,z^-)$ is an equilibrium configuration\ at time~$t$. We expect that the parametrizations of~$u$ and~$z$, in a right neighborhood of~$\sigma^-$, provide an alternate interpolation of sequences $z^-_m \nearrow z^-$, with $z^-_m \neq z^-$, and $u^-_m \to u^-$ such that 
$$  \quad \begin{cases}
	u^-_{m-1} \in \mbox{argmin} \, \big\{  \F ( t \,, u  , z^-_{m} ) \big\},
	\\
	z^-_{m-1}  \in \mbox{argmin} \, \big\{  \F ( t \, , u^-_{m-1}  \, , z \, ) : z \le z^-_{m}  \big\}.
\end{cases} $$
The non-degeneracy condition $z^-_m \neq z^-$ is due to the fact that~$(u^-,z^-)$ is an equilibrium configuration, and thus a separate minimizer of the energy $\F ( t , \cdot, \cdot)$. Indeed, if $z^-_m = z^-$, for some index $m \in \mathbb{N}$, then, by uniqueness of the minimizer, 
$u^-_{m-1} = u^-$ and then $z^-_{m-1} = z^-$. By induction and by monotonicity, $z^-_m=z^-$ for every index $m \in \mathbb{N}$ and then $u^-_m = u^-$ for every $m \in \mathbb{N}$; thus, there would be no transition between $(u^-,z^-)$ to $(u^+,z^+)$. In a similar way, we expect sequences $z^+_m \searrow z^+$ and $u^+_m \to u^+$ in a left neighborhood $\sigma^+$ such that 
$$  \quad \begin{cases}
	u^+_{m+1} \in \mbox{argmin} \, \big\{  \F ( t \,, u  , z^+_{m} ) \big\},
	\\
	z^+_{m+1}  \in \mbox{argmin} \, \big\{  \F ( t \, , u^+_{m+1}  \, , z \, ) : z \le z^+_{m}  \big\}.
\end{cases} $$
However, in this case we cannot exclude that $z^+_m = z^+$ for some index $m \in \mathbb{N}$. We remark that this qualitative behavior is confirmed by numerical computations.

\section*{Acknowledgments}
The work of S.A. was supported by the SFB TRR109. M.N. is member of the Gruppo Nazionale per l'Analisi Matematica, la Probabilit\`a e le loro Applicazioni (GNAMPA) of the Instituto Nazionale di Alta Matematica (INdAM).

\bibliographystyle{siam}
\bibliography{bibliography} 

\begin{thebibliography}{10}

\bibitem{Almi2017}
{\sc S.~Almi and S.~Belz}, {\em Consistent finite-dimensional approximation of
  phase-field models of fracture}, Annali di Matematica Pura ed Applicata,
  (2018).

\bibitem{A-B-N17}
{\sc S.~Almi, S.~Belz, and M.~Negri}, {\em Convergence of discrete and
  continuous unilateral flows for ambrosio-tortorelli energies and application
  to mechanics}, ESAIM: M2AN,  (2018).

\bibitem{MR3304294}
{\sc M.~Ambati, T.~Gerasimov, and L.~De~Lorenzis}, {\em A review on phase-field
  models of brittle fracture and a new fast hybrid formulation}, Comput. Mech.,
  55 (2015), pp.~383--405.

\bibitem{AGS}
{\sc L.~Ambrosio, N.~Gigli, and G.~Savar{\'e}}, {\em {Gradient flows in metric
  spaces and in the space of probability measures}}, {Lectures in Mathematics
  ETH Z{\"u}rich}, Birkh{\"a}user Verlag, Basel, 2005.

\bibitem{MR1075076}
{\sc L.~Ambrosio and V.~M. Tortorelli}, {\em Approximation of functionals
  depending on jumps by elliptic functionals via {$\Gamma$}-convergence}, Comm.
  Pure Appl. Math., 43 (1990), pp.~999--1036.

\bibitem{A-M-M09}
{\sc H.~Amor, J.-J. Marigo, and C.~Maurini}, {\em Regularized formulation of
  the variational brittle fracture with unilateral contact: numerical
  experiments}, J. Mech. Phys. Solids, 57 (2009), pp.~1209--1229.

\bibitem{MR3249813}
{\sc J.-F. Babadjian and V.~Millot}, {\em Unilateral gradient flow of the
  {A}mbrosio-{T}ortorelli functional by minimizing movements}, Ann. Inst. H.
  Poincar\'{e} Anal. Non Lin\'{e}aire, 31 (2014), pp.~779--822.

\bibitem{Balder_RCMP85}
{\sc E.~J. Balder}, {\em {An extension of {P}rohorov's theorem for transition
  probabilities with applications to infinite-dimensional lower closure
  problems}}, Rend. Circ. Mat. Palermo (2), 34 (1985), pp.~427--447 (1986).

\bibitem{MR2341850}
{\sc B.~Bourdin}, {\em Numerical implementation of the variational formulation
  for quasi-static brittle fracture}, Interfaces Free Bound., 9 (2007),
  pp.~411--430.

\bibitem{MR1745759}
{\sc B.~Bourdin, G.~A. Francfort, and J.-J. Marigo}, {\em Numerical experiments
  in revisited brittle fracture}, J. Mech. Phys. Solids, 48 (2000),
  pp.~797--826.

\bibitem{Brezis_73}
{\sc H.~Br\'ezis}, {\em Op\'erateurs maximaux monotones et semi-groupes de
  contractions dans les espaces de {H}ilbert}, North-Holland Publishing Co.,
  Amsterdam-London; American Elsevier Publishing Co., Inc., New York, 1973.
\newblock North-Holland Mathematics Studies, No. 5. Notas de Matem\'atica (50).

\bibitem{MR2669398}
{\sc S.~Burke, C.~Ortner, and E.~S\"uli}, {\em An adaptive finite element
  approximation of a variational model of brittle fracture}, SIAM J. Numer.
  Anal., 48 (2010), pp.~980--1012.

\bibitem{MR2074682}
{\sc A.~Chambolle}, {\em An approximation result for special functions with
  bounded deformation}, J. Math. Pures Appl. (9), 83 (2004), pp.~929--954.

\bibitem{MR3780140}
{\sc A.~Chambolle, S.~Conti, and G.~A. Francfort}, {\em Approximation of a
  brittle fracture energy with a constraint of non-interpenetration}, Arch.
  Ration. Mech. Anal., 228 (2018), pp.~867--889.

\bibitem{ChambolleCrismale_D}
{\sc A.~Chambolle and V.~Crismale}, {\em A density result in ${GSBD}^p$ with
  applications to the approximation of brittle fracture energies}, Arch.
  Ration. Mech. Anal.,  (2018).

\bibitem{ComiPerego_IJSS01}
{\sc C.~Comi and U.~Perego}, {\em Fracture energy based bi-dissipative damage
  model for concrete}, Int. J. Solids Structures, 38 (2001), pp.~6427 -- 6454.

\bibitem{MR3082250}
{\sc G.~Dal~Maso}, {\em Generalised functions of bounded deformation}, J. Eur.
  Math. Soc. (JEMS), 15 (2013), pp.~1943--1997.

\bibitem{MR2186036}
{\sc G.~Dal~Maso, G.~A. Francfort, and R.~Toader}, {\em Quasistatic crack
  growth in nonlinear elasticity}, Arch. Ration. Mech. Anal., 176 (2005),
  pp.~165--225.

\bibitem{Fonseca2007}
{\sc I.~Fonseca and G.~Leoni}, {\em Modern methods in the calculus of
  variations: {$L^p$} spaces}, Springer Monographs in Mathematics, Springer,
  New York, 2007.

\bibitem{MR2106765}
{\sc A.~Giacomini}, {\em Ambrosio-{T}ortorelli approximation of quasi-static
  evolution of brittle fractures}, Calc. Var. Partial Differential Equations,
  22 (2005), pp.~129--172.

\bibitem{HerzogMeyerWachsmuth_JMAA11}
{\sc R.~Herzog, C.~Meyer, and G.~Wachsmuth}, {\em Integrability of displacement
  and stresses in linear and nonlinear elasticity with mixed boundary
  conditions}, J. Math. Anal. Appl., 382 (2011), pp.~802--813.

\bibitem{MR3247391}
{\sc F.~Iurlano}, {\em A density result for {GSBD} and its application to the
  approximation of brittle fracture energies}, Calc. Var. Partial Differential
  Equations, 51 (2014), pp.~315--342.

\bibitem{PhysRevLett.87.045501}
{\sc A.~Karma, D.~A. Kessler, and H.~Levine}, {\em Phase-field model of mode
  {III} dynamic fracture}, Phys. Rev. Lett., 87 (2001), p.~045501.

\bibitem{KneesNegri_M3AS17}
{\sc D.~Knees and M.~Negri}, {\em Convergence of alternate minimization schemes
  for phase field fracture and damage}, Math. Models Methods Appl. Sci., 27
  (2017), pp.~1743--1794.

\bibitem{MR3021776}
{\sc D.~Knees, R.~Rossi, and C.~Zanini}, {\em A vanishing viscosity approach to
  a rate-independent damage model}, Math. Models Methods Appl. Sci., 23 (2013),
  pp.~565--616.

\bibitem{MR3332887}
{\sc D.~Knees, R.~Rossi, and C.~Zanini}, {\em A quasilinear differential
  inclusion for viscous and rate-independent damage systems in non-smooth
  domains}, Nonlinear Anal. Real World Appl., 24 (2015), pp.~126--162.

\bibitem{MR3893258}
{\sc D.~Knees, R.~Rossi, and C.~Zanini}, {\em Balanced viscosity solutions to a
  rate-independent system for damage}, European J. Appl. Math., 30 (2019),
  pp.~117--175.

\bibitem{Lojasiewicz_AIF93}
{\sc S.~{\L}ojasiewicz}, {\em Sur la g\'eom\'etrie semi- et sous-analytique},
  Ann. Inst. Fourier (Grenoble), 43 (1993), pp.~1575--1595.

\bibitem{MR2182832}
{\sc A.~Mielke}, {\em Evolution of rate-independent systems}, in Evolutionary
  equations. {V}ol. {II}, Handb. Differ. Equ., Elsevier/North-Holland,
  Amsterdam, 2005, pp.~461--559.

\bibitem{MR3531671}
{\sc A.~Mielke, R.~Rossi, and G.~Savar\'{e}}, {\em Balanced viscosity ({BV})
  solutions to infinite-dimensional rate-independent systems}, J. Eur. Math.
  Soc. (JEMS), 18 (2016), pp.~2107--2165.

\bibitem{MielkeRoubicek}
{\sc A.~Mielke and T.~Roub{\'i}\v{c}ek}, {\em Rate-Independent Systems: Theory
  and Application.}, vol.~193 of Applied Mathematical Sciences, Springer, New
  York, 2015.

\bibitem{Negri_ACV}
{\sc M.~Negri}, {\em A unilateral {$L^2$}-gradient flow and its quasi-static
  limit in phase-field fracture by alternate minimization}, Adv. Calc. Var., 12
  (2019), pp.~1--29.

\bibitem{NegriKimura}
{\sc M.~Negri and M.~Kimura}, {\em Weak solution of unilateral gradient flow},
  (to appear).

\bibitem{Thomas13}
{\sc M.~Thomas}, {\em {Quasistatic damage evolution with spatial
  {BV}-regularization}}, Discrete Contin. Dyn. Syst. Ser. S, 6 (2013),
  pp.~235--255.

\bibitem{Wu_JMPS17}
{\sc J.-Y. Wu}, {\em A unified phase-field theory for the mechanics of damage
  and quasi-brittle failure}, Journal of the Mechanics and Physics of Solids,
  103 (2017), pp.~72 -- 99.

\end{thebibliography}

\end{document}